\documentclass[11pt]{article}
\usepackage{amsmath}[1996/11/01]
\usepackage{amssymb,amsthm,amsxtra}
\usepackage{amscd}
\usepackage[margin=1in]{geometry}
\usepackage{float}
\usepackage{caption}

\def\[#1\]{\begin{equation}#1\end{equation}}
\makeatletter
\def\beq{%
   \relax\ifmmode
      \@badmath
   \else
      \ifvmode
         \nointerlineskip
         \makebox[.6\linewidth]%
      \fi
      $$
   \fi
}
\def\eeq{%
   \relax\ifmmode
      \ifinner
         \@badmath
      \else
         $$
      \fi
   \else
      \@badmath
   \fi
   \ignorespaces
}

\def\enddisplaymath{\eeq\global\@ignoretrue}
\makeatother

\newtheorem{thm}{Theorem}
\newtheorem{cor}[thm]{Corollary}
\newtheorem{lem}[thm]{Lemma}
\newtheorem{prop}[thm]{Proposition}

\theoremstyle{remark}
\newtheorem*{rem}{Remark}
\newtheorem{rems}{Remark}[thm]

\newtheorem{eg}{Example}
\theoremstyle{definition}

\numberwithin{equation}{section}
\numberwithin{thm}{section}
\numberwithin{eg}{section}
\numberwithin{defn}{section}

\newcommand{\C}{\mathbb C}
\newcommand{\G}{\mathbb G}
\newcommand{\Z}{\mathbb Z}
\newcommand{\A}{\mathbb A}
\newcommand{\Q}{\mathbb Q}
\newcommand{\R}{\mathbb R}
\newcommand{\F}{\mathbb F}
\renewcommand{\H}{\mathbb H}
\renewcommand{\P}{\mathbb P}

\DeclareMathOperator{\AGL}{AGL}
\DeclareMathOperator{\Aut}{Aut}
\DeclareMathOperator{\Lin}{Lin}
\DeclareMathOperator{\im}{im}
\DeclareMathOperator{\Pic}{Pic}
\DeclareMathOperator{\NS}{NS}
\DeclareMathOperator{\Spec}{Spec}
\DeclareMathOperator{\End}{End}
\DeclareMathOperator{\Mat}{Mat}
\DeclareMathOperator{\Res}{Res}
\DeclareMathOperator{\Ind}{Ind}
\DeclareMathOperator{\GL}{GL}
\DeclareMathOperator{\SL}{SL}
\DeclareMathOperator{\Hom}{Hom}
\DeclareMathOperator{\lcm}{lcm}
\DeclareMathOperator{\ch}{char}
\DeclareMathOperator{\Tr}{Tr}
\DeclareMathOperator{\ord}{ord}
\DeclareMathOperator{\Sym}{Sym}
\DeclareMathOperator{\Alt}{Alt}
\DeclareMathOperator{\Proj}{Proj}
\DeclareMathOperator{\PGL}{PGL}
\DeclareMathOperator{\GO}{O}
\DeclareMathOperator{\PSL}{PSL}
\DeclareMathOperator{\PGU}{PU}
\DeclareMathOperator{\SU}{SU}
\DeclareMathOperator{\Dih}{Dih}
\DeclareMathOperator{\Sp}{Sp}
\DeclareMathOperator{\SO}{SO}
\DeclareMathOperator{\GU}{U}
\DeclareMathOperator{\Gal}{Gal}
\DeclareMathOperator{\Lie}{Lie}
\DeclareMathOperator{\red}{red}

\newcommand{\sO}{\mathcal O}
\newcommand{\surj}{\twoheadrightarrow}
\newcommand{\normal}{\trianglelefteq}
\newcommand{\ratto}{\dashrightarrow}

\let\div\relax
\DeclareMathOperator{\div}{div}

\begin{document}

\title{Quotients of abelian varieties by reflection groups}
  \author{
Eric M. Rains\\Department of Mathematics, California
  Institute of Technology}

\date{January 20, 2024}
\maketitle

\begin{abstract}
We prove (by a case-by-case analysis) a conjecture of Bernstein/Schwarzman
to the effect that quotients of abelian varieties by suitable actions
of (complex) reflection groups are weighted projective spaces, and
show that this remains true after reduction to finite characteristic
(including characteristics dividing the order of the group!).  We
also show that an analogous statement holds (with five explicitly
enumerated exceptions) for actions of {\em quaternionic} reflection
groups on supersingular abelian varieties.

\end{abstract}

\tableofcontents

\section{Introduction}

One of the more striking results of invariant theory is the fact
\cite{ShephardGC/ToddJA:1954} that if the finite group $G\subset \GL_n(\C)$
is generated by ``complex reflections'' (i.e., elements such that
$\dim\ker(g-1)=1$), then its ring of polynomial invariants is a free
polynomial ring.  This was originally shown in
\cite{ShephardGC/ToddJA:1954} by a case-by-case analysis, with Chevalley
(unpublished) giving a classification-free argument that extends to
``good'' characteristic, i.e., reflection groups over fields where $|G|\ne
0$.

In \cite{LooijengaE:1976}, quotients of the form $\Lambda_W\otimes E/W$ for
$E$ a complex elliptic curve, $W$ a finite irreducible Weyl group, and
$\Lambda_W$ a suitable lattice were considered, and it was shown that the
Hilbert series of the invariant ring was that of a free polynomial algebra.
This, of course, suggests that the quotient is in fact a weighted
projective space, which was shown in \cite{KacVG/PetersonDH:1984} (for
classical groups) and \cite{BernsteinJ/SchwarzmanO:2006} (uniformly, with
one classical exception).  Together with the Shephard-Todd-Chevalley
result, this suggests that one should consider actions of {\em complex}
reflection groups on abelian varieties.  This was also considered in
\cite{BernsteinJ/SchwarzmanO:2006}, which gave necessary conditions for the
quotient to be a weighted projective space and conjectured that those
conditions were sufficient.  (They also gave analogous conditions for the
invariant ring associated to an equivariant line bundle to be a free
polynomial ring.)

Returning to the real (i.e., Weyl group) case, the author in \cite{elldaha}
needed to understand certain twisted versions of Reynolds operators for
such actions in order to derive flatness for ``spherical elliptic DAHAs''.
In good characteristic, this is straightforward (the usual Reynolds
operator suffices), but it turned out that slightly more complicated
operators worked in bad characteristic as well.  Just as the usual Reynolds
operator proves that invariant rings are flat in good characteristic (i.e.,
that the Hilbert series is independent of the characteristic), these
twisted Reynolds operators show that this remains true for quotients
$\Lambda_W\otimes E/W$ even in bad characteristic!  In particular, it was
shown there that in arbitrary characteristic, $\Lambda_W\otimes E/W$ has
the same Hilbert series as a weighted projective space, suggesting as a
natural conjecture that it {\em is} a weighted projective space in any
characteristic.  (Note that this is {\em not} a logical consequence of
flatness of the Hilbert series; the (homogeneous) quotient
$\Z[w,x,y,z]/(2w-y^2+xz)$ is a polynomial ring over $\Z[1/2]$ and is flat
at $2$, but is not a polynomial ring mod 2.  We will see this phenomenon
arise naturally below.)

Combining this with the conjecture of \cite{BernsteinJ/SchwarzmanO:2006}
leads one to consider actions of complex reflection groups in arbitrary
characteristic.  Although one could simply consider reductions to primes
where the abelian variety has good reduction, there is a somewhat more
intrinsic approach.  Indeed, the notion of a ``reflection'' can be made
purely geometric (i.e., a reflection of an abelian variety is an
automorphism that fixes a hypersurface pointwise), so that one can ask
about reflection group actions more generally, which for supersingular
abelian varieties adds an additional possibility that the group may be a
{\em quaternionic} reflection group.

Note that already in the real characteristic 0 case, it is not true that an
{\em arbitrary} action by a reflection group has weighted projective space
quotient, as there are topological obstructions to this that are not
preserved by equivariant isogeny.  (There is also the obvious caveat that a
reducible reflection group acting on a product can at best have quotient a
{\em product} of weighted projective spaces.)  It turns out that this
topological condition (the necessary condition of
\cite{BernsteinJ/SchwarzmanO:2006}) can be rephrased in algebro-geometric
terms, giving a notion of a {\em crystallographic} reflection group in
arbitrary characteristic.

With that in mind, the main result of this paper (Theorem
\ref{thm:main_theorem} below) is that the quotient of an abelian variety by
a crystallographic reflection group is essentially always a weighted
projective space.  More precisely, we explicitly enumerate the five
(quaternionic) cases where this claim fails, and prove in all other cases
that it holds.  (In particular, since the counterexamples are all
quaternionic, this not only proves the main conjecture of
\cite{BernsteinJ/SchwarzmanO:2006} but extends it to arbitrary
characteristic.)  We also prove a more refined version (again following
\cite{BernsteinJ/SchwarzmanO:2006}) showing that for any ``strongly''
crystallographic equivariant line bundle (apart from the corresponding five
counterexamples), the corresponding ring of invariants is a polynomial
ring.

The facts that the results remain true in bad characteristic and the
presence of counterexamples are significant obstructions to more abstract
arguments \`a la Chevalley, and thus we make significant use of the
classification of such actions.  For the real characteristic 0 case, such
actions were classified in \cite{SaitoK:1985}, and this was extended to the
general characteristic 0 case in \cite{PopovVL:1982}.  For the quaternionic
case, however, the classification itself is new (though relying heavily on
the classification of {\em linear} quaternionic reflection groups in
\cite{CohenAM:1980}), and even in the real and complex cases it is a priori
possible for there to be sporadic actions.  (In fact, there is {\em one}
sporadic action: $ST_5$ has a crystallographic action on a supersingular
abelian surface in characteristic 2 that does not lift to characteristic
0.)  As a result, even in the real and complex cases, we are forced to
rederive the classifications in the process of extending them to finite
characteristic.

We then use a number of different methods to determine the invariant ring
and thus prove that it is polynomial.  For the infinite families of
imprimitive groups, there is an equivariant isogeny of the form $E^n\to A$
which one can use to explicitly compute the invariants (as the quotient of
$E^n$ by a suitable imprimitive group {\em scheme}).  (In particular, the
group $G$ is always normal in a wreath product with quotient
$\Sym^n(\P^1)\cong \P^n$, letting us describe the quotient by $G$ as an
abelian cover of $\P^n$.)  There are also five (or six, depending on one's
conventions) sporadic imprimitive cases, but those all reduce to an action
on $E^2$ and can thus be dealt with essentially by hand.

The other main {\em conceptual} method is to represent the desired quotient
as a moduli space that one can show more directly is a weighted projective
space.  For instance, in the $A_n$ case, the quotient is naturally
identified with a linear system in $E$, and thus with $\P^n$; this also
enables one to compute quotients in cases (e.g., $G_2$) for which $A_n$ is
a normal subgroup.  Less obviously, the well-known combinatorics of the 27
lines on a cubic surface gives rise to an interpretation of $E^6/W(E_6)$ as
the moduli space of affine cubic surfaces with a copy of $E$ at infinity.
The moduli space can be computed more directly as classifying the ways to
add lower-degree terms to the (cubic) equation of $E$, from which it is
easy to see that it is a weighted projective space.  Something similar
holds for $E_7$ and $E_8$, and imposing suitable additional structures on
the surface allows one to deduce analogous statements for a number of
additional cases (including some which are complex or quaternionic).

For the cases (all of which are complex or quaternionic) not covered by the
above constructions, we use a more computational approach based on
explicitly computing invariants.  (We should note that although this
approach is indeed computational, none of the computations are particularly
long: the full verification takes only a handful of CPU minutes!)  The main
difficulty here is even writing down $A$; although $A$ is isomorphic to
$E^n$ as an abelian variety, it is not isomorphic as a {\em polarized}
abelian variety, making it difficult even to write down sections of line
bundles on $A$.  We can get around this in two ways.  First, if $G$ has a
normal subgroup $N$ for which varieties of the form $A/N$ are
well-understood, then we can hope to find the unique variety of this form
with an action of $G/N$ and use that to compute invariants.  In particular,
in small cases, we understand enough about the structure of {\em Kummer}
varieties (e.g., the fact that a Kummer surface is a quartic hypersurface)
to pull this off when $N=[\pm 1]$, while in a couple of other cases, $A/N$
is nice enough (a weighted projective space or a hypersurface in a weighted
projective space) to let us find the desired action.

For cases without a suitable normal subgroup, we can instead work with a
non-normal subgroup.  In particular, if $H\subset G$ is a reflection
subgroup and $B\to A$ is an $H$-equivariant isogeny such that $H$ acts
crystallographically on $B$, then there is an induced morphism $B/H\to A/G$
and we inductively know that $B/H$ is a weighted projective space (and in
nice cases can evaluate the generators).  As a result, we can hope to
represent the invariant ring of $A/G$ as a subring of the invariant ring of
$B/H$, which greatly simplifies the resulting expressions (since $B/H$
tends to have much smaller Hilbert function than $A$ itself).  Moreover, it
turns out to be relatively straightforward to generate (random) equations
satisfied by the pullbacks to $B/H$ of the invariants of $A/G$.  In any
given degree, we only need finitely many such equations to cut out the true
space of pullbacks, and can detect that we have done so given a suitable
lower bound on the dimension.  (Although the full Hilbert series is
difficult to compute directly even in characteristic 0, it turns out that
there are infinitely many degrees for which we can either compute the
precise dimension (for complex groups) or a lower bound (for quaternionic
groups) that turns out in retrospect to be tight.)

In either case, once one has found enough invariants of low degree by
explicit calculation, it is straightforward to test whether they generate:
if the polynomial ring in generators of those degrees has Hilbert function
with the correct growth rate, then they generate iff they have no common
zero on $A$.  This is a Gr\"{o}bner basis calculation, and is thus
straightforward to do in any given characteristic, with some additional
work needed to identify (and then check) the finitely many primes where the
characteristic 0 calculation might fail.

There are a number of open problems left below.  The biggest, of course, is
to come up with more conceptual arguments for the cases dealt with in
Sections 6 and 7 below (e.g., can one interpret $E_0^6/ST_{34}$ as a
natural moduli space?), or even better to give a classification-free proof.
The primary obstruction to the latter is two-fold: first, that we do not
understand {\em why} the counterexamples fail to have polynomial
invariants, and in particular do not see any way to exclude them without
excluding cases that {\em do} work.  The second is that for the stronger
version of the conjecture regarding equivariant line bundles with
polynomial invariants, the condition is not only not quite sufficient, it
is also not quite necessary: there are equivariant line bundles in finite
characteristic for which the invariants are polynomial but the natural
analogue of the condition of \cite{BernsteinJ/SchwarzmanO:2006} fails.
Also, in the real case, one can refine the statement to include a natural
description of the degrees in terms of the standard labels of an associated
affine Dynkin diagram; is there anything analogous for the
complex/quaternionic cases?  Finally, another less significant question is
whether for complex groups there is a direct way to prove polynomiality in
good characteristic from polynomiality in characteristic 0, thus avoiding
the step in several of our computations where we need to compute the
characteristic 0 invariants in order to set up and perform the relevant
Gr\"obner basis calculation.

The plan of the paper is as follows.  Section 2 discusses generalities
about reflection group actions on abelian varieties, with Subsection 2.1 in
particular discussing the algebro-geometric analogue of the necessary
conditions of \cite{BernsteinJ/SchwarzmanO:2006} and stating the main
theorems.  Subsection 2.2 gives some lemmas for checking when a set of
$n+1$ elements generates the invariant ring, and Subsection 2.3 discusses
how one extends the classification of {\em linear} reflection groups to a
classification of actions on abelian varieties (including actions not
fixing the identity).  Finally, Subsection 2.4 discusses when a quotient of
an abelian variety by a group has Cohen-Macaulay invariant ring (which is
nontrivial even in characteristic 0, since the homogeneous coordinate ring
of $A$ itself is almost never Cohen-Macualay!), and in particular shows
that this holds for the derived subgroup of a reflection group.  (We also
discuss conditions under which the quotient scheme is Calabi-Yau.)

Section 3 covers imprimitive groups (as well as actions on elliptic
curves), with the first three subsections classifying the crystallographic
actions and the last three computing invariants.  Section 4 gives the
classification for {\em primitive} groups and gives tables of the different
cases and the degrees of their invariants.  (Subsection 4.4 gives the
analogous tables for the imprimitive cases.)

Section 5 gives the arguments via representations as moduli spaces, with
Subsection 5.1 discussing the linear system (i.e., $A_n$) case, Subsection
5.2 discussing the $E_6$, $E_7$, $E_8$ cases (via Del Pezzo surfaces),
Subsection 5.3 discussing the fact that fixed subschemes of suitable
automorphisms are again weighted projective spaces, and the remaining
subsections applying this to give further cases of the main theorem (for
$F_4$ and certain complex and quaternionic cases).  (These constructions
were, in fact, the main impetus for the author to embark on this project;
related moduli spaces appeared in \cite{fildefs}.  Although that paper only
considered characteristic 0 and worked in the opposite direction, it was
clear that one could use such arguments in reverse and in finite
characteristic.)

Section 6 discusses the calculations via quotients by normal subgroups.
Here one of the cases (discussed in Subsection 6.2) is still conceptual (we
dealt with one subcase in Section 5, which turns out to be enough to
identify the action of $G/N$ as one with classically known invariants),
while the others are dealt with via computers.  This section also includes
three of the five counterexamples to the main theorem, coming from
quaternionic reflection groups in odd characteristic that do not have a
global fixed point.  Section 7 then discusses the method of bootstrapping
from reflection subgroups, with Subsection 7.3 in particular proving the
requisite lower bounds on dimensions, and Subsections 7.5 and 7.7
discussing the two remaining counterexamples (certain quaternionic groups
in characteristic 2 that {\em do} fix the identity).

Finally, Section 8 discusses several cases in which quotients by related
groups turn out to be Calabi-Yau.  In characteristic 0, these come from any
group contained in the kernel of the determinant and containing the derived
subgroup (defined in a suitably crystallographic way).  We give an explicit
description of the invariant ring of the quotient by the derived subgroup,
and we show that in that case and the kernel of determinant case, the
quotient remains Calabi-Yau (and has the same complete intersection
structure) in arbitrary finite characteristic.  (Here there is no analogous
statement for quaternionic groups, several of which are actually perfect!)

{\bf Notation}.  The groups (fixing the identity) acting on an elliptic
curve in characteristic 0 are cyclic of order 2,3,4,6.  We denote the
corresponding generators by $[-1]$, $[\zeta_3]$, $[i]$, $[-\zeta_3]$ and
the groups they generate by $[\mu_d]$, $d\in \{2,3,4,6\}$.  (In good
characteristic, the identification with roots of unity can be made via the
action on holomorphic differentials.)  We similarly define an endomorphism
$[\sqrt{-3}]$ for $j=0$.

For abstract groups, since we are dealing with Weyl groups below, we will
use $Z_d$ to denote the cyclic group of order $d$ and $\Dih_d$ to denote
the dihedral group of order $2d$.  We also recall the standard notation
$p^{1+2m}$ for the extraspecial $p$-group of order $p^{2m+1}$ (which is
familiar from an abelian variety perspective as the Heisenberg group
$\mu_p.A[p]$ for $A$ principally polarized); for $p=2$, this splits into
two cases $2_+^{1+2m}$ and $2_-^{1+2m}$ (corresponding to the two parities
of theta characteristics) with an additional group $2^{2+2m}$ as the common
extension with center $Z_4$.  We also follow standard group-theory
conventions by writing $H.G$ for an extension of $G$ by $H$ that is not
necessarily split.  (If we {\em know} that it is split, we will of course
instead write $H\rtimes G$.)

Finally, $\P^{[d_0,d_1,d_2,\dots,d_n]}$ denotes the weighted projective
space with generators of degrees $d_0,\dots,d_n$; by mild abuse of
notation, this may represent either a scheme or a stack, depending on
context, but in either case is the quotient of $\A^{n+1}\setminus 0$ by the
action of $\G_m$ with weights $d_0,\dots,d_n$.

{\bf Acknowledgements}.  The author would like to thank R. Guralnick for
helpful discussions, and in particular for pointers to the literature on
Weil representations (used in Subsection 7.3).

\section{Actions of reflection groups on abelian varieties}

\subsection{Crystallographic reflection groups}

In \cite{BernsteinJ/SchwarzmanO:2006}, two conjectures were formulated (and
proved in the real case) for invariants of ``crystallographic'' complex
reflection groups.  These are subgroups of $\AGL_n(\C)$ that are generated
by ``reflections''--elements that fix a hyperplane pointwise--and have
compact quotient.  The main conjecture is that the quotient of $\C^n$ by
such a group is a weighted projective space (with the second conjecture
specifying which line bundles give polynomial rings of invariants).  The
quotient of $\C^n$ by the subgroup of translations is an abelian variety,
and thus one can rephrase this in terms of certain quotients of abelian
varieties.  We will be proving this conjecture, but since we are also
interested in degenerations to finite characteristic, it will be helpful to
reformulate the problem in more algebraic terms.

Let $X$ be either a normal analytic variety or an integral, normal scheme
over an algebraically closed field $\bar{k}$ (possibly of finite
characteristic).  A {\em reflection} of $X$ is an automorphism $g\in
\Aut(X)$ that fixes some hypersurface pointwise, and a discrete subgroup of
$\Aut(X)$ is a {\em reflection group} if it is generated by reflections.

\begin{rem}
  One could also consider reflections for which the fixed locus is
  everywhere codimension 1, but in the cases we consider, this additional
  condition is automatic.
\end{rem}

\begin{lem}
  Let $g\in \Aut(X)$ be a reflection and let $H\subset \Aut(X)$ be a
  discrete subgroup normalized by but not containing $g$.  Then the image
  of $g$ in $\Aut(X/H)$ is a reflection.
\end{lem}

\begin{proof}
  Indeed, the image of $X^g$ in $X/H$ remains fixed by $g$ and the image of
  a hypersurface is a hypersurface.
\end{proof}

Let $A$ be an abelian variety over an algebraically closed field $\bar{k}$.
There are two natural notions of automorphism group of $A$, depending on
whether we take automorphisms as a {\em variety} or as a {\em group
  scheme}.  Both of these form proper group schemes, which we denote
respectively by $\Aut(A)$ and $\Aut_0(A)$---the latter is the
stabilizer of the identity in the former---and that there is a natural split
short exact sequence
\[
0\to A\to \Aut(A)\to \Aut_0(A)\to 1.
\]
with the map $\Lin:\Aut(A)\to\Aut_0(A)$ given by $\Lin(g)x:= g\cdot
x-g\cdot 0$.

Now, suppose $g\in \Aut_0(A)$ is a reflection.  Then the fixed subscheme is
a subgroup and thus, having codimension 1 at the identity, has codimension
1 everywhere.  Since $\dim(A)=\dim(\ker(g-1))+\dim(\im(g-1))$, we conclude
that $\im(g-1)$ is a 1-dimensional abelian subvariety, so is in particular
an elliptic curve $E_g$, called the {\em root curve} of $g$.  Let $G\subset
\Aut_0(A)$ be a reflection group, and consider the associated collection
$E_r$ of root curves (where we remove duplicates).  Then there is a natural
map $\prod_r E_r\to A$.  More generally, even if $G\subset \Aut(A)$ does
not fix the identity, we may still define root curves using its image
$\Lin(G)$ in $\Aut_0(A)$.  This gives rise to a natural
$\Lin(G)$-equivariant isogeny: $\prod_r E_r$ has an induced action of
$\Lin(G)$ (the obvious analogue of an induced representation) and thus so
does the kernel of $\prod_r E_r\to A$ and its identity component, letting
us define an abelian variety $A^+$ as the quotient of $\prod_r E_r$ by the
reduced identity component of the kernel of $\prod_r E_r\to A$.

\begin{lem}
  An element $g\in \Aut_0(A)$ is a reflection iff $\Lin(g)$ is a reflection
  and $g\cdot 0$ is in the root curve of $\Lin(g)$.  In particular, the
  reflections form a closed subscheme of $\Aut(A)$.
\end{lem}

\begin{proof}
  If $g:x\mapsto g_0x+t_0$ has a fixed point, then we can conjugate by the
  corresponding translation to make $t_0=0$ and thus conclude that $g_0$
  must be a reflection.  Since $\Aut_0(A)$ is discrete, we see that the
  condition that $g_0$ be a reflection is closed, as (since $A$ is proper)
  is the condition that $x\mapsto g_0x+t_0$ have a fixed point, so that the
  reflections naturally form a closed subscheme.  Moreover, an automorphism
  with linear part $g_0$ fixes $t$ iff it has the form $x\mapsto
  g_0x+(1-g_0)t$, and thus the closed subscheme of $g_0x+A$ corresponding
  to elements with fixed points is nothing other than the image of $1-g_0$,
  a.k.a.~the root curve.
\end{proof}

\begin{rem}
  The argument via properness shows that for any family $S\to \Aut(A)$ of
  automorphisms, there is a closed subscheme of $S$ for which the image is
  a reflection, but this construction could a priori fail to respect base
  change.  (Indeed, we will see this happen for a mild variation of this
  question!)  However, the explicit description in terms of root curves
  shows that this does not happen here.
\end{rem}

This lets one talk about reflection group {\em schemes} in the case of
abelian varieties: a group subscheme $G\subset \Aut(A)$ is a reflection
group iff it is generated by the closed subscheme of reflections.  (Of
course, this is only a relevant distinction in finite characteristic when
finite group schemes can fail to be discrete.)  We follow standard group
theory notation in writing ``.'' to denote a possibly non-split extension.

\begin{prop}
  Let $A$ be a complex abelian variety of dimension $n$, and let
  $G\subset\Aut(A)$.  Then the following are equivalent:
  \begin{itemize}
  \item[(a)]
    The subgroup $\pi_1(A).G\subset \AGL_n(\C)$ is a reflection
    group.
  \item[(b)] For any $\Lin(G)$-equivariant isogeny $A'\to A$, the group
    $\ker(A'\to A).G$ acts on $A'$ as a reflection group.
  \item[(c)]
    The group $\ker(A^+\to A).G$ acts on $A^+$ as a reflection group.
  \end{itemize}
\end{prop}

\begin{proof}
  That (b) implies (c) is obvious; that (a) implies (b) follows from the
  fact that reflections induce reflections on quotients.   So it remains
  only to show that (c) implies (a).  Thus assume (c), and note that we may
  assume WLOG that $A^+=A$, since otherwise we may simply replace $G$ with
  $\ker(A^+\to A).G$.  In particular, we have a short exact sequence
  \[
  0\to K\to \prod_r E_r\to A\to 0
  \]
  where $K$ is an abelian subvariety, which in turn induces a short exact
  sequence
  \[
  0\to \pi_1(K)\to \prod_r \pi_1(E_r)\to \pi_1(A)\to 0.
  \]
  (Here $\pi_2(A)=0$ since $A$ has contractible universal cover and
  $\pi_0(K)=\{0\}$ since $K$ is connected, so that the long exact sequence
  of homotopy groups is indeed short exact.)
  
  Consider the lift $\pi_1(A).G$ to $\AGL_n(\C)$.  Every reflection
  $x\mapsto g_0x+t_0$ in $G$ lifts to a reflection in $\pi_1(A).G$ by
  choosing $t_0$ to lie in the copy of $\tilde{E}_{g_0}$ inside
  $\tilde{A}$.  In particular, the reflection subgroup of $\pi_1(A).G$
  surjects on $G$ (since its image contains every reflection).  Moreover,
  since the choice of lift of $t_0$ is not unique, we also see that the
  reflection subgroup contains $\pi_1(E_r)$ for all $r$, and thus contains
  $\pi_1(A)$ by surjectivity.
\end{proof}

The key observation here is that (b) and (c) make sense in arbitrary
characteristic, and the nontrivial implication carries over.

\begin{prop}
  Let $A$ be an abelian variety of dimension $n$ over an algebraically
  closed field, and let $G\subset\Aut(A)$ be a reflection group scheme such
  that the kernel of the morphism from the product of root curves is
  integral (i.e., reduced and connected).  Then for any
  $\Lin(G)$-equivariant isogeny $A'\to A$, the group scheme $\ker(A'\to
  A).G$ acts as a reflection group on $A'$.
\end{prop}

\begin{proof}
  Any $G$-equivariant isogeny $A'\to A$ is a factor of some $A[N]\to A$,
  and thus it suffices to consider the latter.  Taking $N$-torsion gives an
  exact sequence
  \[
  0\to K[N]\to \prod_r E_r[N]\to A[N]\to 0
  \]
  (since $K$ is an abelian variety) and the analogous argument (lifting so that
  $t_0$ still lies in the root curve) shows that the reflection subgroup of
  $A[N].G$ contains $A[N]$ and surjects on $G$.
\end{proof}

\begin{rems}
  We of course also see that $G$ satisfies the conclusion iff it satisfies
  it for {\em any} $A'$, since then it satisfies it for the quotient $A$
  and thus the original hypothesis.
\end{rems}

\begin{rems}
  Note that we include non-\'etale covers of $A$ here, so this is no
  longer a statement about fundamental groups.
\end{rems}

\begin{rems}
  In general, we will refer to an abelian variety $G$-equivariantly
  isogenous to $A$ as a ``lattice'' for $G$.  This terminology is natural
  in characteristic 0, where it corresponds to a $G$-invariant lattice in
  $\pi_1(A)\otimes_\Z\Q$.
\end{rems}

With this in mind, we say that a reflection group scheme $G\subset \Aut(A)$
is ``crystallographic'' if the induced action of $\ker(A^+\to A).G$ on
$A^+$ is a reflection group scheme.  It is then easy to see that $G$ is
crystallographic iff the discrete group $G/(G\cap A)$ is crystallographic
(acting on $A/(G\cap A)$).

Note that if $G$ contains translations, then we can quotient by those
translations before passing to the cover $A^+$, and thus reduce to the case
that $\ker(A^+\to A)$ meets every root curve in the identity.  Since
$\ker(A^+\to A)$ is $G$-equivariant, it follows that if $t\in \ker(A^+\to
A)$, then for any reflection $x\mapsto g_0x+t_0$ of $G$, $(1-g_0)t\in
\ker(A^+\to A)$ is an element of the intersection of the kernel and the
root curve, and thus $\ker(A^+\to A)$ must be fixed by $g_0$ and thus by
the group $\Lin(G)$ generated by those reflections.  

In particular, any crystallographic reflection group $G$ on $A$ without
translations induces a {\em pointed} crystallographic reflection group
$G_0$ on $A^+$, with $G$ a twist of $G_0$ by some class in $H^1(G_0;A)$.
Thus the non-pointed crystallographic reflection groups lying over a given
pointed crystallographic reflection group $(G_0,A)$ are classified by (a)
subgroups $T\subset A^{G_0}$ meeting every root curve trivially and (b) classes
in $H^1(G_0;A/T)$ such that the induced action of $T.G_0$ on $A$ is a
reflection group.  Note that these can be classified by first computing
$H^1(G_0;A/A^{G_0})$ and then determining the reflection subgroups of
$A^{G_0}.G_0$ for each such cocycle.

Our main theorem (the proof of which will be the bulk of the paper) is then
the following.  We say that the reflection group $G$ is {\em irreducible}
if there is an orbit of root curves such that $\prod_{r\in O} E_r\to A$ is
already surjective.

\begin{thm}\label{thm:main_theorem}
  Let $G\subset \Aut(A)$ be an irreducible crystallographic reflection
  group which is not one of five counterexamples in finite characteristic.
  Then $A/G$ is a weighted projective space.
\end{thm}

\begin{rem}
  For the degrees of the generators, see the tables in Section 4 below.
\end{rem}

In characteristic 0, this was conjectured in
\cite{BernsteinJ/SchwarzmanO:2006}.  Note that the tautological line bundle
$\sO_{A/G}(1)$ on the weighted projective space (i.e., the unique ample
generator of the Picard group) pulls back to a line bundle ${\cal L}$ on
$A$, and one has
\[
\bigoplus_{k\ge 0} \Gamma(A;{\cal L}^k)^G
\cong
\bigoplus_{k\ge 0} \Gamma(A/G;\sO_{A/G}(k)),
\]
so that the ring of invariant sections of powers of ${\cal L}$ is a
polynomial ring.  In some cases, however, there are other line bundles on
$A$ that have polynomial invariants, and \cite{BernsteinJ/SchwarzmanO:2006}
made a refined conjecture about which line bundles have this property.
This is equivalent to asking that the corresponding $\G_m$-torsor $T$ have
quotient a punctured affine space, and they observe that this implies for
topological reasons that the group $\pi_1(T).G$ be a (nonlinear!)
reflection group on $\tilde{T}\cong \C^{n+1}$, and conjecture that this
necessary condition is also sufficient.

Again, we will want to translate this to a condition more suitable for
algebraic geometry.  Note that $\pi_1(T)$ is a central extension of
$\Lambda=\pi_1(A)$ by $\Z$ so that $\pi_1(T).G$ is a central extension of
$\Lambda.G$.  If $G$ is a crystallographic reflection group on $A$, then
the reflection subgroup of $\pi_1(T).G$ surjects on $\Lambda.G$ and thus in
particular contains the commutator subgroup $H:=(\pi_1(T),\pi_1(T).G)$ of
$\pi_1(T)$, which is normal in $\pi_1(T).G$.  This contains the derived
subgroup $\pi_1(T)'\cong \Z$ (finite index in the center) and surjects onto
the commutator subgroup $(\Lambda,\Lambda.G)$, which is itself finite index
in $\Lambda$ since $G$ is irreducible.  We thus conclude that $H$ has
finite index in $\pi_1(T)$, with $\pi_1(T)/H$ central in $(\pi_1(T)/H).G$.
It follows that $\tilde{T}/H$ is a finite abelian (and $G$-central) cover
of $T$.  To identify this cover, note first that the isogeny $A^{++}\to A$
corresponding to $(\Lambda,\Lambda.G)$ has kernel
$\Lambda/(\Lambda,G.\Lambda)\cong H_0(G;\Lambda)$, which by the Weil
pairing is dual to $H^0(G;\Pic^0(\hat{A}))$.  In other words, $A^{++}$ is
nothing other than the unramified cover associated to the subgroup
$\Pic^0(\hat{A})^G\subset \Pic^0(\hat{A})$ of line bundles, universal among
all isogenies for which the pullback homomorphism annihilates
$\Pic^0(\hat{A})^G$.  We then see that $\tilde{T}/H$ is a finite cover of
$T\times_A A^{++}\cong \tilde{T}/(H\Z)$ and thus corresponds to a line
bundle ${\cal L}^{++}$ on $A^{++}$ such that the pullback of ${\cal L}$ is
a positive power of ${\cal L}^{++}$.  Since $\pi_1(T)/H$ is central, we
further see that ${\cal L}^{++}$ must be the pullback of a $G$-invariant
line bundle on $A$.  But by construction of $A^{++}$, the pullback of a
$G$-invariant line bundle depends only on its class in the N\'eron-Severi
group.  Since $G$ is irreducible, $\NS(A)^G\cong \Z$ and thus the group of
pulled back $G$-invariant line bundles is also isomorphic to $\Z$, so that
${\cal L}^{++}$ must be the ample generator of that group.

We thus see that $\pi_1(T).G$ acts as a reflection group on $\tilde{T}$ iff
its image acts as a reflection group on the $\G_m$-torsor corresponding to
$(A^{++},{\cal L}^{++})$.  Note that in general a $\G_m$-equivariant
automorphism of a $\G_m$-torsor $Y\to X$ is a reflection iff it acts as a
reflection on $X$ and acts trivially on the fiber over some codimension 1
fixed point.  (In general, the action on the fiber over a fixed point is
multiplication by a scalar, and thus either the entire fiber is fixed or no
point of the fiber is fixed.)  This, of course, is a purely algebraic
condition on the corresponding line bundle: an equivariant line bundle
induces a 1-dimensional character of the stabilizer at any point, and the
reflections on the $\G_m$-torsor are the elements that stabilize a
codimension 1 point with eigenvalue 1.

\begin{prop}
  Let $G$ act on the pair $(A,{\cal L})$, with associated $\G_m$-torsor
  $T$, and let $(A^{++},{\cal L}^{++})$ be the corresponding cover, with
  associated $\G_m$-torsor $T^{++}$.  Then $\pi_1(T).G$ acts as a reflection
  group on $\tilde{T}$ iff the group $\mu_d.\ker(A^{++}\to A).G$ acts as a
  reflection group on $T^{++}$, where $d$ is such that ${\cal L}^{++}$ is a
  $d$-th root of the pullback of ${\cal L}$.
\end{prop}

\begin{proof}
  If $\pi_1(T).G$ acts as a reflection group, then so does the quotient
  $\mu_d.\ker(A^{++}\to A).G$.  Conversely, if $\mu_d.\ker(A^{++}\to A).G$ is a
  reflection group, then so is $\ker(A^{++}\to A).G$ on $A^{++}$.  Since $A^{++}$
  is the universal $G$-central cover of $A$, the isogeny $A^{++}\to A$ factors
  through $A^+$ and thus we conclude that $\ker(A^+\to A).G$ acts as a
  reflection group on $A^+$, so that $G$ is crystallographic.  That
  $\pi_1(T).G$ is a reflection group then follows from the argument above.
\end{proof}

As before, the condition in terms of a finite cover makes sense in
arbitrary characteristic, so we say that $G$ is {\em strongly
  crystallographic} for $(A,{\cal L})$ if $\mu_d.\ker(A^{++}\to A).G$ is a
reflection group on the $\G_m$-torsor associated to $(A^{++},{\cal L}^{++})$.

One important caveat is that unlike reflections on $A$, being a reflection
on the $\G_m$-torsor is {\em not} a closed condition.  To be precise, given
a family $S\to \Aut(A,{\cal L})$ lying over a given pointed reflection,
there is indeed an induced closed subscheme of strong reflections (which
generates a well-defined strong reflection subgroup): the fixed locus in
the $\G_m$-torsor is invariant under $\G_m$, so we may first quotient by
$\G_m$ to get a closed subscheme of $A\times S$ and then apply the proper
projection to $S$.  However, this construction does not respect base change
in general!  In addition, the closed subgroup generated by a closed
subscheme can also fail to behave well under base change, so in particular
reduction to finite characteristic.

We can now state our second main theorem.

\begin{thm}
  Let $G\subset \Aut(A,{\cal L})$ (with ${\cal L}$ ample) be an irreducible
  strongly crystallographic reflection group which is not one of five
  counterexamples in finite characteristic.  Then
  $\bigoplus_{d\ge 0} \Gamma(A;{\cal L}^d)^G$
  is a free polynomial ring.
\end{thm}

\begin{rem}
  Note that for any crystallographic lattice $A$ for $G$, there is {\em
    some} line bundle for which $G$ is strongly crystallographic.  Indeed,
  the reflection subgroup $R\subset\G_m.\ker(A^{++}\to A).G$ surjects on
  $\ker(A^{++}\to A).G$ and we may take ${\cal L}$ to be the sheaf of $R\cap
  \G_m.\ker(A^{++}\to A)$-invariant sections of ${\cal L}^{++}$.
\end{rem}

\begin{rem}
  Since the $\G_m$-torsor $T$ is the complement of the origin in the cone
  $\Spec(\bigoplus_d \Gamma(A;{\cal L}^d))$, the conclusion is equivalent
  to there being an isomorphism $T/G\cong \A^{n+1}\setminus\{0\}$; the
  action of $\G_m$ on $T$ induces an action on the affine space and thus a
  grading on the corresponding polynomial ring.
\end{rem}

\begin{rem}
  In characteristic 2, this condition is stronger than necessary, and
  indeed there are several examples where the pullback of the tautological
  bundle on the weighted projective space does not satisfy the hypothesis!
  Indeed, it follows by rank 1 considerations that in characteristic $p$,
  the reflection subgroup of $\G_m.\ker(A^{++}\to A).G$ never contains
  $\mu_{p^2}$, while there are cases where the group associated to the
  tautological bundle contains $\mu_4$ or even $\mu_8$.  This is not a
  major issue, as it only arises when ${\cal L}$ is a power of a bundle
  which {\em is} strongly crystallographic, and it is easy to tell when a
  Veronese of a graded polynomial ring is a graded polynomial ring.
\end{rem}

This is equivalent to a stronger version of the identification of the
quotient: there is a morphism of {\em stacks} $[A/G]\to X$ where $X$ is the
{\em stack} version of a weighted projective space and the morphism induces
an isomorphism of coarse moduli spaces, namely the composition
\[
  [A/G]\cong [T_{\cal L}/G_m\times G]\to [(T_{\cal L}/G)/\G_m]\cong
  [(\A^{n+1}\setminus 0)/\G_m].
\]
Conversely, given a morphism $[A/G]\to X$ which induces an isomorphism of
coarse moduli spaces, we can pull back the tautological $\G_m$-torsor
$\A^{n+1}\setminus 0$ over $X$ to obtain a $\G_m$-torsor over $[A/G]$,
which may in turn be represented as $[T/G]$ where $T$ is the pulled back
$G$-equivariant $\G_m$-torsor over $A$.  We thus get a morphism $[T/G]\to
\A^{n+1}\setminus 0$ of $\G_m$-torsors that factors through the coarse
moduli space $T/G$.  Since the quotients by $\G_m$ are isomorphic, the
induced morphism
\[
\Gamma(\A^{n+1}\setminus 0;\sO)\to
\Gamma(T;\sO)^G
\]
is an isomorphism in sufficiently large degree, so that they have the same
function field, and thus (since both algebras are integrally closed) are
isomorphic.

\subsection{Checking polynomial invariants}

In general, it can be difficult to compute the {\em entire} invariant ring
directly, but luckily there is a fairly straightforward way (analogous to
\cite[Prop.~16]{KemperG:1996}) to test whether a given collection of $n+1$
invariants generates the invariant ring.

\begin{lem}
  Let $A$ be an $n$-dimensional abelian variety, let $G\subset \Aut(A)$,
  and let ${\cal L}$ be a $G$-equivariant ample line bundle on $A$.  Then
  \[
  \lim_{d\to\infty} d^{-n}\dim(\Gamma(A;{\cal L}^d)^G) = \frac{\deg({\cal L})}{|G|}.
  \]
\end{lem}

\begin{proof}
  Suppose first that $G$ is reduced.  Then we have $\dim\Gamma(A;{\cal
    L}^d)=d^n\deg({\cal L})$, and since $G$ acts faithfully, it to first
  order divides dimensions by $G$.  (By Galois theory, the invariant {\em
    field} is a vector space of dimension $|G|$, and thus the invariant
  {\em ring} is a bundle of rank $|G|$.)

  More generally, we may consider the action of the reduced group $G/(G\cap
  A)$ on $A/(G\cap A)$.  The degree of the induced line bundle on the
  quotient is $\deg({\cal L})/|G\cap A|$, so that the conclusion remains
  valid as stated.
\end{proof}

\begin{rem}
  The reader is cautioned that taking invariants of a faithful action of a
  nonreduced group on a more general scheme does not always divide degrees
  as expected.
\end{rem}

\begin{lem}
  Let $R=k[x_0,\dots,x_n]$ be a graded polynomial ring with generators of
  degrees $d_0$,\dots,$d_n$ with $\gcd(d_0,\dots,d_n)=1$.  Then
  \[
  \lim_{d\to\infty} d^{-n}\dim(R[d]) = \frac{1}{n!d_0d_1\cdots d_n}.
  \]
\end{lem}

\begin{rem}
  We may represent $R$ as a subring $k[y_0^{d_0},\dots,y_n^{d_n}]$ of a
  free polynomial ring $R^+$, a free $R$-module of rank $d_0\cdots d_n$ with
  $\dim(R^+)\sim d^n/n!$.
\end{rem}

\begin{lem}
  Let $R$ be an integrally closed graded ring with $R_0=k$, and suppose
  that the homogeneous elements $x_0,\dots,x_n\in R$ are algebraically
  independent elements such that $R/(x_0,\dots,x_n)$ is finite-dimensional.
  If
  \[
  \lim_{n\to\infty} d^{-n}\dim(R[d]) = \frac{1}{n! \deg(x_0)\cdots
    \deg(x_n)},
  \]
  then $R=k[x_0,\dots,x_n]$.
\end{lem}

\begin{proof}
  Algebraic independence implies $k[x_0,\dots,x_n]\subset R$ and
  finite-dimensionality of $R/(x_0,\dots,x_n)$ implies that $R$ is a
  finitely generated module over $k[x_0,\dots,x_n]$.  The growth condition
  implies that this module has generic rank 1, and thus that the field of
  fractions of $R$ is $k(x_0,\dots,x_n)$.  It follows that $R$ is the
  integral closure of the integrally closed ring $k[x_0,\dots,x_n]$.
\end{proof}

\begin{lem}
  Let ${\cal L}$ be an ample line bundle on the smooth variety $X$, with
  associated homogeneous coordinate ring $R:=\bigoplus_d \Gamma(X;{\cal
    L}^d)$.  Suppose $x_0,\dots,x_n\in R$ are homogeneous elements of
  positive degree with no common zero on $X$.  Then they are algebraically
  independent.
\end{lem}

\begin{proof}
  Suppose WLOG that $\deg(x_0)=\min_i\deg(x_i)$.  An algebraic relation
  between $x_0,x_1,\dots,x_n$ induces an algebraic relation between
  $x_0,x_1^{\deg(x_0)},\dots,x_n^{\deg(x_0)}$, so we may as well replace
  ${\cal L}$ by ${\cal L}^{\deg(x_0)}$ and reduce to the case that
  $\deg(x_0)=1$.  Since ${\cal L}$ is ample, the sections $x_1,\dots,x_n$
  have a common zero $p$ (the intersection number is positive) which is not
  a zero of $x_0$ by hypothesis, so that the
  functions $f_i:=x_i/x_0^{\deg(x_i)}$ give elements of the local ring at
  $p$ that cut out a finite-dimensional module.  Since the local ring at $p$
  is regular, this implies that $f_1,\dots,f_n$ are a regular sequence
  and thus algebraically independent.  This in turn implies that
  $x_0$,\dots,$x_n$ are algebraically independent as required.
\end{proof}

Putting this all together gives us the following.

\begin{prop}
  Let $G\subset \Aut(A)$ be a group and ${\cal L}$ an $G$-equivariant ample
  line bundle, and suppose that $x_i\in \Gamma(A;{\cal L}^{d_i})^G$, $0\le
  i\le n$, are invariant sections such that
  \[
  |G| = \deg({\cal L})n! d_0\cdots d_n.
  \]
  If the zero loci of $x_0,\dots,x_n$ have no common intersection, then
  \[
  \bigoplus_d \Gamma(A;{\cal L}^d)^G
  \cong
  k[x_0,\dots,x_n].
  \]
\end{prop}

\begin{proof}
  The no common intersection condition ensures that $x_0,\dots,x_n$
  are algebraically independent and that the quotient by the corresponding
  ideal of the homogeneous coordinate ring of $A$ is finite-dimensional.
  This immediately implies that the corresponding quotient of the invariant
  ring is finite dimensional.  Since the homogeneous coordinate ring of $A$
  is integrally closed, so is the subring of invariants, and thus the claim
  follows.
\end{proof}
  
\subsection{First steps towards classification}

Given a surjective morphism $\prod_r E_r\to A$, we can always choose
$n=\dim(A)$ of the curves so that the corresponding morphism $\prod_{1\le
  i\le n} E_i\to A$ is still surjective.  (E.g., we may consider the curves
in some order and omit any curve that does not increase the dimension of
the image.)  In particular, if $A$ admits an irreducible reflection group
action, then there is a surjective morphism of this form in which all of
the curves are isomorphic (i.e., from a single orbit of root curves), and
thus $A$ is isogenous to a power $E^n$.  This lets us mostly identify the
endomorphism ring of $A$:
\[
\End(A)\otimes_\Z \Q\cong \Mat_n(\End(E)\otimes_\Z \Q).
\]
By standard results of the theory of elliptic curves, the coefficient ring
$\End(E)\otimes_\Z \Q$ is either $\Q$ (the generic case), an imaginary
quadratic field, or a quaternion algebra $Q_p$ over $\Q$ ramified at a
single finite prime $p$ (when $E$ is supersingular of characteristic $p$).
It follows that the image of $\Q[G]$ inside $\End(A)\otimes_\Z\Q$
(factoring through $\Q[\Lin(G)]$) is $\Mat_n(K)$ for $K\subset
\End(E)\otimes_\Z\Q$, so that $K$ ranges over the same list.  We say that
$G$ is a {\em real} reflection group if $K=\Q$, a {\em complex} reflection
group if $K$ is an imaginary quadratic field, and a {\em quaternionic}
reflection group if $K$ is a quaternion algebra.  This, of course, simply
distinguishes between the possibilities for the image of $\R[G]$ in
$\End(A)\otimes_\Z \R$.

The key point, of course, is that this identification of a quotient of
$\Q[G]$ with a matrix algebra is nothing other than a choice of {\em
  linear} representation of $G$.  Moreover, if $g\in G$ is a reflection,
then $g-1\in \End(A)$ has 1-dimensional image and thus $g-1\in
\End(A)\otimes_\Z \Q$ has rank 1 as a matrix.  In other words, this linear
representation of $G$ exhibits $G$ as a reflection group in the usual
sense!

For the converse, let $K$ be a division algebra over $\Q$, and let
$G\subset \Mat_n(K)$ be a reflection group such that $\Q[G]\surj
\Mat_n(K)$, and thus in particular $\Z[G]$ maps to an order $\rho(\Z[G])$
in $\Mat_n(K)$.  For any reflection $r\in G$, let $R_r\subset
\rho(\Z[G])$ be the saturation of $(1-r)\rho(\Z[G])(1-r)$ (a subring,
albeit with different multiplicative unit).  Over $\Q$, we can conjugate
$r\in \Mat_n(K)$ to be 0 outside the $11$ entry, so that the saturation
over $\Q$ is isomorphic to $K$; it follows that $R$ is an order in $K$.

Let $V_R$ denote the $(\Z[G],R)$-bimodule $\rho(\Z[G])R$.  (This depends
only on the ``root'' of the reflection and not on the reflection itself.)
Let $G_r\subset G$ be the subgroup of elements such that $gR\subset R$, so
that there is an induced morphism $R\to \Res^G_{G_R} V_R$ and thus (by
Frobenius reciprocity) a morphism $\Ind^G_{G_R}R\to V_R$.  This latter
morphism is a surjective morphism from a free $R$-module, and thus the
kernel $K$ is torsion-free.

Now, suppose we are given an elliptic curve $E$ with an action of $R$ as
endomorphisms.  We can tensor $E$ with any right $R$-module (this is clear,
and functorial, for free $R$-modules and extends via free resolutions) and
thus in particular we obtain a short exact sequence
\[
0\to K\otimes_R E\to (\Ind^G_{G_R} R)\otimes_R E\to V_R\otimes_R E\to 0
\]
of abelian varieties with induced pointed actions of $G$.  For any
reflection $r'\in G$, $(1-r')V_R\otimes_R E$ is the tensor product of $E$
with a rank 1 $R$-module, so is itself an elliptic curve, and thus $G$ acts
on $V_R\otimes_R E$ as a reflection group (and is easily seen to be
irreducible).  Moreover, the middle term in the short exact sequence is the
product over an orbit of root curves, and thus the kernel is already
integral even before we include the other orbits, so that $G$ is
crystallographic.

In other words, for any reflection group, any reflection, and any elliptic
curve with the appropriate endomorphism structure, there is an induced
abelian variety on which the group acts as an irreducible crystallographic
reflection group.  Moreover, this induced abelian variety is universal in
the following sense: given any abelian variety $A$ on which $G$ acts as an
irreducible reflection group, there is an induced $G$-equivariant isogeny
$V_R\otimes_R E\to A$ where $E$ is the root curve corresponding to $R$.
(Indeed, the map from $(\Ind^G_{G_R} R)\otimes_R E$ is nothing other than
the standard map from the product of root curves, and necessarily
annihilates $K\otimes_R E$.)

This gives us a way to classify pointed crystallographic actions of
irreducible reflection groups.  The kernel $K_A$ of the isogeny
$V_R\otimes_R E\to A$ must be a $G$-invariant finite subgroup scheme that
meets the original root curve in the identity, as otherwise the induced
map of root curves would be a nontrivial isogeny.  In particular, if $r$
is a reflection with the given root curve, then $(r-1)K_A\subset K_A$ is
contained in the root curve, so must be trivial.  In other words, $K_A$ is
contained in the subgroup of $V_R\otimes_R E$ fixed by $r$, which by
$G$-invariance implies that it is contained in the subscheme fixed by the
normal subgroup generated by all conjugates of such reflections.  Since $G$
is irreducible, this subscheme is finite (indeed, it is contained in the
$|G|$-torsion), which already makes the problem finite.  A further
simplification comes from our desire to be crystallographic, which forces
$K_A$ to be generated by its intersections with the various root curves.

One caution here is that the moduli space of curves with endomorphism ring
containing $R$ might have multiple components.  The obvious case is when
$R$ is an imaginary quadratic order with class number $>1$, but (a) in that
case the different components are related by the Galois group, and (b) this
case doesn't actually arise from reflection groups.  A more serious issue
is when $R$ is a nonmaximal order in an imaginary quadratic field, as then
we have a Galois orbit (in characteristic 0) for each order containing $R$.
(Again, those orbits will be a single point in the cases of interest.)
Similarly, if $R$ is a nonmaximal order in $Q_p$, then we get a point for
each extension to a maximal order.  (This is an issue even when there is a
unique supersingular curve of characteristic $p$, as there can still be
multiple conjugacy classes of actions $R\to \End(E)$.)  This is not too
serious an issue as long as one is careful, though.  One should also note
that we get a different universal action for each orbit of root curves,
each of which may be used for the full classification, and it turns out
that there is only one case in which {\em every} order is nonmaximal, and
in that case the orders extend uniquely.

Given a pointed crystallographic action of $G$, the next problem is to
classify the associated non-pointed crystallographic actions on quotients
of $A$ by $T\subset A^G$, where $T$ is disjoint from every root curve.
Such an action is in particular a non-pointed action of $G$ on $A/T$ by
reflections.  Non-pointed actions are classified (modulo conjugation by
translations) by classes in $H^1(G;A/T)$.  The condition for a reflection
$r\in G$ to be a reflection under the twisted action is that the twisted
action have a fixed point, or equivalently for the restricted class in
$H^1(\langle r\rangle;A/T)$ to be trivial, and thus we need classes which
are trivial on a suitable generating set of reflections.  Since $T$ is
disjoint from every root curve, each reflection in $G$ that lifts to a
reflection in $T.G$ lifts uniquely, and thus the reflection subgroup of
$T.G$ has the form $T'.G$.  Thus, rather than try each $T$ separately and
look for cocycles for which $T'=T$, we may as well consider only maximal
elements in the corresponding poset.  (Note that we will use a more
hands-on argument in infinite families, and thus this will be a finite
further classification.)

There is one uniform construction of such torsors.  If the minimal
polarization of $A$ has odd degree, then the construction of \cite{PR}
gives a natural class in $H^1(G;\hat{A}[2])\cong H^1(G;A[2])$ such that the
polarization is represented by a line bundle on the torsor.  But the
argument of Proposition 3.12(d) op.~cit.~shows that the restriction of this
class to any cyclic subgroup of $G$ is trivial, and thus in particular
every reflection remains a reflection on the torsor.  (Of course, in most
cases, this canonical $G$-equivariant ``theta'' torsor will end up being
trivial!)

In the real (i.e., Coxeter) case, $G$ has at most two conjugacy classes of
reflections, neither of which alone suffices to generate the group, and
thus the cohomology class must be trivial along every reflection.  But then
we can simultaneously trivialize the cohomology class along the {\em
  simple} reflections, and thus see that the class is trivial.  In other
words, any crystallographic action of an irreducible {\em Coxeter} group
has a common fixed point, so is equivalent to a pointed action.  A similar
reduction applies in most of the sporadic complex and quaternionic cases as
well, as most of the time $G$ can be generated by $n$ reflections such that
every conjugacy class of reflections contains a power of one of the
generators.

One useful fact for ruling out torsors is the following trivial fact.

\begin{lem}
  Let $G\subset \Aut(A)$ be a finite subgroup such that $G\cap A=0$.  If
  $G$ has a normal subgroup that fixes a unique closed point of $A$, then
  so does $G$.
\end{lem}

More generally, we may use the Hochschild-Serre sequence, which in
the initial terms reads
\[
0\to H^1(G/N;A^N)\to H^1(G;A)\to H^1(N;A)^{G/N}\to
\]
so that in particular if $H^1(N;A)=0$, we can reduce to cohomology
with coefficients in the typically finite group (scheme) $A^N$.
This happens quite often, and indeed one can use $N=Z(G)$ when $G$ has
nontrivial center, due to the following fact.

\begin{lem}
  Let $g\in \Aut_0(A)$ be an element of finite order $d$ such that
  $1+g+\cdots+g^{d-1}=0$.  Then $H^1(\langle g\rangle;A)=0$.
\end{lem}

\begin{proof}
  Since $1+g+\cdots+g^{d-1}=0$, $(1-g)$ is surjective, and thus
  $H^1(\langle g\rangle;A)=\ker(1+g+\cdots+g^{d-1})/\im(1-g)=0$.
\end{proof}

\begin{cor}
  Suppose the irreducible subgroup $G\subset \Aut_0(A)$ has a cyclic normal
  subgroup $\langle g\rangle$ of order $N>1$.  Then $H^1(G;A)\cong
  H^1(G/\langle g\rangle;A^g)$.  Furthermore, if $N$ is not a prime power,
  then $H^1(G;A)=0$.
\end{cor}

\begin{proof}
  For any $d|N$, the abelian subvariety $\ker(\Phi_d(g))^0$ is
  $G$-invariant (where $\Phi_d$ is the cyclotomic polynomial), so is either
  0 or $A$.  Since the cyclotomic polynomials are relatively prime, the
  kernels meet in finite subgroups of $A$, and thus there is precisely one
  $d$ for which the kernel is $A$.  Since $\Phi_d(g)=0$ implies $g^d=1$, we
  conclude that $d=N$ and that $g$ has minimal polynomial $\Phi_N(g)$.  In
  particular, we see that $H^1(\langle g\rangle;A)=0$, implying the first
  claim.

  When $N$ is not a prime power, the fixed locus of $g$ on any $\langle
  g\rangle$-equivariant torsor is a coset of $A^g\subset
  A[\Phi_N(1)]=A[1]=0$.  In particular, $g$ fixes a unique point on any
  $G$-equivariant torsor, and thus so does $G$.
\end{proof}

\begin{cor}
  Suppose $A$ has characteristic $p$ and $p$-rank 0 and the irreducible
  subgroup $G\subset \Aut_0(A)$ has a cyclic normal subgroup of order
  $p^l\ne 1$.  Then $H^1(G;A)=0$.
\end{cor}

\begin{proof}
  The same argument applies, except that $A[\Phi_N(1)]=A[p]$.  But the
  $p$-rank condition ensures that $A[p]$ has a unique closed point, and
  thus $g$ still fixes a unique closed point as required.
\end{proof}

\begin{rem}
  For reflection groups, the $p$-rank condition is equivalent to asking for
  the root curves to be supersingular.  This is automatic in the
  quaternionic cases and quite common in the complex cases.
\end{rem}

\subsection{Cohen-Macaulay invariants}

In the usual theory of polynomial invariants in characteristic 0, a major
tool for the computation is the observation that rings of polynomial
invariants are Cohen-Macaulay, as a consequence of the general fact that
invariant subrings of CM rings are CM in good characteristic.  This
argument fails in our case, however, for the simple reason that cones over
abelian varieties are typically not Cohen-Macaulay!  In general, one has
the following.

\begin{prop}
  Let $X$ be a Cohen-Macaulay algebraic variety and let ${\cal L}$ be an
  ample line bundle on $X$.  Then the graded algebra $\bigoplus_{d\ge 0}
  \Gamma(X;{\cal L}^d)$ is Cohen-Macaulay iff $H^p(X;{\cal L}^d)=0$ for
  $0<p<\dim(X)$, $d\in \Z$.
\end{prop}

Since $A$ is regular, it is Cohen-Macaulay as a scheme, and thus the
Cohen-Macaulayness of its cones reduces to the cohomological condition.
Now, line bundles of nonzero degree on an abelian variety have a single
nonvanishing cohomology class, and thus since an ample bundle ${\cal L}$
has $\deg({\cal L})$ global sections, it has vanishing higher cohomology,
and Serre duality implies that ${\cal L}^{-1}$ has no cohomology below
degree $\dim(A)$.  It follows that any nonzero {\em power} of an ample
bundle has no intermediate cohomology and thus that only $\sO_A$ needs to
be checked.  However, the cohomology of $\sO_A$ is well-understood and is
in particular nonzero in every degree, so that $A$ can have CM cones only
when $\dim(A)=1$.

For any finite subgroup $G\subset \Aut(A)$ with $|G|$ invertible on $A$,
the quotient $A/G$ is again Cohen-Macaulay as a scheme, so that again we
need to control vanishing of $H^p(A/G;{\cal L}^d)$.  Since we are in good
characteristic, the cohomology of the quotient scheme agrees with the
equivariant cohomology, and thus we have
\[
H^p(A/G;{\cal L}^d) \cong
H^p([A/G];{\cal L}^d) \cong H^p(A;{\cal L}^d)^G
\]
so again automatically have vanishing of intermediate cohomology for
$d\ne 0$.  Moreover, abelian varieties satisfy $H^p(A;\sO_A)\cong \wedge^p
H^1(A;\sO_A)$ and $H^1(A;\sO_A)\cong \Lie(\Pic^0(A))$.  We thus know how
$G$ acts on $H^1(A;\sO_A)$; and in particular in the case of a complex
reflection group, $H^1(A;\sO_A)$ may be identified with the dual of the
standard reflection representation of $G$.  Moreover, $H^1(A;\sO_A)$ is
invariant under separable isogenies.

Since (by the classification) quaternionic reflection groups do not have
good characteristic, we consider only real and complex reflection groups.
For convenience, we state the result for $A=E^n$ in characteristic 0, but
it is straightforward to deduce analogous results in characteristic prime
to both the order of $G$ and the degree of the relevant isogeny.

\begin{prop}\label{prop:CM_and_CY}
  Let $E$ be an elliptic curve in characteristic 0, and let $G\subset
  \GL_n(\End(E))\subset \GL_n(\C)$.  If $\wedge^p(\C^n)^G=0$ for $0<p<n$,
  then $E^n/G$ has Cohen-Macaulay cones, and if $G\subset \SL_n(\C)$, then
  $E^n/G$ is Calabi-Yau.
\end{prop}

\begin{proof}
  The hypothesis is precisely that $(\wedge^pH^1(E^n;\sO_{E^n}))^G=0$,
  which is precisely what we need to have Cohen-Macaulay cones.

  For $G\subset \SL_n(\C)$, we have already established the requisite
  vanishing of cohomology to be Calabi-Yau, so it remains only to show
  that $\omega_{E^n/G}\cong \sO_{E^n/G}$.  In general, the action
  of $G$ on $\Hom(\omega_{E^n},\sO_{E^n})$ factors through the determinant,
  and thus any chosen isomorphism $\omega_{E^n}\cong \sO_{E^n}$ is
  $G$-invariant, so descends to an isomorphism $\omega_{E^n/G}\cong
  \sO_{E^n/G}$ as required.
\end{proof}

Note that to show vanishing of intermediate cohomology, it suffices to show
it for some subgroup of $G$.  It turns out that this vanishing actually
holds for derived subgroups of reflection groups.

\begin{prop}
  Let the finite group $G\subset \GL_n(\C)$ be an irreducible complex (or
  real) reflection group.  Then $\wedge^p(\C^n)^{G'}=0$ for $0<p<n$.
\end{prop}
  
\begin{proof}
  For $G$ of rank $2$, we note that $G'\ne 1$ since otherwise $G$ would be
  abelian and thus reducible.  But a finite order element of $\SL_2(\C)$
  cannot fix a point in $\C^2$, and thus the claim holds in the only degree
  $p=1$ of interest.

  For $G$ imprimitive of rank $>2$, the diagonal subgroup of $G'$ consists
  of the product-1 subgroup of $\mu_d^n$, which acts diagonally on the
  induced basis of $\wedge^p(\C^n)$, and it is easy to see that any basis
  element is acted on nontrivially by some element of the diagonal
  subgroup.

  Finally, for $G$ primitive of rank $>2$, we may compute
  $\dim\wedge^p(\C^n)^{G'}$ via characters and thus reduce to checking that
  \[
  \frac{1}{|G'|}\sum_{g\in G'} \det(1+sg^{-1})
  =
  1+s^n
  \]
  in the remaining 15 cases.
\end{proof}

\begin{rem}
  This could presumably be shown more directly from known results on the
  structure of semi-invariants in the supercommutative algebra
  $S^*(V)\otimes \wedge^*(V^*)$, see \cite{SheplerAV:1999}.
\end{rem}

\begin{cor}
  Let $E$ be an elliptic curve of characteristic 0, let $G\subset
  \GL_n(\End(E))$ be a complex reflection group, and let $A$ be any abelian
  variety $G$-equivariantly isogenous to $E^n$.  Then for any subgroup
  $G'\subset H\subset G$, $A/H$ has Cohen-Macaulay cones, and if $H\subset
  \SL_n(\End(E))$, then $A/H$ is Calabi-Yau.
\end{cor}

This suggests that we should also investigate quotients by derived
subgroups of crystallographic groups, which we do (for a suitably
crystallographic version of the derived subgroup) in Section
\ref{sec:derived} below.

\section{Imprimitive groups}
\label{sec:imprim}

\subsection{Strongly crystallographic actions in rank 1}

Both as our first case of the main theorem and as an input to understanding
imprimitive groups, we first consider a question for actions on elliptic
curves: What are the strongly crystallographic actions on pairs
$(E,{\cal L})$?

We first observe that there is no need to consider torsors.

\begin{prop}
  Any subgroup $G\subset \Aut(E)$ with $G\cap E=0$ fixes a point of $E$.
\end{prop}

\begin{proof}
  If $G$ is cyclic with generator $g\ne 1$, then $g-1$ is surjective and
  thus the fixed subscheme is nonempty.  If $G$ is not cyclic, then either
  $E$ has characteristic 2 and $G$ has center $[\mu_2]$ or $E$ has
  characteristic 3 and $G$ has normal $3$-Sylow subgroup $[\mu_3]$; in
  either case, the given cyclic normal subgroup fixes a unique closed
  point, which must therefore be fixed by $G$.
\end{proof}

We may thus restrict our attention to $G\subset \Aut_0(E)$.  To understand
strongly crystallographic actions, we need to pull back to $E^{++}$, which by
construction is dual to the quotient $E/E^G$.  Moreover, we may verify in
each case that $E/E^G\cong E$ so that for our question, it
suffices to consider equivariant structures of $E^G\times G$ on pulled back
line bundles, or equivalently equivariant structures of $\mu_d.E^G\times G$
on the minimal pulled back line bundle.  It will be helpful to consider a
somewhat broader question: What are the finite subgroups of
$\G_m.(E^G\times G)$ that are generated by strong reflections?

In general, given a pair $(A,{\cal L})$, a reflection $g\in \Aut(A)$
preserving ${\cal L}$, and a subgroup $T$ of the root curve $E_g$ which is
isotropic for the Weil pairing associated to ${\cal L}$, the strong
reflections in $\G_m.Tg$ can be understood as follows.  Each point of $A$
is fixed by a unique point of $E_g g$, and in particular the subscheme of
$A\times Tg$ where the point of $A$ is fixed is the graph of the morphism
$(1-g)^{-1}T\to T$.  The preimage of $Tg$ in $\Aut(A,{\cal L})$ is a
$\G_m$-torsor over $T$ which pulls back to a $\G_m$-torsor over
$(1-g)^{-1}T$, and there is a $\G_m$-equivariant morphism from that torsor
to $\G_m$ taking each point to the eigenvalue of the corresponding
reflection at that point.  The induced section of that $\G_m$-torsor (i.e.,
where the image is 1) thus induces a morphism
\[
(1-g)^{-1}T\to \G_m.Tg
\]
the image of which is precisely the strong reflection subscheme of
$\G_m.Tg$.  This is clearly constant under translation by the reduced
identity component of $\ker(1-\Lin(g))$, and thus the domain is a finite
subscheme of the quotient ``coroot'' curve $A/\ker(1-\Lin(g))^{0\red}$
(which has an isogeny to $E_g$ given by $1-\Lin(g)$).  For any subscheme of
$\G_m.Tg$, the corresponding reflection subscheme will be the image of a
subscheme of $(1-g)^{-1}T/\ker(1-\Lin(g))^{0\red}$, and thus there are only
finitely many possibilities.  (Finite subschemes of curves are essentially
just effective divisors, so all we can do is reduce the multiplicities; in
particular, each reduced point gives rise to a chain of possible
multiplicities, and it suffices to understand the corresponding chains of
subgroups they generate.)  Note that there is a further simplification
coming from the fact that strong reflections are preserved under
conjugation, causing the morphism to $\G_m.Tg$ to be constant under a
larger subgroup of $(1-\Lin(g))^{-1}T$, and thus allowing us to restrict
our attention to invariant subschemes.

\begin{eg}
  Let $G=[\mu_2]$, so that $E^G\cong E[2]$ and we must consider the action
  of $\G_m.E[2]\times [\mu_2]$ on the line bundle ${\cal L}(E[2])$ (i.e.,
  functions with poles bounded by the subscheme $E[2]$ viewed as a
  divisor).  The strong reflections are controlled by a map $E[4]\to
  \G_m.E[2]\times [\mu_2]$ which factors through $E[4]/E[2]\cong E[2]$.
  Indeed, in odd characteristic, there are four strong reflections, each of
  which fixes (with eigenvalue 1) a different coset of $E[4]/E[2]$.  Each
  subscheme of $E[2]$ gives a strong reflection group, which contains
  $E[2]$ in the five cases where the subscheme has degree 3 or 4.  The
  degree 3 cases do not meet $\G_m$ and correspond to the four
  $[\mu_2]$-invariant bundles of degree 1, while the degree 4 case contains
  $\mu_2$ and corresponds to ${\cal L}(2[0])$.  In characteristic 2, there
  are fewer subschemes of $E[2]$ (9 cases when $E$ is ordinary and 5 when
  $E$ is supersingular).  When the corresponding divisor has degree 3 or 4,
  the subgroup surjects on $E[2]\times [\mu_2]$ and thus gives a strongly
  crystallographic action, again with ${\cal L}$ either invariant of degree
  1 or ${\cal L}(2[0])$.
\end{eg}

\begin{eg}
  Let $G=[\mu_3]$, so that $E^G\cong E[\sqrt{-3}]$ and we must consider the
  action of $\G_m.E[\sqrt{-3}]\times [\mu_3]$ on ${\cal L}(E[\sqrt{-3}])$.
  In this case, the strong reflections lying over $[\zeta_3]$ are
  parametrized by $E[\sqrt{-3}]\cong E[3]/E[\sqrt{-3}]$, and the subschemes
  of degree 2 or 3 of $E[\sqrt{-3}]$ give strongly crystallographic
  actions, corresponding to invariant line bundles of degree 1 or their
  common cube.
\end{eg}

\begin{eg}
  Let $G=[\mu_6]$ so that $E^G=0$ and the pulled back bundle is ${\cal
    L}([0])$.  Since order 6 reflections are naturally products of order 2
  and order 3 reflections, it suffices to consider the possible schemes of
  order 2 or 3 reflections.  For order 2, the strong reflections are
  classified by subschemes of $E[2]$, but the eigenvalue is constant on
  $[\mu_3]$-orbits, so in odd characteristic there are two generating
  strong reflections of order 2 while in characteristic 2 the strong
  reflections of order 2 generate a subgroup in the chain
  \[
  1\subset [\mu_2]\subset \mu_2\times [\mu_2].
  \]
  Similarly, for order 3, the strong reflections over $[\zeta_3]$ are
  classified by subschemes of $E[\sqrt{-3}]$, but the eigenvalue is
  constant on $[\mu_2]$-orbits, so there are again only two strong
  reflections in the image, while in characteristic 3 the different
  subgroups that can arise are
  \[
  1\subset [\mu_3]\subset \mu_3\times [\mu_3].
  \]
  In characteristic not 2 or 3, each ramified point of $E\to E/G$ has an
  associated strong reflection group, of orders 2, 3, and 6 respectively,
  generating $\mu_6\times [\mu_6]$.  Thus the strong reflection subgroups
  correspond to triples $(d_2,d_3,d_6)$ with $d_m|m$, i.e.,
  $\mu_{d_2}\times \mu_{d_3}\times \mu_{d_6}\subset \mu_2\times \mu_3\times
  \mu_6$.  The strongly crystallographic actions correspond to triples with
  $\lcm(d_2,d_3,d_6)=6$, giving four equivariant structures on ${\cal
    L}([0])$, two on ${\cal L}(2[0])$, two on ${\cal L}(3[0])$ and one on
  ${\cal L}(6[0])$.
\end{eg}

\begin{eg}
  For $[\mu_4]$ (with fixed locus $E[1-i]$), the strong reflections over
  $[\zeta_4]$ are classified by $E[1-i]$ while the strong reflections over
  $[-1]$ are classified by $E[2]$, but with eigenvalue invariant under
  $[i]$.  In odd characteristic, this gives three independent generators:
  two of order 4 and one of order 2, so that the strong reflection
  subgroups have the product form $\mu_{d_1}\times \mu_{d_2}\times
  \mu_{d_3}\subset \mu_2\times \mu_4\times \mu_4$.  This gives rise to two
  strongly crystallographic actions of degree 1, three of degree 2, and one
  of degree 4.  In characteristic 2, we have two chains of subgroups to
  consider.  The subgroups generated by reflections over $[\zeta_4]$ are
  \[
  1\subset [\mu_4]\subset \mu_2\times \alpha_2\times [\mu_4],
  \]
  while the subgroups generated by reflections over $[-1]$ are
  \[
  1\subset [\mu_2]\subset \alpha_2\times [\mu_2]\subset \mu_2\times
  \alpha_2\times [\mu_2],
  \]
  so that the strongly crystallographic actions correspond to
  \[
  \alpha_2\times[\mu_4] = \langle [\mu_4],\alpha_2\times [\mu_2]\rangle
  \]
  and
  \[
  \mu_2\times \alpha_2\times[\mu_4].
  \]
\end{eg}
 
\begin{rem}
  This gives an example of the failure of ``strong reflection'' to respect
  base change.  The subscheme of strong reflections over $[-1]$ is a
  multiplicity 3 point in $\Aut(A,{\cal L})$, and in particular has a
  subscheme generating $\mu_2\times [\mu_2]$.  But the closed subscheme of
  strong reflections in $\mu_2\times [\mu_2]$ is reduced, so this is not a
  strong reflection group!
\end{rem}

\begin{eg}
  If $E$ is a supersingular curve of characteristic 3, we may consider
  $G=\Aut_0(E)$, which again satisfies $E^G=0$.  The subgroup of order $3$
  gives rise to the chain
  \[
  0\subset [\mu_3]\subset \mu_3\times [\mu_3]
  \]
  while there are two order 2 strong reflections (the eigenvalue being
  constant under $[\mu_3]$) and two preimages for each order 4 reflection
  differing by $\zeta_4$.  We thus find that the strong reflection
  subgroups containing $[\mu_3]$ (which are the only ones we will need to
  understand!) are naturally identified with subgroups $\mu_{d_1}\times
  \mu_{d_2}\subset \mu_{12}\times \mu_4$.
\end{eg}

\begin{eg}
  Similarly, if $E$ is a supersingular curve of characteristic $2$, then we
  may consider $G=\Aut_0(E)$.  We only consider the strong reflection
  subgroups containing a preimage $Q_8$.  In that case, the preimage of
  $Q_8$ is one of $Q_8$ or $\mu_2\times Q_8$, to which one should add one
  or both of the two conjugacy classes of order 3 refections.  This gives a
  total of 6 strongly crystallographic actions (as well as additional cases
  in which one enlarges $\mu_2$ to $\mu_4$ or $\mu_8$, which turn out to
  also have polynomial invariants).
\end{eg}

\begin{eg}
  Finally, since $G=Q_8$ acting on the supersingular curve $E/\F_2$ is
  generated by elements of order 4, the fixed subscheme is the kernel of
  Frobenius and thus the full group is $\alpha_2\times Q_8$.  The order 2
  element produces the chain
  \[
  0\subset [\mu_2]\subset \alpha_2\times [\mu_2]\subset \mu_2\times
  \alpha_2\times [\mu_2]
  \]
  while each order 4 element gives a chain
  \[
  0\subset [\mu_4]\subset \mu_2\times \alpha_2\times [\mu_4],
  \]
  and thus the strong reflection subgroups containing $[\mu_2]$ correspond
  to products of subgroups of $Q_8$ with $1$, $\alpha_2$, or $\mu_2\times
  \alpha_2$.  In particular, the strongly crystallographic actions are
  $\alpha_2\times Q_8$ and $\mu_2\times \alpha_2\times Q_8$, while the
  additional cases with scalars extended to $\mu_4$ and $\mu_8$ turn out to
  also have polynomial invariants.
\end{eg}

\subsection{Classification in rank $\ge 3$}

An imprimitive reflection group of rank $n\ge 3$ is contained in a wreath
product $G_1^n\rtimes S_n$ for some $G_1$, where $G_1$ is a subgroup of the unit
group of the appropriate division ring and one imposes a suitable condition
on the product of the coefficients of the (monomial) matrices.  To be
precise, they all have the form $G_n(G_1,H_1)$
defined as follows.  Let $G_1$ be any group and let $H_1\subset G_1$ be any
subgroup containing the derived subgroup $G'_1$ of $G_1$.  Then
$G_n(G_1,H_1)$ is defined to be the subgroup of $G_1^n\rtimes S_n$ containing
$S_n$ and the subgroup of $G_1^n$ for which the product of the coordinates
lies in $H_1$.

\begin{lem}
  For any elliptic curve $E$, the group scheme $\Aut(E)^n\rtimes S_n$ is a
  reflection group on $E^n$.
\end{lem}

\begin{proof}
  There are two classes of reflections in $\Aut(E)^n\rtimes S_n$:
  the ``rank 1'' reflections that act as a reflection on a single
  coordinate and as the identity on all others (noting that $g\in \Aut(E)$
  is a reflection iff $\Lin(g)\ne 1$), and the ``rank 2'' reflections that
  act as $(x_i,x_j)\mapsto (gx_j,g^{-1}x_i)$ on two coordinates, fixing all
  others.  (It is easy to see that these are the {\em only} reflections in
  the group.)  In particular, $\Aut(E)^n\rtimes S_n$ contains $S_n$ (which
  is generated by rank 2 reflections) and $\Aut(E)^n$ (which is generated
  by rank 1 reflections), so is a reflection group.
\end{proof}

We are interested in understanding the irreducible finite reflection groups
contained in $\Aut(E)^n\rtimes S_n$.  Such a subgroup $G$ must map to a
transitive subgroup of $S_n$ and thus in particular must contain a rank 2
reflection (as otherwise the image in $S_n$ would be trivial), from which
it follows that it must surject on $S_n$.  In particular, we may choose
preimages of the simple transpositions $(i,i+1)$ in $G$.  We may then
conjugate by a suitable element of $\Aut(E)^n$ to make those preimages lie
in $S_n$ and thus, up to conjugation, $G$ contains $S_n$.  Having made such
a choice, $G$ is then determined by its intersection with the ``diagonal''
subgroup $\Aut(E)^n$, which must be an $S_n$-invariant subgroup.

\begin{prop}
  Let $n\ge 3$.  Then, up to conjugation by $\Aut(E)^n$, the irreducible
  finite reflection subgroups of $\Aut(E)^n\rtimes S_n$ are precisely the
  groups of the form $G_n(G_1,H_1)$ where $\Lin(G_1)\ne 1$ and $H_1$ is
  generated over $G'_1$ by reflections.
\end{prop}

\begin{proof}
  Let $G^+\supset S_n$ be the given finite reflection subgroup, and suppose
  first that $G^+$ is generated by rank 2 reflections.  Let $G_1$ be the
  sub{\em scheme} of $\Aut(E)$ consisting of elements $g$ such that the
  rank 2 reflection $(x_1,x_2)\mapsto (gx_2,g^{-1}x_1)$ is in $G^+$.  Then
  for all $g\in G_1$, $G^+\cap \Aut(E)^n$ contains $(g,g^{-1},1,\dots,1)$
  and thus in particular for $g,h\in G_1$, it contains
  $(gh,g^{-1}h^{-1},1,\dots,1)$.  Since $G^+$ contains $S_n$, it also
  contains $(1,gh,g^{-1}h^{-1},1,\dots,1)$ and thus we may conjugate $(12)$
  by this element to see that $gh\in G_1$ and thus $G_1$ is a group.  Note
  in particular that $(gh,(gh)^{-1},1,\dots,1)\in G^+$ and thus
  $(1,ghg^{-1}h^{-1},1,\dots,1)\in \G^+$, implying by $S_n$-invariance that
  $(G'_1)^n\subset G^+$.  Moreover, the element $(g,g^{-1},1,\dots,1)$ and
  permutations thereof generate the subgroup of $(G_1/G'_1)^n$ on which the
  product is trivial, and thus $G^+$ contains $G_n(G_1,G'_1)$.  Since the
  rank 2 reflections are easily seen to lie in $G_n(G_1,G'_1)$, we conclude
  that $G^+=G_n(G_1,G'_1)$.

  Returning to the general case, if $G^+$ is {\em not} generated by rank 2
  reflections, then the normal subgroup generated by rank 2 reflections is
  still of the form $G_n(G_1,G'_1)$.  Let $H_1$ be the subgroup of $G^+\cap
  \Aut(E)^n$ that fixes all but the first coordinate, viewed as a subgroup
  of $\Aut(E)$; in particular, the rank 1 reflections in $G^+$ have the
  form $x_i\mapsto hx_i$ for some reflection $h\in H_1$.  Conjugating
  $(12)$ by such a rank 1 reflection shows that $H_1\subset G_1$ and we
  immediately conclude that $G^+$ contains $G_n(G_1,H_1)$.  Conversely, if
  $(g_1,\dots,g_n)\in G^+$ then we can multiply by elements of the form
  $(1,\dots,1,g,g^{-1},1,\dots,1)$ to conclude that $(g_n\cdots
  g_1,1,\dots,1)\in G^+$ and thus that $g_n\cdots g_1\in H_1$, so
  that $G^+=G_n(G_1,H_1)$ as required.
\end{proof}

\begin{rem}
  This fails for $n=2$, which gives rise to a finite collection of
  crystallographic actions of imprimitive reflection groups not of the
  above form.  But of course we may feel free to view those additional
  cases as sporadic and deal with them ad hoc.
\end{rem}

\begin{rem}
  Similarly, (strong) reflection subgroups of $(\Aut(E,{\cal L}))^n\rtimes
  S_n$ have the form $G_n(G_1,H_1)$ where $G_1$ is generated by strong
  reflections and $H_1$ is generated over $G'_1$ by strong reflections.
  Note that the product 1 subgroup of $\G_m^n$ acts trivially, so enlarging
  $G_1$ to contain $\G_m$ has no effect on the group that effectively acts.
\end{rem}

Now (still supposing $n\ge 3$) suppose the pointed reflection group
$G_n(G_1,H_1)$ with $G_1\ne 1$ acts irreducibly and crystallographically on
some abelian variety $A$.  If $H_1\ne 1$, then the root curves
corresponding to the rank 1 reflections give rise to an equivariant isogeny
$E^n\to A$.  If $H_1=1$, this fails, but one can still construct suitable
maps $E\to A$ as follows.  Let $r_1$, $r_2$ be distinct rank 2 reflections
mapping to $(12)\in S_n$.  Then $\im(r_1r_2-1)$ is a $2$-dimensional
abelian subvariety of $A$ and the identity component of $\im(r_1r_2-1)\cap
(13)\im(r_1r_2-1)$ is an elliptic curve in $A$.  Moreover, endomorphism
ring considerations tell us that this elliptic curve has an orbit of size
$n$, so again gives an equivariant isogeny $E^n\to A$.

Now, suppose we are given an irreducible crystallographic action on $A$ of
a group $G^+$ with $\Lin(G^+)=G_n(G_1,H_1)$.  Any reflection not in the
subgroup $S_n$ will fix some $S_n$-invariant point, and thus we may assume
some such fixed point is the identity.  The other reflections will then
take the identity to torsion points in the corresponding root curves, and
thus we may assume that $G^+$ is contained in $A[N]\rtimes G_n(G_1,H_1)$
for some $N$.  But then composing with the equivariant isogenies $E^n\to
A^+\to A$ gives an action of a reflection subgroup of $E[N']^n\rtimes
G_n(G_1,H_1)$ on $E^n$ (for some multiple $N'$ of $N$) such that the
quotient by the torsion subgroup is $A$.  In other words, any such
irreducible crystallographic action comes from a group of the form
$G_n(G_1,H_1)$ with $G'_1\normal H_1\normal G_1\subset \Aut(E)$.

\begin{prop}
  If $H_1$ contains a reflection, then $G_n(G_1,H_1)/(G_n(G_1,H_1)\cap
  E^n)$ acts crystallographically on $E^n/(G_n(G_1,H_1)\cap E^n)$.
\end{prop}

\begin{proof}
  Indeed, if $H_1$ contains a reflection, then $E[N]^n.G_n(G_1,H_1)\cong
  G_n(E[N].G_1,E[N].H_1)$ acts as a reflection group on $E^n$ for all $N$,
  implying that the induced action on $A^+$ is as a reflection group.
\end{proof}

If $H_1$ does not contain a reflection, so in particular $H_1=G'_1$, then
it must consist entirely of translations, and thus $\Lin(G_1)$ is abelian.
Since an abelian subgroup of $\Aut_0(E)$ is cyclic, we may choose a
generator $g_0$ of $\Lin(G_1)$.  It then follows easily that
$G_n(G_1,G'_1)$ can be generated over $S_n$ by a single rank 2 reflection
$(x_1,x_2)\mapsto (g_0x_2,g_0^{-1}x_1)$, and is thus generated by $n$
reflections.  This lets us identify $A^+$ with the product of those root
curves, and we find that the dual isogeny $E^n\to \hat{A}^+$ has kernel
contained in $(E^{g_0})^n$.  Thus for $G^+$ to be crystallographic, it
suffices for $(E^{g_0})^n.G^+$ to act as a reflection group on $E^n$
(noting that $x\mapsto (g_0-1)x$ induces $E/E^{g_0}\cong E$).

\begin{prop}
  Let $g\in \Aut(E)$ be a reflection, let $T\subset E$ be a finite
  subgroup, and define $G_1=T\rtimes \langle g\rangle$.  Then
  $G_n(G_1,G'_1)$ is a crystallographic reflection group iff $T$ contains
  $E^{\Lin(g)}$.
\end{prop}

\begin{proof}
  For $G_n(G_1,G'_1)$ to be crystallographic, it suffices for the induced
  action of $(E^{g_0})^n.G_n(G_1,G'_1)$ to be a reflection group.  This
  group is isomorphic to
  \[
  G_n(E^{g_0}.G_1,E^{g_0}.G'_1),
  \]
  where $g_0:=\Lin(g)$, and since the new $H_1$ does not contain any
  reflections, we may rephrase the requirement as
  \[
  (E^{g_0}.G_1)'\cong E^{g_0}.G'_1.
  \]
  Let $T^+=E^{g_0}.T$ (again a translation subgroup of $E$).  Then
  $E^{g_0}.G_1\cong T^+\rtimes \langle g\rangle$ has derived subgroup
  $(1-g_0)T^+$ which since $T^+$ contains $E^{g_0}$ gives a short
  exact sequence
  \[
  0\to E^{g_0}\to T^+\to (E^{g_0}.G_1)'\to 0.
  \]
  In particular, we see that $|(E^{g_0}.G_1)'|=|T|$, and thus we may
  further reduce the requirement to say that $|G'_1|=|T|/|E^{g_0}|$.
  But this in turn reduces to requiring that the exact sequence
  \[
  E^{g_0}\to T\to G_1'\to 0
  \]
  be short exact, and thus that $T$ contains $E^{g_0}$.
\end{proof}

Of course, given a crystallographic action of $G_n(G_1,H_1)$ on $E^n$, we
may as well quotient by $(E\cap H_1)^n$ to obtain a crystallographic action
for which $E\cap H_1=0$.  We have seen that such a group $H_1$ has a fixed
point and thus for crystallographic groups of the form $G_n(G_1,H_1)$ such
that $H_1\cap E=0$, we may assume (by conjugating by an element of the
diagonal copy of $E$ in $E^n$) that $H_1$ fixes the identity.  If this is a
trivial condition because $H_1=1$, then $\Lin(G_1)$ is cyclic, and we may
choose a preimage of a generator of $\Lin(G_1)$ in $G_1$ and assume WLOG
that the preimage fixes a point.  We thus obtain the following exhaustive
list of possibilities.

\begin{thm}
  For $n\ge 3$, any crystallographic reflection group action with
  $\Lin(G^+)$ of the form $G_n(G_1,H_1)$ is equivalent to the action of a
  group $G_n(\tilde{G}_1,\tilde{H}_1)$ on $E^n$, where
  $\tilde{G}_1=E^{G_1}\rtimes \tilde{G}_1$ and $\tilde{H}_1$ is generated
  over $\tilde{G}'_1$ by reflections.
\end{thm}

\begin{proof}
  Consider a crystallographic group of the form $G_n(\hat{G}_1,\hat{H}_1)$
  with $\hat{G}_1=G_1$.  If $G_1$ is cyclic, choose a preimage of $G_1$ in
  $\hat{G}_1$ and conjugate by a suitable element of $E\subset E^n$ so that
  it has a fixed point in $E$, and thus $\hat{G}_1=T\rtimes \langle
  g\rangle$ for some $g$-invariant $T\subset E$.  If $\hat{G}_1$ is not
  cyclic, then it has a cyclic normal subgroup contained in $\hat{G}_1'$
  such that any preimage has a unique fixed point; translating that point
  to the identity again gives $\hat{G}=T\rtimes G_1$.

  Now, since the group is crystallographic, the group
  \[
  G_n(E[6].\hat{G}_1,E[6].\hat{H}_1)
  \]
  is also a reflection group.  We then see that quotienting by
  $(E[6].\hat{G}_1)'\cap E$ gives a group $G_n(\tilde{G}_1,\tilde{H}_1)$
  where $\tilde{G}_1=E^{G_1}\rtimes G_1$, possibly acting on an isogenous
  curve $E$.  Moreover, this group is crystallographic iff $\tilde{H}_1$ is
  generated over $\tilde{G}'_1$ by reflections.
\end{proof}

\begin{rem}
  Here we note that $(E^{G_1}\times G_1)'\cap E=0$, but if $H_1$ contains
  reflections, it may also contain translations; we could, of course,
  quotient by those translations, but since our approach to computing
  invariants depends on $\tilde{G}$, it will be simpler to keep $\tilde{G}$
  fixed.  Also, note from the above proof that the isogeny with kernel
  $E^{G_1}$ is precisely the natural isogeny $E^{++}\to E$, and that $A^{++}$
  is the quotient of $E^n$ by the sum zero subgroup of $(E^{G_1})^n$.  Thus
  the analogous statement for strongly crystallographic actions is that the
  group have the form $G_n(\mu_d.E^{G_1}\rtimes G_1,\tilde{H}_1)$ where
  $\tilde{H}_1$ is generated by strong reflections over
  $(\mu_d.E^{G_1}\times G_1)'$ and contains $\mu_d$.
\end{rem}

\begin{rem}
  We have the following possibilities for $\tilde{G}$:
  \begin{itemize}
  \item[(a)] For any $E$, we may take $\tilde{G}=E[2]\times [\mu_2]$ (of
    order 8).
  \item[(b)] If $j(E)=0$, we may also take $\tilde{G}=E[\sqrt{-3}]\times
    [\mu_3]$ (order 9) or $\tilde{G}=[\mu_6]$.
  \item[(c)] If $j(E)=1728$, we may also take $\tilde{G}=E[i-1]\times [\mu_4]$
   (order 8).
  \item[(d)] If $E/\bar\F_2$ is supersingular, we may also take
    $\tilde{G}=\alpha_2\times Q_8$ or $Q_8\rtimes [\mu_3]$.
  \item[(e)] If $E/\bar\F_3$ is supersingular, we may also take
    $\tilde{G}=Z_3\rtimes [\mu_4]$.
  \end{itemize}
\end{rem}

\subsection{Classification in rank 2}

For rank 2, the reflection subgroups of $\Aut(E)^2\rtimes S_2$ can be more
complicated.  (Compare the discussion of the quaternionic case in
\cite{CohenAM:1980}.)  We may still assume that it contains $S_2$ and thus
is determined by its intersection $\Delta$ with $\Aut(E)^2$.  Let
$G_1$ denote the image of $\Delta$ under projection onto the first
coordinate, and let $H_1\normal G_1$ denote the kernel of projection onto
the second coordinate.  Then one has $G_1/H_1\cong \Delta/H_1^2\subset
(G_1/H_1)^2$ and thus there is an automorphism $\alpha$ of $G_1/H_1$ such
that $(g_1,g_2)\in \Delta$ iff $g_2H_1=\alpha(g_1H_1)$.  For this to be a
reflection group, $G_1/H_1$ must be generated by the subscheme on which
$gH_1=\alpha(gH_1)^{-1}$ (i.e., the closed subscheme corresponding to rank
2 reflections), and this suffices if $H_1$ contains a reflection or is
trivial.

Although we could in principle work with general subgroups of $\Aut(E)$ (or
even $\G_m.\Aut(E)$) as we did for $n\ge 3$, it is somewhat complicated to
enumerate the possibilities by hand, especially since $\alpha$ depends on
the choice of rank 2 reflection we conjugated to $(12)$, so one must
quotient by a nontrivial equivalence relation.  Moreover, imprimitive
reflection groups of rank 2 can have multiple systems of imprimitivity
(e.g., the diagram automorphism of $B_2$ does not respect the coordinate
basis), giving rise to nontrivial automorphisms and isomorphisms.  We will
thus instead first classify the subgroups of $\Aut_0(E)^2\rtimes S_2$ and
then determine the possible lattices and torsors.

In the complex case there are no sporadic possibilities for $\Lin(G^+)$ (we
may as well take $\alpha$ to be $g\mapsto g^{-1}$).  There is a sporadic
phenomenon, however, in that $G_2(G_1,1)$ is actually a {\em real}
reflection group: $G_2([\mu_3],1)\cong A_2$, $G_2([\mu_4],1)\cong B_2$,
$G_2([\mu_6],1)\cong G_2$.  (Relatedly, the real group $G_2([\mu_2],1)\cong
A_1A_1$ fails to be irreducible.)  Note that $G_2([\mu_4],1)\cong B_2\cong
G_2([\mu_2],[\mu_2])$ is an example of a nontrivial isomorphism between
imprimitive reflection groups.

We next turn to the classification of lattices.  Since $H_1\ne 1$, the two
coordinate copies of $E$ are root curves, and thus any lattice with the
given root curve is a quotient $E^2/T$ where $T\subset (E^{H_1})^2$ does
not meet either coordinate copy of $E$.  In particular, $T$ has the form
$(1,\alpha)T_1$ where $T_1\subset E^{H_1}$ is $G_1$-invariant and $\alpha$
is an automorphism of $T_1$ of order 2 (since $G^+$ contains $S_2$).
Moreover, for this to be crystallographic, $T$ must be generated by its
intersections with root curves of rank 2 reflections.

For $G_2([\mu_6],[\mu_6])$, this already tells us that the lattice is
unique.  For $G_2([\mu_4],[\mu_4])$, $E^H=E[1-i]$ is cyclic, so the only
lattices have $T=0$ or $T$ diagonal in $E[1-i]^2$, both of which extend to
higher rank.  Similarly, for $G_2([\mu_3],[\mu_3])$, a nonzero $T$ must be
either the diagonal copy of $E[\zeta_3]$ or the antidiagonal copy of
$E[\zeta_3]$, and only the latter is contained in a root curve.

For $G_2([\mu_4],[\mu_2])$, $T_1\subset E[2]$ must be an
$[\mu_4]$-invariant subgroup, so $T_1\in \{0,E[1-i],E[2]\}$.  Since these
form a chain, $T$ must be contained in the intersection with a single root
curve, and thus $T\in \{0,(1,1)E[1-i],(1,1)E[2],(1,i)E[2]\}$.  The latter
two cases are related to $T=0$ by an outer automorphism of
$G_2([\mu_4],[\mu_2])$ (swapping its three systems of imprimitivity), and
thus the only additional example is the diagonal $E[1-i]$, which again
extends to higher rank.

For $G_2([\mu_6],[\mu_3])$, the possible subgroups are $0$, $(1,1)E[\sqrt{-3}]$,
and $(1,-1)E[\sqrt{-3}]$, with the latter two being equivalent and
sporadic.

For $G_2([\mu_6],[\mu_2])$, $T_1$ must be a $[\mu_3]$-invariant subgroup of
$E[2]$.  In characteristic not 2, this forces $T_1\in \{0,E[2]\}$.  The
rank 2 reflections are all conjugate, so WLOG $T$ contains an element of
the form $(s,s)$ and thus also contains $([\zeta_3]s,[\zeta_3^{-1}]s)$.
(In particular, $T$ meets each of the three rank 2 root curves in a group
of order 2.)  This gives three nonzero possibilies for $T$ that are
conjugate under $G_2([\mu_6],[\mu_6])$ and thus gives a single sporadic
lattice.  (In this case, the rank two root curves have endomorphism ring
$\Z[\sqrt{-3}]$ ($j=54000$) rather than $\Z[(-1+\sqrt{-3})/2]$.)  In
characteristic 2, the curve is supersingular, so $E[2]\cong \alpha_4$, and
this has $[\mu_3]$-invariant subgroup $\alpha_2$ (the kernel of Frobenius).
Since the group acts irreducibly on $\alpha_2^2$, this does not give any
additional lattices.  Note that the action of $[\mu_3]$ on $\alpha_4$ does
not extend to an action of $S_3$, so the sporadic lattice fails in this
case!  More precisely, the degeneration of that subgroup to characteristic
2 is $\alpha_2^2$, so introduces a kernel to the map from the rank 1 root
curves.  This corresponds to the fact that every copy of $\Z[\sqrt{-3}]$ in
the relevant quaternion order extends to $\Z[\zeta_3]$.  (For $p\in
\{3,5\}$, the curve with $j=0$ has two equivalence classes of endomorphisms
squaring to $-3$, one of which does not extend.)

We next turn to torsors.  For the cases with $G_1=[\mu_6]$, the
antidiagonal copy of $[\mu_6]$ is cyclic and normal, and thus those cases
do not admit nontrivial torsors.\footnote{This is in apparent contradiction
with \cite[\S5]{GoryunovV/ManSH:2006}, which purports to give a torsor
action of (in our notation) $G_2([\mu_6],[\mu_3])$ in characteristic 0, but
their action actually {\em does} have a global fixed point; see
\cite[Rem.~2.4]{PuenteP:2023} for an explicit calculation.} In the cases
with $G_1=H_1$, both conjugacy classes of reflection are needed to generate
the group, and we can generate it with $2$ reflections, and thus there are
no nontrivial torsors on which the group acts by reflections.  Since the
cases with $H_1=1$ are real, we are left with only the case
$G_2([\mu_4],[\mu_2])$ to consider.

In that case, we have center of order $4$, and thus any torsor is trivial
in characteristic 2.  In characteristic not 2, every quotient by a subgroup
of $A^G$ is again crystallographic, so that we have two lattices to
consider.  The normal reflection subgroup $H:=G_2([\mu_2],[\mu_2])$ is
real, so has a global fixed point, and thus the remaining torsors are given
by classes in $H^1(G/H;A^H)$ such that the images of reflections lie in the
corresponding root curves.  Since there is a single conjugacy class of
reflections with nontrivial images in $G/H$, this translates to asking for
classes in the image of $H^1(G/H;E_r\cap A^H)\to H^1(G/H;A^H)$ where $E_r$
is any of the root curves.  (They form a single $H$-orbit, so $E_r\cap A^H$
is independent of the choice of root.)  For the degree 1 polarization,
$G/H$ acts as $[i]$ on $E_r\cap A^H\cong E_r[2]$, so the cohomology is
trivial, while for the degree 2 polarization, it acts trivially, and thus
there is a unique nontrivial torsor action (which extends to higher rank).

There are, in addition, two ways in which $B_2$ behaves sporadically: the
polarization of the degree 4 lattice is twice a split principal
polarization, while for the degree 2 lattice, the induced polarization on
$A^{++}$ has degree 8 (as opposed to degree 4 for $n>2$).  This sporadic
behavior for $B_2\cong G_2([\mu_4],1)$ induces sporadic behavior for the
degree 2 lattices for $G_2([\mu_4],[\mu_2])$ and $G_2([\mu_4],[\mu_4])$; in
the first case, $A^{++}\to A$ is multiplication by $1-i$, while in the
second case, $A^{++}\cong E^2$, related by a $2$-isogeny.  (These groups are
all $2$-groups, and thus automatically fix some 2-torsion point of $A^\vee$,
so that $A^{++}\to A$ cannot be an isomorphism!)  In all other cases,
$A^{++}$ is as expected from higher rank.

\medskip

We next turn to lattices in the rank 2 cases of the quaternionic cases we
considered above.  For $G_2(Q_8,[\mu_2])$, $E^H=E[2]\cong
\alpha_4$, so $T_1\in \{0,\alpha_2,\alpha_4\}$.  Since this is a chain, $T$
must be contained in some rank 2 root curve, and since these are all
conjugate, we may assume WLOG that $T$ is (anti)diagonal (i.e., lies in the
root curve corresponding to $(12)$).  The case $T_1=E[2]$ is related by an
outer automorphism of the group (which has {\em five} systems of
imprimitivity, all equivalent under outer automorphisms!) to the case
$T_1=0$, so only the case $T_1=\alpha_2$ is new, and this extends to higher
rank.

For $G_2(Q_8,[\mu_4])$ and $G_2(Q_8,Q_8)$, $E^H\cong E[1-i]\cong \alpha_2$,
and only the diagonal subgroup gives a crystallographic lattice.  This
again extends to higher rank.

For $G_2(\Aut(E),Q_8)$, $E^H\cong \alpha_2$, but the action of $\Aut(E)/H$
on $\alpha_2^2$ is irreducible, and thus there are no additional lattices.
Similarly, for $G_2(\Aut(E),\Aut(E)$, $E^H=1$ and there is nothing
sporadic.

In all of these characteristic 2 cases, the group has a normal subgroup of
order 2 (i.e., the center) and thus any equivariant torsor is trivial.

For characteristic 3, we again find that the group acts irreducibly on
$(E^H)^2$ and thus there are no additional lattices.  The groups
$G_2(\Aut(E),[\mu_3])$ and $G_2(\Aut(E),[\mu_6])$, have cyclic normal
subgroups of order 6, so no possible torsors.  But these give normal
subgroups of $G_2(\Aut(E),\Aut(E))$ that fix a unique point on every
torsor, so that $G_2(\Aut(E),\Aut(E))$ also fixes a unique point on every
torsor.

\medskip

In the quaternionic case, however, there are also two sporadic possibilities
for $\Lin(G^+)$.  Here the classification was done in \cite{CohenAM:1980},
although with duplication (the author implicitly assumed uniqueness of the
system of imprimitivity).  One obtains the following sporadic cases in
which $G\subset \Aut_0(E)$ for some $E$.  (In each case, there is a unique
equivalence class of $\alpha$, so we can simply denote the group by
$G_2(G_1,H_1)$.)  For $\ch(E)=3$, the unique example is
$G_2(\Aut_0(E),[\mu_2])$, since
\[
G_2(\Aut_0(E),1)\cong G_2([\mu_6],[\mu_2])
\]
is actually a complex reflection group.  For $\ch(E)=2$, the unique
example is $G_2(\Aut_0(E),[\mu_2])$; the other two apparent examples are
again complex, with
\[
G_2(Q_8,1)\cong G_2([\mu_4],[\mu_2])\qquad\text{and}\qquad
G_2(\Aut_0(E),1)\cong ST_{12}
\]
(In both cases, we verified the classification of \cite{CohenAM:1980} by using
Magma to enumerate the conjugacy classes of subgroups of
$G_2(\Aut_0(E),\Aut_0(E))$ and check which are generated by reflections and
remain irreducible inside $\Sp_4$.)

For the sporadic group $G_2(\Aut_0(E/\bar\F_3),[\mu_2])$, the only
nontrivial possibility for $T$ is the diagonal $E[2]$, but this is related
by an outer automorphism to the original lattice.  This group has a cyclic
normal subgroup of order 6 (in fact, it has one of order 12!), so again
there are no torsors.

For the sporadic group $G_2(\Aut_0(E/\bar\F_2),[\mu_2])$, the only
nontrivial possibilities for $T$ are the diagonal $\alpha_2$ and the
diagonal $\alpha_4$, but neither of these are invariant under the group, so
we have no non-split lattice.  Since we still have center of order 2, we
again find that there are no torsor actions.

\smallskip

In general, for rank 2 imprimitive quaternionic reflection groups, or
equivalently for groups of the form $G_2(G_1,H_1)$ with $G_1\subset
\SL_2(\R)$ finite, a necessary condition to act on an abelian variety is
that the character take values in $\Q$, or equivalently that
$\Tr(g_1)+\Tr(g_2)\in \Q$ whenever $(g_1,g_2)\in G$.  If $H_1$ has a
nontrivial element $h$, then $\sum_{0\le d<\ord(h)} \Tr(g h^d)=0$ and thus
$G_1$ must have rational character (which always works, and is accounted
for above).  Thus the only other way this can occur is if the character of
$G_1$ is defined over a real quadratic field, when $G_2(G_1,1)$ twisted by
an automorphism that conjugates the character will have rational character.
(Of course, in most cases this will not be a reflection group!)  These
cases must be considered in the full classification, but are effecively
sporadic, since the given system of imprimitivity does not correspond to
actual abelian subvarieties.  (This, of course, is another reflection of
the nonuniqueness of systems of imprimitivity.)

The additional groups that arise in this way are the following:
\begin{itemize}
  \item $G_2(\mu_{12}.\Z/2\Z,1)\cong G_2(\Aut_0(E),[\mu_2])$ (characteristic 3)
  \item $G_2(2.S_4,1)\cong G_2(\Aut_0(E),[\mu_2])$ (characteristic 2)
  \item $G_2(2.\Alt_5,1)\cong 2.\Sym_5$ (characteristic 5),
\end{itemize}
so that the first two are not actually new, and the third is primitive over
$\End(E)$ so is outside the scope of this section.

\medskip

In sum, we have the following four cases of sporadic rank 2 imprimitive
reflection group actions:
\begin{align}
  G_2([\mu_6],[\mu_2])&\qquad\text{(nonsplit degree 1, rank 2 root curves
    have order $\Z[\sqrt{-3}]$)}\notag\\
  G_2([\mu_6],[\mu_3])&\qquad\text{(degree 3)}\notag\\
  G_2(\Aut(E/\bar\F_3),[\mu_2])&\qquad\text{(split degree 1)}\notag\\
  G_2(\Aut(E/\bar\F_2),[\mu_2])&\qquad\text{(split degree 1).}\notag
\end{align}

\subsection{Invariants in rank $1$}

A large number of cases of the main theorem in rank 1 come from the
following fact.  Here the equivariant structure on ${\cal L}([0])$ is
induced from the action on functions.

\begin{prop}
  Let $G\subset \Aut_0(E)$ be a nontrivial subgroup, and let $m$ divide
  $|G|$.  Then the graded algebra
  \[
  \bigoplus_{d\ge 0} \Gamma(E;{\cal L}(dm[0]))^G \cong k[x,y],
  \]
  where $\deg(x)=1$, $\deg(y)=|G|/m$.
\end{prop}

\begin{proof}
  Let us first assume $m=1$.  We note that since $G$ contains a reflection,
  $E/G$ cannot have genus 1, and thus $E/G\cong \P^1$.  (This remains true
  over non-closed fields as the image of the identity is a marked point on
  $\P^1$.)  The pullback of $\sO_{\P^1}(1)$ is then naturally a
  $G$-equivariant line bundle on $E$.  If we express $\sO_{\P^1}(1)$ as
  ${\cal L}(p)$ where $p$ is the image of $0\in E$, then we see that the
  pullback is ${\cal L}(|G|[0])$, and thus
  \[
  \bigoplus_{d\in |G|\Z} \Gamma(E;{\cal L}(d[0]))^{G} \cong k[z,w],
  \]
  where $\deg(z)=\deg(w)=|G|$.  The action of $G$ on $\Gamma(E;{\cal
    L}([0]))=k\langle 1\rangle$ is trivial, so that there is an invariant
  $x$ of degree 1.  In particular, $x^{|G|}$ is a nonzero linear
  combination of $z$ and $w$, and thus we may extend it to a regular
  sequence $(x^{|G|},y)$ in $k[z,w]$ with $y$ an independent combination of
  $z$ and $w$.  This is then also a regular sequence in the cone over $E$,
  so that $(x,y)$ is a regular sequence in the cone, and thus the claim
  follows by degree considerations.

  The claim for $m>1$ then follows by observing that the $m$th Veronese is
  generated by $x^m$ and $y$.
\end{proof}

\begin{rem}
  In characteristic 2, the actions of $[\mu_4]$, $Q_8$, and
  $\Aut_0(E/\F_2)$ on ${\cal L}(4m[0])$ with $4m$ dividing $|G|$ have
  polynomial invariants but are not strongly crystallographic.
\end{rem}

This is enough to establish the main theorems for $G\in
\{[\mu_2],[\mu_3]\}$, as in both cases the strongly crystallographic
equivariant bundles are equivalent under conjugation by $E^G$ to powers of
${\cal L}([0])$.  Similarly, for $[\mu_4]$ in characteristic 2, any
equivariant bundle is a power of ${\cal L}([0])$, and similarly for $Q_8$,
while for $[\mu_6]$ in bad characteristic and the other two quaternionic
cases, the $p$-Sylow has polynomial invariants and the quotient by the
$p$-Sylow acts by independently rescaling $x$ and $y$.

We are thus left with considering $[\mu_4]$ and $[\mu_6]$ in characteristic
0.  For $[\mu_4]$, there is one additional equivariant bundle to consider:
${\cal L}(E[1-i])$, again with the action as functions.  The bundle has an
invariant function ($1$) and squares to an equivariant bundle isomorphic to
${\cal L}(4[0])$, so the same argument shows that it has polynomial
invariants of degrees $1,2$.

Finally, for $[\mu_6]\cong [\mu_2]\times [\mu_3]$, every strongly
crystallographic case is a power of one of four equivariant structures on
${\cal L}([0])$: $[\mu_2]$ either acts as functions or fixing the fiber at
the identity, and similarly for $[\mu_3]$.  We have already dealt with the
case that both act as on functions, and otherwise we may explicitly
identify divisors of invariant sections of degrees 2 and 3:
\[
(2[0],E[2]-[0]), (E[\sqrt{-3}]-[0],3[0]), (E[\sqrt{-3}]-[0],E[2]-[0]);
\]
since in each case the divisors are disjoint, the sections are
algebraically independent and thus generate by degree considerations.
We thus find that in those three cases the invariant ring is polynomial
with $\deg(x)=2$, $\deg(y)=3$, and again polynomiality persists
for the square and cube of the equivariant bundle.

\subsection{Invariants in rank $\ge 3$}

The key idea for computing quotients by groups of the form
$G_n(\tilde{G}_1,\tilde{H}_1)$ is that there is a short exact sequence
\[
1\to G_n(\tilde{G}_1,\tilde{H}_1)\to \tilde{G}_1^n\rtimes S_n \to
\tilde{G}_1/\tilde{H}_1\to 1,
\]
and thus $E^n/G_n(\tilde{G}_1,\tilde{H}_1)$ is an abelian cover of
$\Sym^n(E/\tilde{G}_1)$ with structure group $\tilde{G}_1/\tilde{H}_1$.
This makes it relatively straightforward to compute generators of the
invariant {\em field}, from which one can easily find generators of the
invariant ring.

Given two $n$-tuples $(x_1,\dots,x_n)$, $(y_1,\dots,y_n)$, define
polynomials $e_k(\vec{x},\vec{y})$, $0\le k\le n$ by
\[
\sum_{0\le k\le n} e_k(\vec{x},\vec{y}) t^k
=
\prod_{1\le i\le n} (x_i+ty_i).
\]
The following is then classical.

\begin{lem}
  Let $X=\Proj(k[x,y])\cong \P^1$.  Then $\Sym^n(X)\cong \P^n$, with
  homogeneous coordinates $e_k(x_1,\dots,x_n;y_1,\dots,y_n)$ for $0\le k\le
  n$.
\end{lem}

There is a variation starting with an arbitrary weighted projective line,
which also exhibits a toy version of the calculation for imprimitive groups.

\begin{lem}
  Let $X$ be a weighted projective line with degrees $d_1$, $d_2$.  Then
  $\Sym^n(X)$ is a weighted projective space with degrees $d_1$, $d_2$, and
  (with multiplicity $n-1$) $\lcm(d_1,d_2)$.
\end{lem}

\begin{proof}
  Let $x$, $y$ be the generators, with degrees $d_1$, $d_2$ respectively.
  Then the order $\lcm(d_1,d_2)$ Veronese is freely generated by
  $X:=x^{\lcm(d_1,d_2)/d_1}$ and $Y:=y^{\lcm(d_1,d_2)/d_2}$ and thus the
  invariants of $S_n$ in the Veronese are generated by $n+1$ elements of
  degree $\lcm(d_1,d_2)$.  Replacing the generating invariants
  $\prod_i X_i$ and $\prod_i Y_i$ by $\prod_i x_i$ and $\prod_i y_i$
  respectively gives a new system of parameters for the invariant ring, and
  degree considerations show that those new parameters generate as
  required.
\end{proof}

\begin{rem}
  We may view this as computing the invariant ring of
  $G_n(\mu_{\lcm(d_1,d_2)/d_1}\times \mu_{\lcm(d_1,d_2)/d_2},1)$ on $X^n$.
\end{rem}

\begin{cor}
  For any nontrivial subgroup $G_1\subset \Aut_0(E)$, let
  $\tilde{G}_1:=E^{G_1}\times G_1$.  Then
  \[
  \bigoplus_{d\in \Z} \Gamma(E^n;{\cal L}(d[E^{G_1}])^{\boxtimes
    n})^{G_n(\tilde{G},\tilde{G})} \cong k[X,Y_1,\dots,Y_n],
  \]
  where $\deg(X)=1$, $\deg(Y_1)=\cdots\deg(Y_n)=|G_1|$.
\end{cor}

\begin{proof}
  We first observe that the pullback of ${\cal L}([0])$ under the
  quotient by $E^{G_1}$ is ${\cal L}([E^{G_1}])$, and thus taking
  $(E^{G_1})^n$-invariants reduces to computing
  \[
  \bigoplus_{d\in \Z} \Gamma(E^n;{\cal L}(d[0])^{\boxtimes n})^{G_n(G_1,G_1)}.
  \]
  We again have a canonical invariant of degree $1$, while in degrees a
  multiple of $|G_1|$, we get the homogeneous coordinate ring of
  $\Sym^n(\P^1)$, and thus have independent generators $z_0,\dots,z_n$ of
  degree $|G_1|$.  Again, we may replace one of those generators by $x$ to
  obtain a sequence that by degree considerations must generate the
  invariant ring.
\end{proof}

\begin{rem}
  Note that if $x,y$ are our chosen generators for the univariate quotient,
  then the generators for the multivariate case may be taken to be
  $x_1\cdots x_n$ and $e_k(\vec{x}^{|G_1|},\vec{y})$ for $1\le k\le n$.
  In particular, we see that it is straightforward to compute these
  invariants!
\end{rem}

Of course, the quotient {\em scheme} here is still isomorphic to $\P^n$,
and the stack we get is a very simple $\mu_{|G_1|}$-cover of $\P^n$ (i.e.,
adjoining a root of one of the homogeneous coordinates).  It will thus be
convenient to replace $\tilde{G}_1$ by the even larger group
\[
\hat{G}_1:=\mu_{|G_1|}\times E^{G_1}\times G_1,
\]
so that $G_n(\hat{G}_1,\hat{G}_1)$ still acts on the cone but has invariant
ring a free polynomial ring in generators of degree $|G_1|$.  One
convenient aspect of this is that different choices of lift of
$\tilde{H}_1$ to $\hat{G}_1$ correspond to different equivariant
structures, making it easier to determine which equivariant structures give
well-behaved quotients (and in particular to verify that the theorem holds
for strongly crystallographic groups).  Conveniently, the subgroup
$\mu_{|G_1|}^{n-1}$ acts trivially (since the $\G_m^{n-1}$ already acts
trivially) and thus the invariants of $G_n(\hat{G}_1,\hat{G}'_1)$ are the
same as those of $G_n(\tilde{G}_1,\tilde{G}'_1)$, but now viewed as a
$\tilde{G}_1/\tilde{G}'_1$-cover of $\P^n$.

The key observation is that the cone over $E$ is itself a
$\hat{G}_1/\hat{G}'_1$ cover of $\P^1$, which in most cases makes it easy
to write down the requisite additional invariants.  Indeed, in good
characteristic, the group $\hat{G}_1/\hat{G}'_1$ is multiplicative, and
thus the cover is a Kummer extension as a field, and clearing the poles
of the generating $l$-th roots gives the desired invariants.

\begin{eg}
  Let $G_1=[\mu_2]$ and let $E$ be a curve not of characteristic 2.  Since
  we need to include $2$-torsion, we take $E$ of the form
  $y^2=x(x-1)(x-\lambda)$ with $\lambda\notin\{0,1,\infty\}$.  (Here, of
  course, $\lambda$ is the usual modular function for elliptic curves with
  full level $2$ structure.)  The group $\hat{G}_1/\hat{G}'_1$ is
  elementary abelian of order $16$, so (since $2$ is invertible) may be
  taken to be $\mu_2^4$.  Moreover, the action of $\mu_2^4$ on global
  sections of ${\cal L}(E[2])$ is faithful, and thus there is a basis of
  the latter on which $\mu_2^4$ acts as the obvious diagonal group.  If
  that basis is $t_1,\dots,t_4$, then the invariants are generated by
  $t_1^2,\dots,t_4^2$, which must all be linear combinations of the
  homogeneous coordinates $T$, $U$ on $\P^1$.  We find that we may make a
  change of basis so that
  \begin{align}
  x_1^2 &= U\notag\\
  x_2^2 &= T\notag\\
  x_3^2 &= T-U\notag\\
  x_4^2 &= T-\lambda U.
  \end{align}
  We then find that for $1\le j\le 4$, the products
  \[
  X_j:=\prod_i x_j^{(i)}
  \]
  are invariants for $G_n(\mu_2^4,1)$, satisfying
  \begin{align}
  X_1^2 &= \prod_i U_i\notag\\
  X_2^2 &= \prod_i T_i\notag\\
  X_3^2 &= \prod_i (T_i-U_i)\notag\\
  X_4^2 &= \prod_i (T_i-\lambda U_i).
  \end{align}
  For $n\ge 3$, the elements
  $X_1^2,\dots,X_4^2$ are linearly independent and thus we may extend them
  to a basis on $\Sym^n(\P^1)$.  This remains a system of parameters if we
  replace them by $X_1,\dots,X_4$, and thus the $G_n(\mu_2^4,1)$-invariants
  are a polynomial ring with degrees $1^4 2^{n-3}$ (i.e., four generators
  of degree 1 and $n-3$ of degree 2).  More generally, for any subset
  $S\subset \{1,\dots,4\}$, the elements $X_i$ for $i\notin S$ remain
  invariant under $G_n(\mu_2^4,\mu_2^S)$, and again we get polynomial
  invariants, with degrees $1^{4-|S|} 2^{n-3+|S|}$ (valid as long as $n\ge
  |S|-3$).
\end{eg}

\begin{rem}
  The four ``coordinate'' copies of $\mu_2$ inside $\mu_2^4$ are (strong
  reflection) preimages of the four reflections inside $E[2]\rtimes
  [\mu_2]$, and thus this shows that the main theorem holds for the
  infinite families of Coxeter groups other than $A_n$.  (See also
  \cite{KacVG/PetersonDH:1984} for the characteristic 0 case).
\end{rem}

\begin{eg}
  For the curve of $j$-invariant $0$ in characteristic not 3, the group
  associated to $[\mu_3]$ may be identified (over $k[\zeta_3]$) with
  $\mu_3^3$, giving a model for $E$ as
  \[
  x_1^3=T,\ x_2^3=U,\ x_3^3=T-U.
  \]
  We find that the invariants for $G_n(\mu_3^3,1)$ are generated over
  $\Sym^n(\P^1)$ by products satisfying
  \begin{align}
  X_1^3 &= \prod_i U_i\notag\\
  X_2^3 &= \prod_i T_i\notag\\
  X_3^3 &= \prod_i (T_i-U_i),
  \end{align}  
  and the right-hand sides are linearly independent so long as $n\ge 2$.
  We thus obtain a polynomial ring with degrees $1^3 3^{n-2}$ or, for any
  subset $S\subset \{1,\dots,3\}$ a polynomial ring with degrees $1^{3-|S|}
  3^{n-2+|S|}$.  Again, the three coordinate copies of $\mu_3$ lie over the
  six reflections in $\tilde{G}_1$.
\end{eg}

\begin{eg}
  For the curve of $j$-invariant 1728 in characteristic not 2, the group
  associated to $[\mu_4]$ may be identified with $\mu_4^2\times \mu_2$,
  giving a model for $E$ as
  \[
  x_1^4=T,\ x_2^4=U,\ y^2 = T-U,
  \]
  (where now $\deg(y)=2$).  We again find that for $n\ge 2$ and any
  sequence $d_1|4$, $d_2|4$, $d_3|2$, the invariants of $G_n(\mu_4^2\times
  \mu_2,\mu_{d_1}\times \mu_{d_2}\times \mu_{d_3})$ are a polynomial ring
  with degrees $d_1,d_2,2d_3,4^{n-2}$.
\end{eg}

\begin{eg}
  For the curve of $j$-invariant $0$ in characteristic not 2 or 3, the
  group associated to $[\mu_6]$ may be identified with $\mu_2\times
  \mu_3\times \mu_6$ with model
  \[
  z^2=T,\ y^3=U,\ x^6=T-U,
  \]
  and for $d_2|2$, $d_3|3$, $d_6|6$ and $n\ge 2$ we get a quotient with
  polynomial invariants of degrees $d_6$, $2d_3$, $3d_2$ and $6^{n-2}$.
\end{eg}

This settles the infinite families in the complex case, at least in large
characteristic.  Before considering the quaternionic cases, we should clean
up the complex case by dealing with the missing finite characteristic
cases.

\begin{eg}
  Consider the characteristic 2 version of a $[\mu_2]$ case.  The lift to
  characteristic 0 is generated over the $\P^n$ of
  $G_n(E[2]\times[\mu_2],E[2]\times[\mu_2])$-invariants by $0\le m\le 4$
  invariants of degree 1, and one can choose a basis so that each invariant
  vanishes precisely on those points in $E^n$ such that one of the
  coordinates is in a corresponding coset of $E[2]$ in $E[4]$.  The basis
  elements all have the same (up to scalars) reduction mod 2, but suitable
  linear combinations will have linearly independent reductions.  On the
  other hand, the ideal in $E^n$ over characteristic 0 cut out by those $m$
  invariants (i.e., that each of the $m$ cosets appears as one of the
  coordinates) also has a limit (the ideal of the fiber of the Zariski
  closure) and since every degree 1 element of the ideal is invariant, this
  persists in the limiting ideal.  We thus see that the degree 1 invariants
  in characteristic 2 cut out the Zariski closure, i.e., the condition that
  the corresponding divisor in $E/E[2]\cong E$ contains the given subscheme
  of $E[4]/E[2]\cong E[2]$.  This is a codimension $m$ linear condition on
  $\P^n$ (i.e., that the corresponding polynomial have a given factor), and
  thus we may choose $n+1-m$ additional degree 2 invariants to make the
  overall intersection trivial as required.  We thus see that we again have
  degrees $1^m2^{n+1-m}$.
\end{eg}

\begin{rem}
  Note that we have {\em fewer} cases in characteristic 2, as the cases
  correspond to choices of subschemes of $E[2]$.  (We readily verify that
  each subscheme of $E[2]$ in characteristic 2 is the limit of a subcheme
  in characteristic 0.)  For an ordinary curve, $E[2]$ consists of two
  double points, and thus we have a total of 9 subschemes.  In terms of the
  group $H_1$, there is a natural splitting $G_1\cong (\mu_2\times
  \Z/2\Z)^2$ and $H_1$ corresponds to choosing one of $1$, $\Z/2\Z$ or
  $\mu_2\times \Z/2\Z$ for each of the two factors.  For a supersingular
  curve, $E[2]$ is a quadruple point and there are $5$ choices of
  subscheme, corresponding to the groups
  \[
  1\subset [\mu_2]\subset \alpha_2\times [\mu_2]\subset \alpha_4\times
  [\mu_2]\subset \mu_2\times \alpha_4\times [\mu_2].
  \]
  In either case, we still have at least one subscheme of each size, so we
  get the same set of possible degrees of invariants.
\end{rem}

\begin{eg}
  The same argument applies to the $\mu_3$ cases in characteristic 3,
  replacing $E[4]/E[2]$ by $E[3]/E[\sqrt{-3}]$, and thus
  for $H_1$ in the list
  \[
  1\subset [\mu_3]\subset \alpha_3\times [\mu_3]\subset \mu_3\times
  \alpha_3\times [\mu_3],
  \]
  the quotient is a polynomial ring with degrees $1^m 3^{n+1-m}$,
  where $|H_1|=3^{3-m}$.
\end{eg}

\begin{eg}
  For $\mu_4$, the bad characteristic is 2.  Here there is a slight
  additional complication in that $G_n(E[1-i]\times [\mu_4],1)$ also has a
  degree 2 invariant in characteristic 0.  This is not a problem if both
  degree 1 invariants appear, because it is then still true that the degree
  2 elements of the relevant ideal are all invariant (since the degree 1
  invariants are then all sections of the line bundle).  The main
  consequence is that when lifting $H_1$ to characteristic 2, we must
  choose the lift in such a way as to first eliminate the degree 2
  invariant.  In particular, the case $H_1=[\mu_2]$ lifts to the {\em
    torsor} action, not the action fixing the identity!  This can be done
  for every subgroup of $\alpha_2\times [\mu_4]$ other than $[\mu_4]$,
  showing that the main theorem applies in those cases.  For the remaining
  case, we observe that in characteristic 0, the quotient of $E^n$ by
  $G_n(E[1-i]\times [\mu_4],[\mu_4])$ may be viewed as the quotient of
  $(E/[\mu_4])^n$ by $G_n(E[1-i],1)$, which puts enough constraints on the
  ideal to show that its limit is still generated by invariants as
  required.
\end{eg}

\begin{eg}
  For $\mu_6$ in characteristic 3, we first note that the quotient of $E^n$
  by $G_n(\mu_6\times [\mu_6],[\mu_3])$ may be expressed as the quotient of
  $(E/[\mu_3])^n$ by $G_n(\mu_6\times \mu_2,1)$, and thus the same argument
  we used in characteristic 0 shows that we obtain a weighted projective
  space with degrees $d_1,3d_2,6,\dots,6$.  Similarly, the usual argument
  shows that $G_n(\mu_6\times [\mu_6],\mu_2\times [\mu_2])$ continues to
  have polynomial invariants (of degrees $2,2,6,\dots,6$) in characteristic
  3.  If $\delta_1$ is the (unique up to scalars) section of degree 1
  vanishing when one of the coordinates is the identity and $\delta_3$ is
  the section of degree 3 vanishing when one of the coordinates is a
  nontrivial $2$-torsion point, then $\delta_1$ and $\delta_3$ are
  eigeninvariants for $G_n(\mu_6\times [\mu_6],\mu_2\times[\mu_2])$.
  Moreover, by considering the behavior near associated points of the
  respective divisors, we find that any nontrivial eigensection is a
  multiple of $\delta_1$, $\delta_3$, or $\delta_1\delta_3$ as appropriate.
  This tells us that the invariants for $G_n(\mu_6\times [\mu_6],1)$ are
  flat and it remains only to observe that $\delta_1^2$ and $\delta_3^2$
  are independent invariants to see that for $H\in \{1,\mu_2,[\mu_2]\}$ the
  invariant ring remains polynomial.
\end{eg}

\begin{rem}
  Note that again only one of the two lifts of $[\mu_3]$ to characteristic
  0 gives flat invariants.
\end{rem}

\begin{eg}
  For the reduction of $\mu_6$ to characteristic 2, the argument is
  analogous; $G_n(\mu_6\times[\mu_6],[\mu_2])$ involves a multiplicative
  group acting on $(\P^{[1,2]})^n$, $G_n(\mu_6\times [\mu_6],\mu_3\times
  [\mu_3])$ has polynomial invariants by the usual argument, and the ring
  of $G_n(\mu_6\times [\mu_6],1)$-invariants adjoins cube roots of
  independent invariants (to give $\delta_1$ vanishing when one of the
  coordinates is the identity and $\delta_2$ vanishing when one of the
  coordinates is a nontrivial $\sqrt{-3}$-torsion point).
\end{eg}

The reduction to multiplicative actions on weighted projective lines also
works for two of the quaternionic cases.

\begin{eg}
  Let $E$ be a supersingular curve of characteristic 3 and consider the
  quotient of $E^n$ by $G_n(\mu_{12}\times \Aut(E),H_1)$ with
  $\Aut(E)'\subset H_1$.  We may equivalently view this as the quotient by
  $G_n(\mu_{12}\times Z_4,H_1/\Aut(E)')$ of $(E/\Aut(E)')^n$,
  or equivalently as the quotient by $G_n(\mu_{12}\times \mu_4,\hat{H}_1)$
  of $(\P^{[1,3]})^n$, which by the usual argument gives polynomial
  invariants of degrees $d_1,3d_2,12,\dots,12$.
\end{eg}

\begin{eg}
  Let $E$ be a supersingular curve of characteristic 2 and consider the
  quotient of $E^n$ by $G_n(\mu_{24}\times \Aut(E),H_1)$ with
  $\Aut(E)'\subset H_1$.  We may equivalently view this as the quotient by
  $G_n(\mu_{24}\times Z_3,H_1/\Aut(E)')$ of $(E/\Aut(E)')^n$,
  or equivalently as the quotient by $G_n(\mu_{24}\times \mu_3,\hat{H}_1)$
  of $(\P^{[1,8]})^n$, which by the usual argument gives polynomial
  invariants of degrees $d_1,8d_2,24,\dots,24$.
\end{eg}

It remains only to consider the characteristic 2 cases with
$G_1=\mu_8\times \alpha_2\times Q_8$.  Here the situation is somewhat
trickier, in that the quotient by $G_1'\cong [\mu_2]$ is (unweighted)
$\P^1$ but the residual action is no longer multiplicative.  We can still
deal with this using the following fact (which also gives direct arguments
for some of the cases above).

\begin{prop}
  Let $G_1\subset \G_a$ be a finite subgroup scheme of order $p^l$ in
  characteristic $p$ and let it act on $\P^{[1,d]}$ by $x\mapsto x+t w^d$
  where $x$ is the generator of degree $d$.  Then $G_n(G_1,1)$ has
  polynomial invariants on $(\P^{[1,d]})^n$, with degrees $1,d,p^l
  d,\dots,p^l d$.
\end{prop}

\begin{proof}
  A finite subgroup scheme of $\G_a$ is cut out by a ``linear'' equation of
  the form $q(t):=\sum_{0\le i\le l} c_l t^{p^l}=0$ (with $c_l=1$), and one
  finds that the $G_1$-invariants in $\P^{[1,d]}$ are generated by $w$ and
  $w^{p^l d}q(x/w^d)$.  The function $\sum_i (x_i/w_i^d)$ on
  $(\P^{[1,d]})^n$ is invariant under the sum zero subscheme of $\G_a$, so
  certainly under the sum zero subscheme of $G_1^n$, and thus we have
  invariants $\prod_i w_i$ of degree 1 and $(\prod_i w_i^d)\sum_i
  (x_i/w_i^d)$ of degree $d$.  When these invariants vanish, the
  corresponding divisor in $(\P^{[1,d]}/\mu_{p^ld}\times G_1)\cong
  \P^{[p^ld,p^ld]}$ contains the point $w=0$ with multiplicity 2, and thus
  we may choose $n-1$ further invariants of degree $p^l d$ to give a
  generating set of the appropriate degree.
\end{proof}

\begin{rem}
  There is an important technical point here: there are (nonreduced)
  subgroup schemes $G_1\subset \Aut(\P^1)$ for which the quotient by
  $G_n(G_1,1)$ does not have the degree one expects (e.g., the kernel of
  the Frobenius homomorphism $\PGL_2\to \PGL_2$).  In general,
  semicontinuity ensures that the quotient by a group scheme always has
  {\em at least} the degree one expects, with equality guaranteed in the
  case of a reduced group scheme.  However, when the quotient $\P^1/G_1$
  has the expected degree, the quotient by $G_n(G_1,1)$ will also have the
  expected degree, as we have
  \[
  \deg((\P^1)^n)/|G_n(G_1,1)|\le
  \deg((\P^1)^n/G_n(G_1,1))\le
  |G_1|\deg((\P^1/G_1)^n/S_n)
  \]
  and the two bounds agree.
\end{rem}

\begin{rem}
  In addition to the multiplicative and additive cases, there is one more
  type of abelian subgroup of $\PGL_2$, namely the Heisenberg group of an
  abelian variety with degree 2 polarization.  This is a priori trickier to
  deal with since the subgroup of $\GL_2$ has nontrivial derived subgroup
  $\mu_2$, so that one is directly led to consider a group acting on a
  conic rather than on $\P^1$ itself (and one must assume $n>1$ since the
  coordinate ring of a conic is not polynomial).  However, for the cases
  coming from elliptic curves, we have already shown that the quotient has
  polynomial invariants, as we may identify it with the quotient by
  $G_n(E[2]\times [\mu_2],[\mu_2])$.  The remaining case $\alpha_2^2$ is
  ill-behaved, as the degree 2 Frobenius morphism $\P^1\to \P^1$ factors
  through the quotient by $\alpha_2^2$.  The invariants are still
  polynomial in that case (at least for $n>2$), but the degree is larger
  than expected and every $G_n(\mu_2.\alpha_2^2,1)$-invariant is actually
  $G_n(\mu_2.\alpha_2^2,\mu_2.\alpha_2^2)$-invariant.  We omit the details.
\end{rem}  

\begin{rem}
  One can use this idea to iteratively find explicit invariants in
  the cases above.  The idea is to choose an additive subgroup (e.g.,
  $[\mu_2]$ in characteristic 2), use the Proposition to determine the
  additional invariant required for that sum-zero subgroup, and then
  modify it by adding additional (invariant) terms to make it invariant
  under the rest of the group.  The result will then be a cover of
  $(\P^{[1,pd]})^n$ by an element invariant under the full group, allowing
  us to induct to a larger group with additive quotient.  The resulting
  expressions can be fairly complicated, however.
\end{rem}

\begin{eg}
  If $G_1=\alpha_2\times Q_8$ and $H_1$ is one of
  \[
  [\mu_2],[\mu_4],\alpha_2\times [\mu_2],[Q_8],\alpha_2\times
  [\mu_4],\alpha_2\times [\mu_8],
  \]
  then $E/H_1$ is a weighted projective line with degrees $1$ and $|H_1|/2$
  and the residual action of $G_1/H_1$ is as a subgroup of $\G_a$,
  and thus the proposition applies to give invariants of degrees
  $1,|H_1|/2,8,\dots,8$.  It is then clear that we will continue to have
  polynomial invariants if we include a factor of $\mu_d$ where $d$ divides
  $|H_1|/2$.
\end{eg}

\subsection{Invariants in rank 2}

From the classification of imprimitive reflection groups in rank 2, one
immediate conclusion is that every such group contains rank 1 reflections.
It follows that the quotient by $H^2$ is isomorphic to $\P^1\times \P^1$
(possibly with nontrivial weighting), and thus we may reduce to considering
reflection groups inside $\Aut(\P^1)^2\rtimes S_2$, or more precisely those
groups generated by {\em rank 2} reflections.  The preimage of this group
inside $\SL_2^2\rtimes S_2$ acts naturally on $\Gamma(\P^1\times
\P^1;\sO(1)\boxtimes \sO(1))$.  The corresponding embedding $\P^1\times
\P^1\to \P^3$ identifies $\P^1\times \P^1$ with a smooth quadric, and thus
we see that every element of $\SL_2^2\rtimes S_2$ preserves this quadric.
It is, in fact, easy to see that it preserves the equation on the nose (it
suffices to check that $S_2$ does so) and thus this gives a (well-known)
map
\[
\SL_2^2\rtimes S_2\to \GO_4
\]
(with kernel the diagonal copy of $\mu_2$ inside $\SL_2^2$).  We may then
(at least in good characteristic) compute the invariant ring on $\P^1\times
\P^1$ by computing the invariants of the corresponding subgroup of $\GO_4$
and then quotienting by the equation of the quadric.

\begin{lem}
  An element of $\SL_2^2\rtimes S_2$ is a rank 2 reflection iff its image
  in $\GO_4$ is a reflection.
\end{lem}

\begin{proof}
  The rank 2 reflections are all conjugate to $(12)$, so it suffices to
  observe that $(12)$ maps to a reflection in $\GO_4$.
\end{proof}

\begin{rem}
  To be precise, in odd characteristic, there are two maps $\SL_2^2\rtimes
  S_2\to \GO_4$ as we can multiply by the map $S_2\cong Z(\GO_4)\cong \mu_2$,
  and the claim fails for the other such map.
\end{rem}

There is also a map $\GO_4\to \PSL_2^2\rtimes S_2$ coming from the fact that
the kernel of $\SL_2^2\rtimes S_2\to \PSL_2^2\rtimes S_2$ contains the
kernel of the map to $\GO_4$.

\begin{lem}
  The preimage of a rank 2 reflection in $\PSL_2^2\rtimes S_2$ in $\GO_4$
  contains a unique reflection.
\end{lem}
  
\begin{proof}
  A rank 2 reflection has the form $(x,y)\mapsto (gy,g^{-1}x)$ and thus any
  lift $\tilde{g}\in \SL_2$ induces a rank 2 reflection in $\SL_2^2\rtimes
  S_2$ and thus a reflection in $\GO_4$.  But the preimage in $\GO_4$ is a
  coset of $\mu_2$ and the negative of a reflection is not a reflection,
  and thus the lift to a reflection is unique.
\end{proof}

\begin{prop}
  Let $G\subset \PSL_2^2\rtimes S_2$ be a subgroup generated by rank 2
  reflections, defined over a field in which $|G|\ne 0$.  Then there is a
  $G$-equivariant line bundle on $\P^1\times \P^1$ such that
  the corresponding invariant ring is a free polynomial ring.
\end{prop}

\begin{proof}
  Since $G$ is generated by rank 2 reflections, it is the image of the
  reflection group in $\GO_4$ generated by the preimages of its rank 2
  reflections.  Any reflection group in good characteristic has polynomial
  invariants, with at least one generator of degree 2, and thus the
  quotient by the invariant quadric is again a free polynomial ring.
\end{proof}

\begin{rem}
  This result can fail in bad characteristic, even when the group lifts to
  characteristic 0!  For instance, there is a reflection group of the form
  $G_2(\SL_2(\F_5),1)$ for which the invariant ring over $\F_5$ has Hilbert
  series $(1+t^{15})/(1-t)(1-t^6)(1-t^{20})$ (or
  $(1+t^{16})/(1-t^2)(1-t^6)(1-t^{20})$ for $G_2(\SL_2(\F_5),\mu_2)$.)
\end{rem}

\begin{rem}
  There is a canonical preimage in which every rank 2 reflection in $G$
  lifts to a reflection in $\GO_4$, but there may be other preimages in
  which only a conjugacy-invariant generating subset of rank 2 reflections
  lift.
\end{rem}

\begin{cor}
  Let $G\subset \PSL_2\rtimes S_2$ be any subgroup generated by
  reflections, defined over a field in which $|G|\ne 0$.  Then there is a
  $G$-equivariant line bundle on $\P^1\times \P^1$ such that the
  corresponding invariant ring is a free polynomial ring.
\end{cor}

\begin{proof}
  Let $H\subset \PSL_2$ be the group corresponding to the rank 1
  reflections in $G$.  Then $\P^1/H$ is isomorphic to $\P^1$ and thus
  $G/H^2$ acts on $\P^1\times \P^1$, and is generated by rank 2
  reflections.
\end{proof}

Note that this gives an alternate approach to the classification of
imprimitive quaternionic reflection groups of rank 2.  Indeed, the above
maps in the complex case correspond to maps
\[
U(\H)^2\rtimes S_2\to \GO_4\to \PGU(\H)^2\rtimes S_2
\]
of compact groups, and thus there is an injection from (a) quaternionic
reflection groups that are imprimitive of rank 2 and generated by rank 2
reflections to (b) Coxeter groups of rank at most 4.  (To be precise, every
Coxeter group pulls back to a reflection group in $U(\H)^2\rtimes S_2$, but
of course some of those will not be {\em quaternionic} reflection groups!)
More generally, any irreducible reflection subgroup of $U(\H)^2\rtimes S_2$
that contains $-1$ has a normal subgroup generated by rank 2 reflections,
and thus the corresponding subgroup of $\GO_4$ at least {\em normalizes} a
Coxeter group.

It thus remains only to determine the invariants in bad characteristic.
The nonsporadic cases follow from the calculation for $n\ge 3$.

\begin{eg}
  Consider the sporadic lattice for $G_2([\mu_6],[\mu_2])$ in odd
  characteristic.  This corresponds to a group of the form
  $G_2(E[2].[\mu_6],[\mu_2])$ on $E^2$ and thus (taking the line bundle to
  be $\sO(2[0])$) a group of the form $G_2(\mu_2.E[2].[\mu_3],1)\subset
  \SL_2^2\rtimes S_2$.  The image in $\GO_4$ may be identified with the
  permutation representation of $S_4$, and thus even in characteristic 3,
  we have invariants of degrees $1$, $3$, $4$ forming a regular sequence
  with the natural quadric.  It follows that the quotient of $E^2$ is a
  weighted projective space with degrees $1$, $3$, $4$.
\end{eg}
  
\begin{eg}
  Consider the sporadic lattice for $G_2([\mu_6],[\mu_3])$.  This
  corresponds to a group of the form $G_2(E[\sqrt{-3}].[\mu_6],[\mu_3])$ on
  $E^2$ and thus $G_2(E[\sqrt{-3}].[\mu_2],1)$ on $(\P^1)^2$.  In
  characteristic not 3, this corresponds to the Coxeter group $A_1A_2$
  resulting in invariants of degrees $1,2,3$.  In characteristic 3, the
  ``Coxeter'' group is a nonreduced group scheme, but the invariants are
  still straightforward to compute.  (Note that $\mu_2.\alpha_3.[\mu_2]$
  embeds uniquely in $\SL_2$, so we do not need to do computations on $E$.)
\end{eg}
  
\begin{eg}
  Let $E/\bar{F}_3$ be supersingular, and consider $G_2(\Aut_0(E),[\mu_2])$.
  The quotient $E/[\mu_2]$ is a weighted projective space with degrees
  $1,2$, and $\Aut_0(E)/[\mu_2]$ has generators acting on the coordinates by
  $(x,y)\mapsto (x,-y)$ and $(x,y)\mapsto (x,y+x^2)$.  A straightforward
  computation (given that $\Aut_0(E)/[\mu_2]$ is embedded diagonally in
  $G_2(\Aut(E)/[\mu_2],1)$) gives invariants of degrees $1,4,6$.  (We
  could, of course, use the line bundle ${\cal L}(2[0])$ instead, giving
  invariants of degrees $1,2,3$.)
\end{eg}

\begin{eg}
  Let $E/\bar{F}_2$ be supersingular, and consider $G_2(\Aut_0(E),[\mu_2])$.
  The quotient by $[\mu_2]$ again has weights $1,2$ and the group acts by
  $(x,y)\mapsto (x,\zeta_3 y)$ (embedding antidiagonally in the imprimitive
  group) and $(x,y)\mapsto (x,y+x^2)$ (embedding diagonally in the
  imprimitive group) with invariants of degrees $1,6,8$.  (Relative to
  ${\cal L}(2[0])$, the invariants have degrees $1,3,4$, and the image
  in $\GO_4$ is the permutation representation of $S_4$.)
\end{eg}

\section{Classification of crystallographic actions of primitive groups}
\label{sec:classify}

\subsection{Real reflection groups}

In characteristic 0, these were classified in \cite{SaitoK:1985}, and
in general characteristic were essentially classified in \cite{elldaha},
but it is straightforward enough to rederive this from the above
considerations.  (In particular, the condition in \cite{elldaha} was based
on simple roots, and thus although the actions given there are certainly
crystallographic, it is not entirely obvious, though true, that those are
the only crystallographic actions.)  As mentioned, there are no
crystallographic torsor actions (since Coxeter groups of rank $n$ are
generated by $n$ reflections), and thus the only question is to determine
the lattices.  For simply laced groups, there is a single conjugacy class of
involutions and thus a unique choice of crystallographic lattice, namely
$\Lambda\otimes E$ where $\Lambda$ is the root lattice.  We are thus left
with (since we may ignore imprimitive cases) $G_2$ and $F_4$.

For $G_2$, any crystallographic lattice for $G_2$ is a quotient of the
lattice $\Lambda_{A_2}\otimes E$ for the reflection subgroup $A_2$ by a
$G$-invariant subgroup of $(\Lambda_{A_2}\otimes E)^{A_2}\cong E[3]$.  The
quotient by $E[3]$ again has the form $\Lambda_{A_2}\otimes E$ relative to
the other conjugacy class of roots (which is related by an outer
automorphism), and thus the only other lattices are quotients by subgroups
$T\subset E[3]$ of order 3, which are again crystallographic (e.g., by the
explicit description as an action on $E\times E/T$ given in
\cite{elldaha}).  We thus have one family of lattices (over the moduli
stack of elliptic curves) with polarization of degree 3 and one (over the
moduli stack of $3$-isogenies) with principal polarization.

Similarly, for $F_4$, any crystallographic lattice is a quotient of
$\Lambda_{D_4}\otimes E$ by a $F_4$-invariant subgroup of
$(\Lambda_{D_4}\otimes E)^{D_4}\cong E[2]^2$.  Since $F_4$ acts on $E[2]^2$
as $F_4/D_4\cong \SL_2(\Z/2\Z)$, we see that the invariant subgroups have
the form $T^2$ for $T\subset E[2]$.  Again, each invariant subgroup gives a
crystallographic lattice, and the the quotient by $E[2]^2$ is related by an
outer automorphism.  We thus find one family of lattices (over the moduli
stack of elliptic curves) with polarization of degree 4 and one (over the
moduli stack of $2$-isogenies) with principal polarization.

It remains to determine the strongly crystallographic equivariant bundles.
It turns out that in all of the primitive cases above, the fixed subschemes
of the reflections are integral, implying that $A^{++}=A$ and that there is
a unique strongly crystallographic bundle; more precisely, there is a
unique invariant bundle representing the minimal polarization, and the
strongly crystallographic equivariant structure is the one that acts
trivially on the fiber over the identity.  The action on the fiber at the
identity is not always obvious from our constructions below, but in each
case we obtain polynomial invariants, for which being strongly
crystallographic is a necessary condition in characteristic 0 \cite[\S
  0.2]{BernsteinJ/SchwarzmanO:2006}.  (And, of course, triviality of the
action on the fiber over the identity survives reduction to finite
characteristic.)

The corresponding cases of the main theorems were shown in characteristic 0
by \cite{BernsteinJ/SchwarzmanO:2006,KacVG/PetersonDH:1984}, and we will
show below that this remains true in all finite characteristics.  We
summarize the results in the table below.  The first column gives the
group, the second gives the degree of the minimal polarization, the third
gives the degrees of the invariants relative to the minimal polarization,
and the fourth indicates the subsection in which the result is shown.  The
last column describes the anticanonical class on the quotient in
characteristic prime to the orders of the reflections.  Since abelian
varieties have trivial canonical bundle, this is the same as the
ramification divisor of the quotient map, and thus is the linear
combination of reflection hypersurfaces in which the coefficient of a
reflection hypersurface is one less than the order of the stabilizer of the
generic point.  The sum over any given orbit of hypersurfaces is a
$G$-invariant divisor, so represents a multiple of the minimal invariant
polarization; a factor $[d]^m$ in that column (here $m=1$ always)
represents an orbit such that the sum is $d$ times the minimal polarization
and the components appear with multiplicity $m$.  (In fact, this
decomposition continues to hold in characteristic 2, as the different
orbits remain distinct on reduction.)  There is also an explicit formula
for the anticanonical class on a weighted projective space, namely
$\omega\cong \sO(\sum_i d_i)$, which gives a sanity check on the degrees as
stated; see also Proposition \ref{prop:canclass} below.  (In each case, the
degrees of the generating invariants are given by the standard labels of an
associated affine Dynkin diagram; for a conceptual explanation for this
fact, see the computation of the Hilbert series in
\cite{LooijengaE:1976,SaitoK:1990}.)

\begin{table}[H]
  \begin{center}
  \begin{tabular}{|c|c|l|c|l|}
    \hline
    $\Lin(G)$ & $\deg$ &$d_1,\dots,d_{n+1}$&\S&$-K$\\
    \hline
    $A_n$ &$n$ & $1,\dots,1$&\ref{ss:linsys} &$[n+1]$\\
    $E_6$ &$3$ & $1,1,1,2,2,2,3$&\ref{ss:dps} &$[12]$\\
    $E_7$ &$2$ & $1,1,2,2,2,3,3,4$&\ref{ss:dps} &$[18]$\\
    $E_8$ &$1$ & $1,2,2,3,3,4,4,5,6$&\ref{ss:dps} &$[30]$\\
    $F_4$ &$4$ & $1,1,2,2,3$&\ref{ss:surf_auts} &$[3][6]$\\
    $F_4$ &$1$ & $1,2,2,3,4$&\ref{ss:surf_auts} &$[6][6]$\\
    $G_2$ &$3$ & $1,1,2$&\ref{ss:linsys} &$[1][3]$\\
    $G_2$ &$1$ & $1,2,3$&\ref{ss:linsys} &$[3][3]$\\
    \hline
  \end{tabular}
\end{center}\end{table}

\subsection{Complex reflection groups}

Here, again, these have been classified in characteristic 0
(\cite{PopovVL:1982}), but a priori there could be (and are!) additional
cases in finite characteristic, which requires us to redo the
classification.  The possible linear groups were classified by Shephard and
Todd in \cite{ShephardGC/ToddJA:1954}, and we denote the $k$-th group in
their classification by $ST_k$.  A priori, $4\le k\le 37$, but we can
eliminate any case for which the center of the $\Z$-span of the matrices is
not an imaginary quadratic order.  We are thus left with 12 cases, namely
$k\in \{4,5,8,12,24,25,26,29,31,32,33,34\}$.  (In each case, the center is
a maximal order, so there are no residual choices there.)

We first consider the possibilities of lattices.  After eliminating cases
with a single orbit of root curves, we are left with considering $ST_5$ and
$ST_{26}$.

For $ST_5$ (over $\Z[\zeta_3]$), there are two orbits of root curves, and
the corresponding reflection subgroups are related by an outer
automorphism.  We thus start with the unique crystallographic lattice for
$ST_4$ and quotient by a subgroup of the group of $ST_4$-invariants in
$E^2$.  That fixed subgroup is $E[2]$, on which $ST_5$ acts by $[\zeta_3]$.
In odd characteristic, the only invariant subgroups are $0$ and $E[2]$, and
the quotient by $E[2]$ gives a lattice related by an outer automorphism, so
does not give a new case.  However, in characteristic 2, the kernel of
Frobenius is an invariant subgroup $\alpha_2$, so we get a new lattice with
polarization of degree 3.

For $ST_{26}$ (again over $\Z[\zeta_3]$), neither conjugacy class of root
subgroups generates the group alone, so we may start with the
crystallographic lattice for $ST_{25}$ (with polarization of degree 9),
with (cyclic) fixed subgroup $E[\sqrt{-3}]$.  This has precisely two
subgroups, both invariant, and the quotient by $E[\sqrt{-3}]$ gives an
additional lattice of degree 3.

Having classified the lattices, we turn to torsors.  Again, once we
eliminate cases that can be generated by $n$ reflections, there are only
two cases to consider, namely $ST_{12}$ and $ST_{31}$ (both of which luckily
have unique lattices).

For $ST_{12}$, there are no fixed points on the crystallographic lattice
and thus any torsor action must have the same underlying abelian variety.
Since the center contains $[\mu_2]$, the torsor actions are classified by
$H^1(G/[\mu_2];A[2])$.  This is trivial (as a set) in characteristic 2 and
has a unique class otherwise, which may be verified to be crystallographic.
(Indeed, it is the restriction of the natural theta torsor.)

Similarly, for $ST_{31}$, the underlying variety is the same, and the
actions are classified by the cohomology group $H^1(G/[\mu_4];A[1-i])$.
Again, this is trivial in characteristic 2 and has a unique class
otherwise, which again is the restriction of the theta torsor.

With the exception of $ST_8$, all reflections again have integral fixed
subscheme, so that there is a unique (when the group is pointed) or at most
one (for torsor actions) strongly crystallographic bundle, which we may
again identify in characteristic 0 via polynomiality of its invariants.
Since $ST_8$ is at least {\em generated} by such reflections, the only
choice in that case is whether the strongly crystallographic contains
$\mu_2$, and again we find that both cases work.

This argument fails for the sporadic lattice for $ST_5$, since that does
not lift to characteristic 0.  In particular, there are 9 equivariant
structures on the minimal invariant line bundle, only one of which is
strongly crystallographic.  However, for any choice which does not fix the
fiber over the identity, the invariants of degree prime to 3 have a
nontrivial common divisor, namely the sum of the fixed subschemes of the
reflections with nontrivial character.  However, since the bundle we
consider below has polynomial invariants of degrees $1$, $3$, and $4$, the
intersection of the prime-to-3 invariants has codimension 2, and thus the
bundle is indeed the unique strongly crystallographic bundle.

In each case, we again find (which is now new even in characteristic 0)
that the main theorems hold in arbitrary characteristic, with results
summarized in the table below.  An asterisk in the degree
column denotes actions without global fixed points, and for $ST_8$ we have
two rows to reflect the fact that there are multiple choices of strongly
crystallographic line bundle.
\begin{table}[H]\begin{center}
  \begin{tabular}{|c|c|l|c|l|}
    \hline
    $\Lin(G)$ & $\deg$ &$d_1,\dots,d_{n+1}$&\S & $-K$\\
    \hline
    $ST_4$ & $6$ & $1,1,2$ & \ref{ss:complexsurf0} & $[2]^2$\\
    $ST_5$ & $6$ & $1,2,3$ & \ref{ss:complexsurf0} & $[1]^2[2]^2$\\
    $ST_5$ & $3$ & $1,3,4$ & \ref{ss:complexsurf0} & $[2]^2[2]^2$\\
    $ST_8$ & $2$ & $2,3,4$ & \ref{ss:complexsurf1728} & $[3]^3$\\
    $ST_8$ & $2$ & $2,4,6$ & \ref{ss:complexsurf1728} & $[3]^3[3]$\\
    $ST_{12}$ & $1$ & $2,4,6$ & \ref{ss:complexsurf8000} & $[6]$\\
    $ST_{12}$ & $1\rlap{${}^*$}$ & $1,3,8$ & \ref{ss:ST12_normal} & $[6]$\\
    $ST_{24}$ & $1$ & $1,2,4,7$ & \ref{ss:ST24} & $[14]$\\
    $ST_{25}$ & $9$ & $1,2,2,3$ & \ref{ss:complexsurf0} & $[4]^2$\\
    $ST_{26}$ & $9$ & $1,2,3,4$ & \ref{ss:complexsurf0} & $[2][4]^2$\\
    $ST_{26}$ & $3$ & $1,3,4,6$ & \ref{ss:complexsurf0} & $[4]^2[6]$\\
    $ST_{29}$ & $1$ & $1,2,4,5,8$ & \ref{ss:ST31} & $[20]$\\
    $ST_{31}$ & $1$ & $2,4,6,8,10$ & \ref{ss:complexsurf1728} & $[30]$\\
    $ST_{31}$ & $1\rlap{${}^*$}$ & $1,4,5,8,12$ & \ref{ss:ST31} & $[30]$ \\
    $ST_{32}$ & $9$ & $2,3,4,5,6$ & \ref{ss:complexsurf0} & $[10]^2$ \\
    $ST_{33}$ & $2$ & $1,1,3,3,4,6$ & \ref{ss:ST33} & $[18]$\\
    $ST_{34}$ & $1$ & $1,3,4,6,7,9,12$ & \ref{ss:ST34} & $[42]$\\
    \hline
  \end{tabular}
\end{center}\end{table}
Note that the section reference for $ST_{29}$ does not prove the theorem in
characteristic 2, for which see Subsection \ref{ss:ST29_2}.  Also, for the
non-torsor cases of $ST_{12}$ and $ST_{31}$, the given section only
discusses characteristic not 2 (but does discuss the torsor case for
$ST_{31}$ in characteristic 2).

In characteristic 0, the rank 2 cases other than $ST_{12}$ were shown in
\cite{ShvartsmanOV:1979,TokunagaS/YoshidaM:1982}), while the torsor case of
$ST_{12}$ was shown in \cite{KoziarzV/RitoC/RoulleauX:2021}, and $ST_{24}$
was dealt with in \cite{MarkushevichD/MoreauA:2022b}.  The non-torsor
$ST_{12}$ case and the remaining higher rank cases are apparently new even
in characteristic 0.

The stated decomposition of the anticanonical class remains valid in finite
characteristics (including the sporadic, but tamely ramified, $ST_5$ case),
with the one exception being the second $ST_8$ case, for which the
decomposition in characteristic 2 is $[3]^4$.  (This is tricky to compute
directly, since it involves wild ramification, but flatness tells us that
the canonical divisor of the reduction is the reduction of the canonical
divisor.)  Also, the torsor and non-torsor cases for $ST_{12}$ and $ST_{31}$
have the same reductions to characteristic 2, and the invariants of the
reduction are as given in the torsor case.

\subsection{Quaternionic reflection groups}

For quaternionic groups, the analogue of the Shephard-Todd classification
was given in \cite{CohenAM:1980}.  There are a couple of technical issues
(beyond the fact that the list in \cite{CohenAM:1980} has duplications and
includes some complex cases), with the biggest being that the notions of
``irreducible'' and ``primitive'' there are over $\C$ rather than over
$\Q$.  Luckily, we can sanity check the list by comparing to the list of
maximally finite quaternionic groups given in \cite{NebeG:1998} (which
gives a complete classification up to rank 10 for quaternion algebras with
center $\Q$).  (Also, that list is available as a database in Magma, making
it particularly convenient to use.)  In particular, for each such group, it
is straightforward to check whether the reflection subgroup remains
irreducible and quaternionic (i.e., whether it has the same span over
$\Q$).  We can then do a similar check for each subgroup to obtain the
following lists for the different choices of ramified prime, with no prime
$>5$ appearing.

In $\GL_n(Q_5)$, the only possibility is a central extension $2.S_5\subset
\GL_2(Q_5)$.  (This appears in \cite{CohenAM:1980} as an imprimitive group,
but is primitive for our purposes, since the systems of imprimitivity are
only defined over a {\em real} quadratic field.)

In $\GL_n(Q_3)$, the only possibilities are $2.\!\Alt_5\subset
2.\!\Alt_6\subset \GL_2(Q_3)$ and $\pm\PGU_3(\F_3)\subset \GL_3(Q_3)$.

In $\GL_n(Q_2)$, the maximal possibilities are $\SL_2(\F_3)\otimes
A_2\subset \GL_2(Q_2)$, $2_-^{1+4}.\!\Alt_5\subset \GL_2(Q_2)$,
$2_-^{1+6}.\Omega^-_6(\F_2)\subset \GL_4(Q_2)$, and $\pm\SU_5(\F_2)\subset
\GL_5(Q_2)$.  We also have primitive reflection subgroups
$2_-^{1+4}.S_3,2_-^{1+4}.\!\Dih_5\subset 2_-^{1+4}.\!\Alt_5$ and
$2_-^{1+6}.3^{1+2}.Z_2,2_-^{1+6}.G_3(Z_3,1)\subset
2_-^{1+6}.\Omega^-_6(\F_2)$.

Again, in most of the cases there is a single orbit of roots with maximal
associated order, so that the only cases that might have multiple lattices
are $\SL_2(\F_3)\otimes A_2$ and $2_-^{1+4}.S_3$.

For $\SL_2(\F_3)\otimes A_2$, the two classes of order 2 reflections are
related by an outer automorphism, and the associated orders are maximal,
so we obtain a quotient of the $A_2$ lattice by an $\SL_2(\F_3)$-invariant
subgroup of $E[3]$.  Since $\SL_2(\F_3)$ acts faithfully on $E[3]$, the
only other possibility is the quotient by all of $E[3]$, which is related
by an outer automorphism to the original lattice.

For $2_-^{1+4}.S_3$, there are two orbits of roots, and neither associated
order is maximal.  Luckily, the orders (consisting of Hurwitz quaternions
with integer coefficients) are contained in unique maximal orders, and thus
embed uniquely in $\End(E)$.  Thus despite the lack of maximality, we still
have no additional freedom in the initial lattice.  Since the subgroup
$ST_8$ generated by the unique conjugacy class of order 4 reflections fixes
only the identity, there is again a unique lattice.

We next turn to torsors.  Since all of the above groups contain $[-1]$ and
have no fixed points in $A$, the torsors (ignoring the crystallographic
condition) are classified by $H^1(G/[\mu_2];A[2])$.  This is trivial in
characteristic 2 and the crystallographic condition rules out $2.\!\Alt_5$
(which can be generated by two reflections).  For the remaining three cases
$2.S_5$, $2.\!\Alt_6$, $\pm\PGU_3(\F_3)$, there is a unique class in
$H^1(G/[\mu_2];A[2])$ which can be verified to be the theta class, so
crystallographic.

We again have $A^{++}=A$ for every primitive quaternionic case, so there is a
unique minimal invariant bundle and the strongly crystallographic
equivariant structure is the one with eigenvalue 1 at the identity.  For
the characteristic 2 groups $2_-^{1+4}.\!\Alt_5$,
$2_-^{1+6}.\Omega^-_4(\F_2)$ and $\pm \SU_5(\F_2)$, the equivariant
structure is unique (so strongly crystallographic) since $G/G'$ is a
$2$-group, and thus this remains true for the restrictions to
crystallographic subgroups.  (Since $2_-^{1+4}.\!\Alt_5$ has reflections of
order 4, the square of the minimal bundle is also strongly
crystallographic, and this remains true for the crystallographic subgroups
containing such reflections.)  For the remaining cases, we may argue as in
the sporadic $ST_5$ case: the invariants of degree prime to 2 or 3 have
intersection of codimension $>1$ and thus the group must act trivially on the
fiber over the identity.

We summarize the structures of the invariant rings below, noting
the five cases with dashes denoting that the invariant ring is not
actually polynomial.  We omit the information about the anticanonical
class as in most cases we have a wildly ramified map to a wild
stack, making it tricky to compute the multiplicities in cases with
multiple orbits of roots.
\begin{table}[H]\begin{center}
  \begin{tabular}{|c|c|l|c|}
    \hline
    $\Lin(G)$ & $\deg$ &$d_1,\dots,d_{n+1}$&\S\\
    \hline
    $2.S_5$ & $1$ & $4,6,10$ & \ref{ss:ST12_normal}\\
    $2.S_5$ & $1\rlap{${}^*$}$ & \text{---} & \ref{ss:ST12_normal}\\
    \hline
    $2.\!\Alt_5$ & $3$ & $1,4,5$ & \ref{ss:alt6}\\
    $2.\!\Alt_6$ & $3$ & $4,6,10$ & \ref{ss:quatsurf}\\
    $2.\!\Alt_6$ & $3\rlap{${}^*$}$ & \text{---} & \ref{ss:alt6}\\
    $\pm\PGU_3(\F_3)$ & $1$ & $2,8,14,18$ & \ref{ss:ST24}\\
    $\pm\PGU_3(\F_3)$ & $1\rlap{${}^*$}$ & \text{---} & \ref{ss:ST24}\\
    \hline
    $\SL_2(\F_3)\otimes A_2$ & $3$ &$1,3,8$ & \ref{ss:linsys}\\
    $2_-^{1+4}.S_3$ & $2$ & $2,3,8$ & \ref{ss:alt5}\\
    $2_-^{1+4}.S_3$ & $2$ & $2,6,8$ & \ref{ss:alt5}\\
    $2_-^{1+4}.\!\Dih_5$ & $2$ & $2,5,8$ & \ref{ss:alt5}\\
    $2_-^{1+4}.\!\Dih_5$ & $2$ & $2,8,10$ & \ref{ss:alt5}\\
    $2_-^{1+4}.\!\Alt_5$ & $2$ & $5,8,12$ & \ref{ss:alt5}\\
    $2_-^{1+4}.\!\Alt_5$ & $2$ & $8,10,12$ & \ref{ss:quatsurf}\\
    $2_-^{1+4}.\!\Alt_5$ & $2$ & $8,12,20$ & \ref{ss:quatsurf}\\
    $2_-^{1+6}.3^{1+2}.Z_2$ & $2$ & \text{---} & \ref{ss:badomega6}\\
    $2_-^{1+6}.G_4(Z_3,1)$ & $1$ & $1,4,6,9,16$ & \ref{ss:omega6}\\
    $2_-^{1+6}.\Omega^-_6(\F_2)$ & $1$ & $1,9,16,24,40$ & \ref{ss:omega6}\\
    $\pm\SU_5(\F_2)$ & $1$ & \text{---} & \ref{ss:SU5}\\
    \hline
  \end{tabular}
\end{center}\end{table}
Note that the line bundle giving $8,12,20$ for $2_-^{1+4}.\!\Alt_5$ is not
strongly crystallographic.  Also, in the two tamely ramified cases $2.S_5$
and $\pm\PGU_3(\F_3)$, we find that the canonical classes have structure
$[20]$ and $[42]$ respectively.  Note that there is an argument given in
Subsection \ref{ss:quatsurf} that reduces the $\pm\PGU_3(\F_3)$ case to the
characteristic 3 case of $ST_{34}$, but although the {\em reduction} is
conceptual, the direct arguments for both cases are computational.

\subsection{Tables for imprimitive cases}

For convenience, we give the analogous tables for imprimitive groups, as
discussed in Section \ref{sec:imprim} above.  We omit the reference column
and in the quaternionic cases omit the decomposition of the ramification
divisor.

For the real cases, we have the following.  Note that a given line may
represent multiple cases with nonisomorphic line bundles but the same
invariants.  Also, for $n=2$, the $D_2$ case should be omitted (being
reducible), as should the $C_2$ case with degree 4 polarization (as that is
twice a principal polarization), and the decomposition of the canonical
divisor for the degree 2 case should be replaced as stated in the last row.
\begin{table}[H]\begin{center}
  \begin{tabular}{|c|c|l|l|}
    \hline
    $\Lin(G)$ & $\deg$ &$d_1,\dots,d_{n+1}$&$-K$\\
    \hline
    $D_n$ & $4$ & $1,1,1,1,2,\dots,2$ & $[2n-2]$\\
    $C_n$ & $4$ & $1,1,1,2,\dots,2$ & $[1][2n-2]$\\
    $C_n$ & $2$ & $1,1,2,\dots,2$ & $[1][1][2n-2]$\\
    $C_n$ & $1$ & $1,2,\dots,2$ & $[1][1][1][2n-2]$\\
    $C_n$ & $1$ & $2,\dots,2$ & $[1][1][1][1][2n-2]$\\
    \hline
    $C_2$ & $2$ & $1,1,2$ & $[1][1][1][1]$\\
    \hline
  \end{tabular}
\end{center}\end{table}

We next turn to the complex cases.  Again many of the lines correspond to
multiple choices of equivariant bundles.
\begin{table}[H]\begin{center}
  \begin{tabular}{|c|c|l|l|}
    \hline
    $\Lin(G)$ & $\deg$ &$d_1,\dots,d_{n+1}$&$-K$\\
    \hline
    $G_n([\mu_3],1)$ & $3$ & $1,1,1,3,\dots,3$ & $[3n-3]$\\
    $G_n([\mu_3],[\mu_3])$ & $3$ & $1,1,3,\dots,3$ & $[1]^2 [3n-3]$\\
    $G_n([\mu_3],[\mu_3])$ & $1$ & $1,3,\dots,3$ & $[1]^2[1]^2[3n-3]$\\
    $G_n([\mu_3],[\mu_3])$ & $1$ & $3,\dots,3$ & $[1]^2[1]^2[1]^2[3n-3]$\\
    \hline
    $G_n([\mu_4],1)$ & $2$ & $1,1,2,4,\dots,4$ & $[4n-4]$\\
    $G_n([\mu_4],[\mu_2])$ & $2$ & $1,2,2,4,\dots,4$ & $[1][4n-4]$\\
    $G_n([\mu_4],[\mu_2])$ & $2$ & $2,2,2,4,\dots,4$ & $[1][1][4n-4]$\\
    $G_n([\mu_4],[\mu_2])$ & $2\rlap{${}^*$}$ & $1,1,4,\dots,4$ & $[2][4n-4]$\\
    $G_n([\mu_4],[\mu_2])$ & $1$ & $1,2,4,\dots,4$ & $[1][2][4n-4]$\\
    $G_n([\mu_4],[\mu_2])$ & $1$ & $2,2,4,\dots,4$ & $[1][1][2][4n-4]$\\
    $G_n([\mu_4],[\mu_4])$ & $2$ & $1,2,4,\dots,4$ & $[1]^3[4n-4]$\\
    $G_n([\mu_4],[\mu_4])$ & $2$ & $2,2,4,\dots,4$ & $[1][1]^3[4n-4]$\\
    $G_n([\mu_4],[\mu_4])$ & $1$ & $1,4,\dots,4$ & $[1]^3[2][4n-4]$\\
    $G_n([\mu_4],[\mu_4])$ & $1$ & $2,4,\dots,4$ & $[1][1]^3[2][4n-4]$\\
    $G_n([\mu_4],[\mu_4])$ & $1$ & $4,\dots,4$ & $[1]^3[1]^3[2][4n-4]$\\
    \hline
    $G_2([\mu_4],[\mu_2])$ & $2$ & $1,2,2$ & $[1][1][1][1][1]$\\
    $G_2([\mu_4],[\mu_2])$ & $2$ & $2,2,2$ & $[1][1][1][1][1][1]$\\
    $G_2([\mu_4],[\mu_2])$ & $2\rlap{${}^*$}$ & $1,1,4$ & $[1][1][1][1][2]$\\
    $G_2([\mu_4],[\mu_4])$ & $2$ & $1,2,4$ & $[1]^3[2][2]$\\
    $G_2([\mu_4],[\mu_4])$ & $2$ & $2,2,4$ & $[1]^3[1][2][2]$\\
    \hline
    $G_n([\mu_6],1)$ & $1$ & $1,2,3,6,\dots,6$ & $[6n-6]$\\
    $G_n([\mu_6],[\mu_2])$ & $1$ & $1,2,6,\dots,6$ & $[3][6n-6]$\\
    $G_n([\mu_6],[\mu_2])$ & $1$ & $2,2,3,6,\dots,6$ & $[1][6n-6]$\\
    $G_n([\mu_6],[\mu_2])$ & $1$ & $2,2,6,\dots,6$ & $[1][3][6n-6]$\\
    $G_n([\mu_6],[\mu_3])$ & $1$ & $1,3,6,\dots,6$ & $[2]^2[6n-6]$\\
    $G_n([\mu_6],[\mu_3])$ & $1$ & $2,3,3,6,\dots,6$ & $[1]^2[6n-6]$\\
    $G_n([\mu_6],[\mu_3])$ & $1$ & $3,3,6,\dots,6$ & $[1]^2[2]^2[6n-6]$\\
    $G_n([\mu_6],[\mu_6])$ & $1$ & $1,6,\dots,6$ & $[2]^2[3][6n-6]$\\
    $G_n([\mu_6],[\mu_6])$ & $1$ & $2,3,6,\dots,6$ & $[1]^2[3][6n-6]$\\
    $G_n([\mu_6],[\mu_6])$ & $1$ & $2,3,6,\dots,6$ & $[1][2]^2[6n-6]$\\
    $G_n([\mu_6],[\mu_6])$ & $1$ & $2,3,6,\dots,6$ & $[1]^5[6n-6]$\\
    $G_n([\mu_6],[\mu_6])$ & $1$ & $2,6,\dots,6$ & $[1][2]^2[3][6n-6]$\\
    $G_n([\mu_6],[\mu_6])$ & $1$ & $2,6,\dots,6$ & $[1]^5[3][6n-6]$\\
    $G_n([\mu_6],[\mu_6])$ & $1$ & $3,6,\dots,6$ & $[1]^2[2]^2[3][6n-6]$\\
    $G_n([\mu_6],[\mu_6])$ & $1$ & $3,6,\dots,6$ & $[1]^5[2]^2[6n-6]$\\
    $G_n([\mu_6],[\mu_6])$ & $1$ & $6,\dots,6$ & $[1]^5[2]^2[3][6n-6]$\\
    \hline
    $G_2([\mu_6],[\mu_2])$ & $1$ & $1,3,4$ & $[2][6]$\\
    $G_2([\mu_6],[\mu_3])$ & $3$ & $1,2,3$ & $[1][1]^2[3]$\\
    \hline
  \end{tabular}
\end{center}\end{table}
The last two entries are the sporadic lattices.  There are mild changes in
bad characteristic (apart from changes in the structure of the ramification
divisor): the second $G_n([\mu_4],[\mu_2])$ case fails in characteristic 2
(but the line bundle ceases to be strongly crystallographic) while the
first case does not appear (being subsumed in the non-torsor case).  In
addition, the last $G_n([\mu_4],[\mu_4])$ case has the same invariants, but
is no longer strongly crystallographic.  The $2,2,3,6,\dots,6$ case and the
$2,3,6,\dots 6$ case with ramification $[1][2]^2[6n-6]$ do not appear in
characteristic 2, while the $2,3,3,6,\dots,6$ case and the $2,3,6,\dots 6$
case with ramification $[1]^2[3][6n-6]$ do not appear in characteristic 3.

For the imprimitive quaternionic cases in characteristic 3, we have the
following, omitting the cases arising as Veronese subalgebras.
\begin{table}[H]\begin{center}
  \begin{tabular}{|c|c|l|}
    \hline
    $\Lin(G)$ & $\deg$ &$d_1,\dots,d_{n+1}$\\
    \hline
    $G_n(\Aut_0(E/\bar\F_3),[\mu_3])$ & $1$ & $1,3,12,\dots,12$ \\
    $G_n(\Aut_0(E/\bar\F_3),[\mu_6])$ & $1$ & $1,6,12,\dots,12$ \\
    $G_n(\Aut_0(E/\bar\F_3),[\mu_6])$ & $1$ & $2,3,12,\dots,12$ \\
    $G_n(\Aut_0(E/\bar\F_3),\Aut(E/\F_3))$ & $1$ & $1,12,\dots,12$ \\
    $G_n(\Aut_0(E/\bar\F_3),\Aut(E/\F_3))$ & $1$ & $3,4,12,\dots,12$ \\
    \hline
    $G_2(\Aut_0(E/\bar\F_3),[\mu_2])$ & $1$ & $1,4,6$ \\
    \hline
  \end{tabular}
\end{center}\end{table}

Finally, for the imprimitive quaternionic groups in characteristic 2, we
have the following, again omitting Veronese cases.
\begin{table}[H]\begin{center}
  \begin{tabular}{|c|c|l|}
    \hline
    $\Lin(G)$ & $\deg$ &$d_1,\dots,d_{n+1}$\\
    \hline
    $G_n(Q_8,[\mu_2])$ & $2$ & $1,1,8,\dots,8$\\
    $G_n(Q_8,[\mu_2])$ & $1$ & $1,2,8,\dots,8$\\
    $G_n(Q_8,[\mu_4])$ & $2$ & $1,2,8,\dots,8$\\
    $G_n(Q_8,[\mu_4])$ & $1$ & $1,4,8,\dots,8$\\
    $G_n(Q_8,Q_8)$ & $2$ & $1,4,8,\dots,8$\\
    $G_n(Q_8,Q_8)$ & $1$ & $1,8,\dots,8$\\
    \hline
    $G_n(\Aut_0(E/\bar\F_2),Q_8)$ & $1$ & $1,8,24,\dots,24$\\
    $G_n(\Aut_0(E/\bar\F_2),\Aut_0(E/\bar\F_2))$ & $1$ & $1,24,\dots,24$\\
    $G_n(\Aut_0(E/\bar\F_2),\Aut_0(E/\bar\F_2))$ & $1$ & $3,8,24,\dots,24$\\
    \hline
    $G_2(\Aut_0(E/\bar\F_2),[\mu_2])$ & $1$ & $1,6,8$\\
    \hline
  \end{tabular}
\end{center}\end{table}

\section{Geometric constructions}

\subsection{Linear systems}
\label{ss:linsys}

One approach for proving primitive special cases of the main theorem is to
exhibit the quotient as a moduli space.  This comes in two flavors, with
the simpler corresponding to the $W(A_n)$ case.  For $A_n$, there is a
unique lattice, $\Lambda_{A_n}\otimes E$, and the $S_{n+1}$-equivariant
exact sequence
\[
0\to \Lambda_{A_n}\to \Z^{n+1}\to \Z\to 0
\]
produces an exact sequence
\[
0\to \Lambda_{A_n}\otimes E\to E^{n+1}\to E\to 0,
\]
exhibiting $\Lambda_{A_n}\otimes E$ as the sum-zero subvariety of
$E^{n+1}$.  In particular, we may view points of the abelian variety as
$(n+1)$-tuples of points of $E$ and points of the quotient as unordered
$(n+1)$-tuples, or equivalently as divisors of degree $n+1$.  From the
latter perspective, the sum-zero condition is nothing other than the
condition that the divisor be linearly equivalent to $(n+1)[0]$, and thus
$\Lambda_{A_n}\otimes E/S_{n+1}$ is naturally isomorphic to the
corresponding linear system, i.e., the projective space of sections of
${\cal L}((n+1)[0])$.   We thus find that $\Lambda_{A_n}\otimes
E/S_{n+1}\cong \P^n$.

One consequence is that it is straightforward, given a sum-zero point of
$E^{n+1}$, to compute its image in the quotient $\P^n$: find the unique
(up to scalars) function with divisor $\sum_{0\le i\le n} [x_i]-(n+1)[0]$
and expand in some fixed basis of ${\cal L}((n+1)[0])$, e.g., monomials in
$x$ and $y$ of degree $<2$ in $y$.  (We will make use of this below in the
more computational approaches.)  Note that we may replace $(n+1)[0]$ by any
other linearly equivalent divisor without changing the resulting map.

In addition, there is a residual action of $E[n+1]\rtimes \Aut(E)$ on the
quotient, acting (as it must to be $S_{n+1}$-equivariant) diagonally on the
coordinates in $E^{n+1}$, and it is straightforward to describe the
corresponding action in terms of $\P^n$.  Indeed, $E[n+1]\rtimes \Aut(E)$
acts projectively on ${\cal L}((n+1)[0])$, with $\Aut(E)$ acting in the
obvious way, while the action of $E[n+1]$ is twisted by Weil functions
(functions with divisor $(n+1)[t]-(n+1)[0]$ for $t\in E[n+1]$).  The action
on the quotient $\P^n$ is then nothing other than the dual representation
(i.e., replacing $\rho(g)$ by $\rho(g)^{-t}$).

One immediate consequence is that we can prove the main theorem for $G_2$
as well, using the fact that $W(A_2)$ is normal in $W(G_2)$.  (A similar
approach lets one deduce the $B_3$ cases from $A_3\cong D_3$, but of course
we have already dealt with those above.)  Here, as discussed above, there
are essentially two lattices to consider: in addition to the lattice
$\Lambda_{A_2}\otimes E$, we may also quotient by the diagonal action of
the kernel of a $3$-isogeny.

For the original lattice, we are quotienting by $[\mu_2]$, which acts on
${\cal L}(3[0])=\langle 1,x,y\rangle$ in the usual way, from which we
readily see that the dual representation has invariants of degrees $1,1,2$.
For the $3$-isogeny case, if $K$ is the kernel of the $3$-isogeny, then we
may replace ${\cal L}(3[0])$ by ${\cal L}(K)$ (as usual, we are viewing the
subgroup scheme $K\subset E$ as a subscheme and thus as a divisor), and
thus have a natural linear action of $K\rtimes [\mu_2]$.  When $K$ is
\'etale, this is geometrically nothing other than the permutation action of
$S_3$, so has dual of the same form and invariants of degrees $1,2,3$.
Similarly, when $K\cong \mu_3$ in characteristic 3, it splits ${\cal
  L}(K)={\cal L}(3[0])$ into three distinct eigenspaces with the nontrivial
eigenspaces swapped by $[-1]$, and again the invariants have degrees
$1,2,3$.  The remaining case is $K\cong \alpha_3$, which may be computed
directly.  Indeed, the action of $\alpha_3$ on an appropriate basis is
$(u_1,u_2,u_3)\mapsto (u_1+\epsilon u_2-\epsilon^2 u_3,u_2+\epsilon
u_3,u_3)$ with $[-1]$ acting by $u_2\mapsto -u_2$, and the invariants
$u_3$, $u_2^2+u_1u_3$ and $u_1^3$ are clearly algebraically independent.
  
There are several other examples of rank 2 reflection groups in which
$W(A_2)$ is normal.  The idea, just as in the $G_2$ case, is to extend by
the diagonal action of a subgroup of $\Aut_0(E)$, with the caveat being that
this is not always a reflection group.  The reflections in $W(A_2)\times
\Aut_0(E)$ come in precisely two flavors: conjugates of reflections in
$W(G_2)$, and elements of the form $[\zeta_3^{\pm 1}] (123)^{\pm 1}$.
We thus see that the subgroup of $\Aut_0(E)$ must either be generated by
elements of order 3 or by such elements along with $[-1]$.  We thus find
the following cases: $1$, $[\mu_2]$, $[\mu_3]$, $[\mu_6]$, and $\SL_2(\F_3)$.

The $[\mu_3]$, $[\mu_6]$ cases correspond to imprimitive complex reflection
groups $G_2([\mu_3],[\mu_3])$ and $G_2([\mu_6],[\mu_3])$, with the system
of imprimitivity given by the two eigenspaces of $(123)\in S_3$, and we
recover the previous calculation.  The final case (corresponding to $E$ a
supersingular curve of characteristic 2) can be computed explicitly from
the action of $\SL_2(\F_3)$ on $y^2+y=x^3$, and a further explicit calculation
shows that the invariant ring has degrees 1,3,8.

\subsection{Anticanonical Del Pezzo surfaces}
\label{ss:dps}

The other geometric approach, though only applicable in sporadic cases,
covers a large number of cases, most directly the Weyl groups of type $E$.
It is well-known that the configuration of $-1$-curves on a del Pezzo
surface (e.g., the 27 lines on a cubic surface) have combinatorics related
to such Weyl groups.  For instance, a del Pezzo surface $X$ of degree $1$
has 240 $-1$-curves, which may be identified with the 240 roots of a root
system of type $E_8$.  Indeed, if $e$ is such a $-1$-curve, then $e+K_X$
has self-intersection $-2$, and the lattice spanned by such elements (the
orthogonal complement of $K_X$) only differs from the $E_8$ lattice by
negating the inner product.  To see this, note that $X$ is, at least
geometrically, expressible as a blowup of $\P^2$ in a sequence of 8 points,
and thus $\Pic(X)$ has a basis consisting of $h$ (the pullback of
$\sO_{\P^2}(1)$) and the consecutive exceptional curves $e_i$.  The
expression as an iterated blowup (a ``blowdown structure'', following
\cite{HarbourneB:1997}) is not unique, and thus neither is the given basis.

In that basis, $K_X=-3h+e_1+\cdots+e_8$ and $K_X^\perp$ is easily seen to
be a negative definite even lattice of determinant 1.  Indeed, $K_X^\perp$
has basis $h-e_1-e_2-e_3$, $e_1-e_2$,\dots, $e_7-e_8$, which are easily
seen (after negating the intersection form) to be the simple roots of a
root system of type $E_8$.  It thus has $240$ elements of self-intersection
$-2$, and subtracting $K_X$ from any such root gives a $-1$-curve.

To relate this picture to our problem, we note that the anticanonical
divisor $-K_X$ is effective and (apart from a few cases in characteristic 2
and 3) represented by a pencil of (generically smooth) genus 1 curves.  If
we choose such a curve $C$, then there is an induced homomorphism
$\rho_X:\Pic(X)\to \Pic(C)$ such that $\deg(\rho_X(D))=-D\cdot K_X$.  In
particular, $\rho_X(-K_X)\in \Pic^1(C)$ determines a marked point of $C$
allowing us to view it as an elliptic curve $E$, and, given a choice of
blowdown structure, $\rho_X|_{K_X^\perp}$ gives a point of
$\Hom(\Lambda_{E_8},E)\cong E\otimes \Lambda_{E_8}$.  Moreover, we can
recover $X$ as an iterated blowup from $\rho_X$ in the following way.  The
divisor class $\rho(h)$ has degree 3, so embeds $C$ in $\P^2$ (more
precisely, in a Brauer-Severi variety geometrically isomorphic to $\P^2$),
while the divisor classes $\rho(e_i)$ have degree 1, so give a sequence of
$8$ points of $C$.  After blowing up the first $d$ of those points, we
still have an embedding of $C$ as an anticanonical curve, and thus
$\rho(e_{d+1})$ specifies a point of the surface, allowing us to blow it
up.

We thus obtain a natural family of smooth projective rational surfaces over
$\Lambda_{E_8}\otimes E$, which is $W(E_8)$-invariant in a certain weak
sense.  For $1\le i\le 7$, the reflection in $e_i-e_{i+1}$ swaps
$\rho(e_i)$ and $\rho(e_{i+1})$.  If the points are the same, this does
nothing, while when the points are different, the surfaces agree since
blowups in distinct points commute.  For the remaining simple root, if
$\rho(e_1)$,\dots,$\rho(e_3)$ are not collinear, then we can obtain a new
description as an iterated blowup by applying a quadratic transformation.
This corresponds to a reflection in $h-e_1-e_2-e_3$, and the
non-collinearity condition is nothing other than the requirement that
$\rho(h-e_1-e_2-e_3)\ne 0$, so that the quadratic transformation is defined
on the complement of the reflection hyperplane.  We thus see in general
that any two fibers of the family lying over the same $W(E_8)$-orbit are
isomorphic.  This does not, however, extend to an action of $W(E_8)$ on the
total space of the family: e.g., when $\rho(h-e_1-e_2-e_3)=0$, the divisor
class $h-e_1-e_2-e_3$ is effective, and thus the corresponding reflection
does not preserve the effective cone!

A related issue is that the surfaces thus obtained are not in general del
Pezzo surfaces: the anticanonical bundle is nef, but fails to be ample
whenever $\rho$ maps some root to the identity (making the corresponding
positive root effective).  This gives rise to an additional family of
surfaces over $\Lambda_{E_8}\otimes E$ by taking images under the
anticanonical map
\[
X\to \Proj\bigoplus_{0\le i} \Gamma(X;\omega_X^{-i}).
\]
The fibers of the new family have ample anticanonical bundle, but may be
singular.  This new family {\em does} have an action of $W(E_8)$ (see the
discussion in \cite{fildefs}), and thus descends to a family over
the quotient $\Lambda_{E_8}\otimes E/W(E_8)$ we want to understand.

The point of the above discussion is that there is an alternate description
of this family that lets us directly identify it with a weighted projective
space.  Since $X$ comes with an embedding of $E$ as an anticanonical curve,
there is (up to scalars) a corresponding section $t\in
\Gamma(X;\omega_X^{-1})$ vanishing on $E$.  Taking the quotient of the
graded algebra by $t-1$ gives a natural {\em filtered} algebra, and the
associated graded of the filtered algebra is the quotient by $t$.  By
cohomological considerations (again see \cite{fildefs}), this associated
graded is nothing other than the homogeneous coordinate ring of $E$ itself.
In particular, if we lift the generators of the coordinate ring of $E$ to
the filtered algebra, they will still generate, and the relation(s) they
satisfy will differ by the addition of lower order terms.  Since the
homogeneous coordinate ring of $E$ has generators $w,x,y$ of degrees
$1,2,3$ and a single relation of the form
\[
y^2+a_{10}wxy+a_{30}w^3y=x^3+a_{20}w^2x^2+a_{40}w^4x+a_{60}w^6,
\]
the filtered algebra also has a single relation, of the form
\begin{align}
y^2+(a_{10}w+a_{11})xy+{}&
(a_{30}w^3+a_{31}w^2+a_{32}w+a_{33})y=\notag\\
x^3&{}+(a_{20}w^2+a_{21}w+a_{22})x^2+(a_{40}w^4+a_{41}w^3+a_{42}w^2+a_{43}w+a_{44})x\notag\\
&{}+(a_{60}w^6+a_{61}w^5+a_{62}w^4+a_{63}w^3+a_{64}w^2+a_{65}w+a_{66}).
\end{align}
Now, there is also a unipotent group acting by adding lower-order terms to
the generators (i.e., by changing the choice of lifted generators):
\[
(y,x,w)\mapsto (y+b_1x+c_1w^2+b_2w+b_3,x+d_1w+c_2,w+e_1)
\]
and thus the family involves quotienting by this action.  This group acts
freely (by \cite[Cor.~2.5]{fildefs}, since $E$ is smooth), so that
(\cite[Prop.~2.3]{fildefs}) we can exhaust the corresponding freedom by
setting appropriate coefficients to 0, and thus obtain a family over a
weighted projective space with $(5-4)$ generators of degree 1, $(4-2)$
generators of degree 2, $(3-1)$ generators of degree 3, $(2-0)$ generators
of degree 4, and $(1-0)$ of degrees 5 and 6, or in other words of degrees
$122334456$.  (Note that although the reference largely assumes
characteristic 0, the results we use are proved in arbitrary
characteristic.)  Since this is integral and 8-dimensional, it is a
quotient of $(\Lambda_{E_8}\otimes E)/W(E_8)$, and degree considerations
(along with the identification in \cite{fildefs} of the line bundle
corresponding to the action of $\G_m$ on the anticanonical section) tell us
that it must be isomorphic to $(\Lambda_{E_8}\otimes E)/W(E_8)$, and thus
the main theorem applies in the $E_8$ case.

One caveat here is that although fibers over distinct points of
$(\Lambda_{E_8}\otimes E)/W(E_8)$ correspond to nonisomorphic pairs
$(X,E)$, and {\em discrete} automorphisms of fibers are given by the
quotient of the stabilizer by the reflection subgroup of the stabilizer,
singular fibers can also have infinitesimal automorphisms; we will see
examples of this phenomenon below.
 
For $E_7$ and $E_6$, the argument is similar.  The main difference is that
now $-K_X$ has degree $2$ or $3$ respectively, so that the anticanonical
curve is no longer canonically an elliptic curve.  We thus have instead a
smooth genus 1 curve $C$ with a marked divisor class $D$ of degree $d$ and
a homomorphism $\rho_X:\Pic(X)/\langle K_X\rangle\to \Pic(C)/\langle
D\rangle$ which can be used in a similar way to construct families of
surfaces.  The key point is that a representatation as an iterated blowup
gives an isomorphism $\Pic(X)/\langle K_X\rangle\cong
\Lambda_{E_{9-d}}^\perp$, and thus if $C=E$ and $D=d[0]$, then the family
of smooth surfaces is parametrized by
\[
\Hom(\Lambda_{E_{9-d}}^\perp,E)\cong \Lambda_{E_{9-d}}\otimes E,
\]
and we again obtain a family of singular surfaces over the quotient we want
to understand.

For $d=2$ (i.e., $E_7$), the original model of $C$ is as a degree 4
hypersurface in a weighted projective space of degrees $1,1,2$, and we find
that the coefficients of lower order terms have degreees $1^62^43^24$,
while the unipotent action has coefficients of degrees $1^42$, so that
$(\Lambda_{E_7}\otimes E)/W(E_7)$ is a weighted projective space of degrees
$1^22^33^24$.  For $d=3$, the original model of $C$ is a cubic curve in
$\P^2$, and thus the coefficients lower order terms have degrees $1^62^33$
and the unipotent action has degrees $1^3$, telling us that
$(\Lambda_{E_6}\otimes E)/W(E_6)$ is a weighted projedtive space of degrees
$1^32^33$.

Something similar applies for $d=4$, where $C$ is now a complete
intersection of two quadrics in $\P^3$ and we get a description of
$(\Lambda_{D_5}\otimes E)/W(D_5)$ as a weighted projective space.  (For
$d\ge 5$, $C$ is no longer a complete intersection, and thus the
coefficients of lower order terms must satisfy certain equations in order
for the family to be a flat deformation of the cone over $C$; for $d=5$,
the quotient is $(\Lambda_{A_4}\otimes E)/W(A_4)\cong \P^4$, while for
$d=6$ the group is decomposable ($A_2A_1)$ and the quotient is $\P^2\times
\P^1$.)

Note that in each case there is an induced action of $\Aut_D(C)$ (i.e., the
automorphisms of $C$ fixing $D$) on the weighted projective space, which is
fairly straightforward to compute: the group acts projectively on the
original model of $C$ and normalizes the unipotent action, and thus the
action of any element can be computed by first acting on the generators of
the filtered algebra in the same way and then using the unipotent action to
put it back into standard form (i.e., zeroing out suitable coordinates).
This actually works scheme-theoretically, and thus in particular for
infinitesimal automorphisms in $\Aut_D(C)$ when those exist.

We note in passing that there is an analogous construction of
$\Lambda_{D_n}\otimes E/D_n$ as a moduli space of rational surfaces with
anticanonical curve isomorphic to $E$ and a marked rational ruling.  Since
we have simpler ways to compute $D_n$ invariants, and the construction does
respect triality, so does not help with the $F_4$ cases, we omit
the details.

\subsection{Automorphisms}
\label{ss:surf_auts}

There are a number of other cases of the main theorem in which we can view
the quotient as a moduli space of surfaces.  The key idea is that if
$Y\subset \Lambda_{E_m}\otimes E$ is a coset of an abelian subvariety, then
we can pull back to a family of surfaces over $Y$.  The quotient of $Y$ by
its stabilizer in $W(E_m)$ maps to its image in $(\Lambda_{E_m}\otimes
E)/W(E_m)$, and thus when the stabilizer acts as a reflection group, we can
hope to identify the quotient by showing that suitable generators of the
ambient weighted projective space vanish.  More generally, if the
reflection group of interest is only a subgroup of the stabilizer, we still
obtain a family of (possibly singular) surfaces over the quotient, but now
those surfaces have additional structure.

The problem with this approach as stated is that it is difficult in general
to find a modular interpretation of the subfamily that is simple enough to
let us write down the corresponding family of equations, even when such an
interpretation exists.  Moreover, if the structure corresponding to $Y$ is
not unique, its image in $(\Lambda_{E_m}\otimes E)/W(E_m)$ will tend to
intersect itself, and thus recovering the quotient of $Y$ requires
normalization.

As an example, consider the case that $Y$ is a reflection hypersurface
in $\Lambda_{E_8}\otimes E_8$, corresponding to the case that the
surface is singular.  This maps to a degree 60 hypersurface in
$(\Lambda_{E_8}\otimes E_8)/W(E_8)$, since the invariant divisor
given by the sum of reflection hypersurfaces represents 30 times
the invariant principal polarization and each hypersurface appears
with even multiplicity in any invariant.  Since there are 1366606
monomials of degree 60 on the weighted projective space, this is
hard enough to compute, let alone identify as a weighted projective
space (since $Y\cong \Lambda_{E_7}\otimes E$).  We can get around
this by moving the singular point to the origin (and thus in
particular separating it from any other singularities), or even
better by blowing down the anticanonical curve passing through the
singular point to obtain a del Pezzo surface of degree 2.

There is, however, one type of additional structure that is unique and in
fact automatically (in good characteristic) preserves the weighted
projective space structure.  This comes from the following general fact
about weighted projective spaces.

\begin{prop}
  Let $A$ be a finitely generated graded polynomial algebra with generators
  of positive degree, and let $G$ be a linearly reductive affine group
  scheme acting on $A$ preserving the grading.  Then the $G$-fixed
  subscheme of $\Spec(A)$ has the form $\Spec(B)$ where $B$ is a graded
  polynomial algebra with generators of positive degree.
\end{prop}

\begin{proof}
  Suppose first that $A$ is $G$-equivariantly the symmetric algebra of a
  $G$-module $M$.  The $G$-fixed subscheme is cut out by equations $g\cdot
  m=m$ where $m$ ranges over a basis of $M$ and $g$ ranges over $G$.  More
  precisely, there is a natural map $k[G]\otimes M\to M$ acting by
  $g\otimes m\mapsto gm-m$ and the $G$-fixed subscheme is cut out by the
  ideal generated by the image of that map.  The image is clearly
  $G$-invariant and contains a vector in each nontrivial irreducible
  constituent of $M$ and thus the $G$-fixed subscheme has the form
  $\Spec(B)$ where $B$ is the symmetric algebra of the quotient $N$ of $M$
  by the nontrivial isotypic components.  If $M$ is homogeneous with
  respect to the grading on $A$, then so is $N$, and thus $B$ is graded as
  required.

  For more general actions, note that for any $d>0$, the degree $d$
  component $A_d$ of $A$ is a $G$-module, and there is a natural invariant
  submodule given by the intersection of $A_d$ with the subring of $A$
  generated by $A_e$ for $e<d$.  Since $G$ is linearly reductive, that
  invariant submodule has a complement $M_d$.  Moreover, $\dim(M_d)$ is
  equal to the number of generators of $A$ of degree $d$, and thus $M_d=0$
  for all but finitely many degrees.  It is then straightforward to see
  that for any way of choosing an equivariant complement in each degree, we
  have $A\cong \Sym^*(\bigoplus_{d>0} M_d)$.  Thus the argument in the
  previous paragraph applies.
\end{proof}

\begin{rem}
  In our case, since we are dealing with finite group schemes, the
  ``linearly reductive'' condition is automatic in characteristic 0, while
  in characteristic $p$ reduces to saying that (a) the connected component
  of the identity is a multiplicative $p$-group (i.e., a product of groups
  $\mu_{p^l}$) and (b) the quotient by that subgroup has order prime to
  $p$.  Of course, the claim {\em can} still hold for more general groups,
  but does fail for some cases of interest to us.
\end{rem}

Note that this applies to the action of the group on the graded polynomial
algebra.  If instead we are given an action of a linearly reductive group
scheme $G$ on a weighted projective space (even in the stacky form), we
cannot quite apply the proposition, as $G$ need not act on the homogeneous
coordinate ring.  We do, however, have a natural action of an extension
$\G_m.G$ on the ring, and such an extension remains linearly reductive.  In
particular, the proof of the Proposition still lets us write the
homogeneous coordinate ring as the symmetric algebra of a $\G_m.G$-module
which is positively graded for the action of $\G_m$.  The $G$-fixed
subscheme (or substack) is then the {\em union} of subschemes corresponding
to the different (multi)sections $G\to \G_m.G$, each of which is a weighted
projective space.  We may also consider subgroups $\mu_l\subset \G_m$,
which correspond to the substacks on which the isotropy group contains
$\mu_l$; in scheme-theoretic terms, these typically correspond to
components of the singular locus.  (The exception is when they have
codimension 1, as a $\mu_l$-quotient ``singularity'' of codimension 1 is
not actually singular!)

In our setting, the weighted projective space arises as a quotient $A/G$
and there is an action of a group $N/G$ on the quotient, where $N$ is the
normalizer of $G$ inside $\Aut(A)$.  (This is not quite correct as stated,
as the weighted projective space structure depends on a choice of
$G$-equivariant line bundle, and $N$ is required to preserve that structure
as well; that is, we replace $N$ by the kernel of the natural map $N\to
\Hom(G,\G_m)$ measuring the failure to preserve equivariance.)  Thus for
any linearly reductive subgroup of $N/G$, the fixed locus of the quotient
will be a union of weighted projective spaces.

Now, fix a subgroup $H\subset N/G$ with $|H|\in k^*$ (so in particular $H$
is \'etale) and suppose $x\in A/G$ is a (not necessarily closed) point in
the $H$-fixed locus.  Then the preimage of $x$ in $A$ is a $G$-orbit of
points $y_1,\dots,y_m$, and for any $h\in H$ there is at least one preimage
$n_h\in N$ such that $n_hy_1=y_1$.  (Moreover, since $h$ has order prime to
$p$, we may choose $n_h$ to have order prime to $p$, as otherwise we may
replace $n_h$ by the power $n_h^{ap}$ such that $ap$ is congruent to $1$
modulo the order of $h$.)  For each $h$, the fixed subscheme of $n_h$ is a
coset of a {\em reduced} subgroup scheme of $A$, and thus the intersection
of the fixed subschemes of the preimages is a coset of a reduced subgroup
scheme, or equivalently a finite union of cosets of an abelian subvariety
$B$.  If $y_1$ is not the generic point $y'_1$ of the coset $y_1+B$, then
the image $x'\in A/G$ of $y'_1$ is a point in the $H$-fixed locus
generalizing $x$.  It follows immediately that every associated point of
the $H$-fixed locus in $A/G$ is the image of the generic point of a coset
of an abelian subvariety!

In general, it is not too difficult to use this to identify cosets of
abelian subvarieties corresponding to the $H$-fixed locus in $A/G$.
Indeed, for any subgroup $H^+$ of $N$ with image $H$ and kernel of
invertible order, the fixed locus is either empty or a union of cosets of
abelian subvarieties, and each such coset is fixed by a possibly larger
group of the form $H^+$.  We thus obtain a collection of pairs $(H^+,y+B)$
of groups and cosets such that $H^+$ is the stabilizer of $y+B$ inside $HG$
and $y+B$ is a component of the fixed locus of $H^+$.  We then see that
$y+B$ corresponds to a component of the $H$-fixed locus in $A/G$ iff
every subgroup of $H^+$ surjecting on $H$ has fixed subscheme of the same
dimension.  Note that since every component of the fixed subscheme (in good
characteristic) gives a special case of the main theorem, we can still get
useful results without having to identify {\em every} component.

There are two caveats here: First, although this gives us cases of the main
theorem, it does not tell us the degrees of the invariants unless we have
some other means of determining how $H$ acts on $A/G$.  Second, since the
different components of the $H$-fixed locus are distinguished by how $H$
acts on the equivariant line bundle, it can be subtle to determine the
bijection between the components of the $H$-fixed locus and the $G$-orbits
of cosets arising from the above calculations.  In practice, these issues
are relatively minor: in the cases arising from families of surfaces, we
{\em do} know the action of $H$, and the components we are interested in
tend to be distinguished up to equivariant isomorphism by their dimension.

In most cases, a similar calculation applies to the $\mu_l$-fixed locus
(i.e., the points of the weighted projective stack with stabilizer
containing $\mu_l$).  We may assume $l$ is prime (if $l=ql'$ with $q$
prime, then the $\mu_l$-fixed subscheme is the $\mu_{l'}$-fixed subscheme
of the $\mu_q$-fixed subscheme), and for simplicity assume characteristic
0.  If the $\mu_l$-fixed locus is not codimension 1, then it corresponds to
a component of the singular locus of the quotient, and thus to a component
of the locus of points where the stabilizer is not generated by
reflections.  More precisely, a point of $A$ maps to the $\mu_l$-fixed
locus iff the quotient of its stabilizer by the reflection subgroup of the
stabilizer has order a multiple of $l$.  We again obtain a collection of
pairs $(H^+,y+B)$ where $H^+$ is the stabilizer of $y+B$, $y+B$ is a
component of the fixed subscheme of $H^+$, and the reflection subgroup of
$H^+$ has index $l$.  The coset $y+B$ maps to the generic point of the
$\mu_l$-fixed locus iff it is maximal among the cosets arising in this way,
or equivalently iff any generator of $H^+$ over its reflection subgroup has
fixed subscheme of dimension $\dim(B)$.  (Compare the calculation in
\cite{MarkushevichD/MoreauA:2022a} of the singular locus of the quotient
corresponding to $ST_{24}$.)

As an example of the latter, consider the $\mu_2$-fixed locus in $(E\otimes
\Lambda_{E_8})/W(E_8)$, a weighted projective space of degrees $2,2,4,4,6$
(i.e., the even degrees of generators of the $E_8$ quotient).  There are 6
possibilities for an element of $W(E_8)$ that might generate $H^+$: they
must have even order and have $4$-dimensional fixed subscheme.  We may rule
out the order 6 cases (as their cubes still generate $H^+$ and have fixed
locus at least as large) and can easily see that the {\em identity}
components of the fixed loci do not work (the element being contained in
the reflection subgroup of the stabilizer).  There is only one case that
allows nonidentity components: the order 2 element $g_2$ with centralizer
of order $221184$ has fixed subgroup isomorphic to $E^4\times E[2]^2$.  For
any component that is fixed by a reflection, multiplying $g_2$ by that
reflection increases the dimension of the fixed subgroup, and we thus find
that the only components that survive are those (a single orbit) for which
the two elements of $E[2]$ generate $E[2]$.  The stabilizer is then an
extension of $W(F_4)$ by an elementary abelian subgroup of order $2^4$ that
acts on $E^4$ by translations, allowing us to see that this corresponds to
$\Lambda_{F_4}\otimes E/W(F_4)$.  Unfortunately, this argument breaks badly
in characteristic 2, as in that case the fixed subgroup is highly
nonreduced, and the appearance of $\mu_2$ in the stabilizers is quite
subtle.

Luckily, there is an alternate approach to $\Lambda_{F_4}\otimes E/W(F_4)$
that avoids this: rather then embed $F_4$ in $E_8$, we embed it in $E_6$ as
the sublattice fixed by a diagram automorphism.  This corresponds to the
coset of $N$ containing $[-1]$, and we can easily compute the action on the
weighted projective space using the model as a family of cubic surfaces.
To be precise, we may model $\Lambda_{E_6}\otimes E/W(E_6)$ as the family
of filtered algebras extending the cubic curve corresponding to the {\em
  Weierstrass} model of the curve.  Over $\Z[1/6]$, such surfaces have
canonical form
\[
y^2z - x^3-a_4 xz^2-a_6 z^3 +
i_1 xy+j_1 xz + k_1 z^2
+i_2 y + j_2 x + k_2 z
+i_3
=0,
\]
and one lift of the action of $[-1]$ corresponds to negating $y$ and thus
negating $i_1$ and $i_2$.  The fixed locus of this action is a weighted
projective space with generators $1,1,2,2,3$, while the other lift negates
$j_1$, $k_1$, $i_2$, $i_3$ so has fixed locus a weighted projective space
with generators $1,2,2$.  On the other hand, it is easy to see that the
generic point of $\Lambda_{F_4}\otimes E\subset \Lambda_{E_6}\otimes E$ is
{\em only} fixed by the diagram automorphism, and thus is a component of
the $[-1]$-fixed locus.

To extend this to finite characteristic, we note that the map
$\Lambda_{F_4}\otimes E\to (\Lambda_{E_6}\otimes E)/W(E_6)$ always maps to
the $[-1]$-fixed locus.  If we can show by direct means that the
$[-1]$-fixed locus remains $4$-dimensional and the unique component of
dimension $4$ remains a weighted projective space of the same degrees, then
the map from $(\Lambda_{F_4}\otimes E)/W(F_4)$ is an isomorphism by degree
considerations.  This is actually automatic in odd characteristic: in any
odd characteristic, the homogeneous coordinate ring of the $E_6$ quotient
is a symmetric power of a $[\mu_2]$-equivariant graded module, and that module
lifts equivariantly to characteristic 0, so the fixed locus is the ``same''
union of weighted projective spaces!  Thus it remains only to consider the
characteristic 2 case.  The above canonical form of course fails to work,
but there is still a covering by ``canonical'' forms: for each choice of 3
degree 2 monomials, there is an open subset of the family of cubic curves
where the other degree 2 monomials can be eliminated (and uniquely so!)
using the unipotent freedom, and those open subsets cover the space of
smooth cubics (and in particular the smooth Weierstrass cubics).  For each
such choice, we can determine the action of $[-1]$ by first acting linearly
on $y,x,z$ in the same way that $[-1]$ acts on the curve and then reducing
to the corresponding ``canonical'' form.  In particular, every smooth
ordinary Weierstrass curve of characteristic 2 has a model with $a_1=1$,
$a_3=0$, giving us a ``canonical'' form
\[
y^2z + xyz +x^3+a_2 x^2z+a_4 xz^2+a_6z^3
+i_1 y^2+j_1 x^2+k_1 z^2
+i_2 y + j_2 x + k_2 z
+i_3
=0
\]
with $[-1]$ acting by
\[
(i_1,j_1,k_1,i_2,j_2,k_2,i_3)\mapsto (i_1,j_1+i_1,k_1,i_2,j_2+i_2,k_2,i_3).
\]
(In this case, the homogeneous action of $[-1]$ preserves the standard
form.)  In particular, we see that the $[-1]$-fixed locus is $i_1=i_2=0$,
so is still a weighted projective space.  (Moreover, we see that the locus
$i_1=i_2=0$ has a natural interpretation as a moduli stack of affine cubic
surfaces on which the action of $[-1]$ extends.)  For supersingular curves,
we use a somewhat different ``canonical'' form:
\[
y^2z + yz^2 +x^3+a_4 xz^2+a_6z^3
+i_1 xy+j_1 yz+k_1 xz
+i_2 y +j_2 x + k_2 z
+i_3
=
0.
\]
Here $[-1]$ acts on the cubic surface as $y\mapsto y+z+j_1$ and
on invariants as
\[
(i_1,j_1,k_1,i_2,j_2,k_2,i_3)
\mapsto
(i_1,j_1,k_1+i_1,i_2,j_2+i_1j_1,k_2+i_2,i_3+i_2j_1).
\]
Again, the fixed locus is $i_1=i_2=0$ and thus gives the correct weighted
projective space, finishing the proof of the main theorem for
$\Lambda_{F_4}\otimes E/W(F_4)$.

As for $G_2$, there is another crystallographic action of $W(F_4)$
associated to a choice of $2$-isogeny $\phi$ on $E$ (corresponding to the
quotient of $\Lambda_{F_4}\otimes F_4$ by the corresponding subgroup of
$E[2]^2$).  This naturally leads one to consider the $\ker(\phi)$-fixed
locus in $\Lambda_{E_7}\otimes E_7/W(E_7)$ (this being the only one of our
three cases in which we can see a cyclic $2$-torsion subgroup).  Over
$\Z[1/2]$, the typical degree 2 curve of genus 1 with an explicit action of
a $2$-torsion element has the form
\[
y^2 = x^4 + a_2 x^2 w^2 + a_4 w^4
\]
with the $2$-torsion element acting by $(y,x,w)\mapsto (-y,-x,w)$ or
$(y,x,w)\mapsto (-y,x,-w)$.  The associated surface has reduced form
\[
y^2-x^4-a_2 x^2 w^2-a_4 w^4
+i_1 x^2 w + j_1 x w^2
+i_2 x^2 + j_2 xw + k_2 w^2
+i_3 x + j_3 w
+i_4
\]
and the $2$-torsion element acts by negating $j_1$, $j_2$, and $i_3$ or
by negating $i_1$, $j_2$, and $j_3$, so that in either case the fixed
subscheme is a weighted projective space of degrees $1,2,2,3,4$.

To identify this as a crystallographic quotient, we note that the
corresponding element $g\in W(E_7)$ cannot fix any $-1$-curve, as the
corresponding point of $E$ would need to be invariant under the
translation.  There is a unique conjugacy class of elements of 2-power
order that fix no $-1$-curve and have 4-dimensional fixed subgroup, so that
the above weighted projective space arises as a quotient of one of the
corresponding fixed components.  (In fact, {\em both} weighted projective
spaces arise as such quotients, since they are swapped by the action of the
other $2$-torsion point!)  In terms of the parametrization as $6$-fold
blowups of $\P^1\times \P^1$, we are thus left with the surfaces such that
\[
x_2-x_1=x_4-x_3=x_6-x_5=t
\]
(where $t$ is the chosen $2$-torsion point) and $s+f-x_1-x_3-x_5\in
E[2]\subset \Pic^1(E)$.  We thus have one component for each choice of
$s+f-x_1-x_3-x_5$.  The element that swaps $x_1$ and $x_2$ preserves the
subscheme but adds $t$ to $s+f-x_1-x_3-x_5$, so we in fact have two orbits
of components that are swapped under translating all six $x_i$ by the other
generator of the $2$-torsion group.  So we may as well consider the
component for which $s+f-x_1-x_3-x_5=0$.  We may thus parametrize the
subvariety by $s,f,x_1,x_3$.  The action $(e_1,e_2,e_5,e_6)\mapsto
(e_2,e_1,e_6,e_5)$ adds $t$ to $x_1$ while $(e_3,e_4,e_5,e_6)\mapsto
(e_4,e_3,e_6,e_5)$ adds $t$ to $x_3$, giving a translation subgroup of the
stabilizer, while the remainder of the stabilizer acts as $W(F_4)$.  We
thus find that the weighted projective space is the quotient
$\Lambda_{F_4}\otimes E/\langle t\rangle^2.W(F_4)$.

The corresponding map $\Lambda_{F_4}\otimes E\to \Lambda_{E_7}\otimes
E/W(E_7)$ extends to arbitrary pairs $(E,t)$ (including in characteristic
2) and gives a morphism
\[
\Lambda_{F_4}\otimes E/\langle t\rangle^2.W(F_4)\to \Lambda_{E_7}\otimes
E/W(E_7)
\]
factoring through the $\langle t\rangle$-fixed subscheme of the codomain.
(The one caveat is that in characteristic $2$, $t$ may become $0$, in which
case $\langle t\rangle$ should be replaced by the corresponding nonreduced
subgroup.)  Indeed, this reduces to the claim that the morphism
\[
\Lambda_{F_4}\otimes E\to \Lambda_{E_7}\otimes E/W(E_7)
\]
is invariant under precomposition with $\langle t\rangle^2.W(F_4)$ and
postcomposition with $\langle t\rangle$.  Both invariance conditions are
closed, and thus reduce to the generic (i.e., characteristic 0) case.  We
thus see that, just as in the other $F_4$ case, it suffices to show that
the $\langle t\rangle$-fixed subscheme remains a weighted projective space
with the correct degrees.  This is automatic when $\langle t\rangle$ is
linearly reductive, which covers both the odd characteristic cases and the
case of ordinary curves of characteristic 2 with $\langle t\rangle$
nonreduced.  For the \'etale characteristic 2 case, we find that the
typical surface with an action of $\langle t\rangle$ has the form
\[
y^2 + (x^2 + a_1 t x + j_2)y + (a_4 t^4+i_1 t^3+i_2 t^2+i_3 t+i_4)=0,
\]
while the typical surface with an action of $\alpha_2$ has the form
\[
y^2 + (t^2+i_1tu) y = tx^3 + i_2 x^2u^2+j_2txu^2+i_3 tu^3+i_4 u^4=0,
\]
so that in both cases the fixed subscheme behaves well.

Note that we have now settled all of the cases corresponding to {\em real}
reflection groups; in characteristic 0, this was shown in
\cite{BernsteinJ/SchwarzmanO:2006,KacVG/PetersonDH:1984}, but the above
arguments show that it continues to hold in every finite characteristic.

Some further notes on the above constructions.  First, the main technical issue
with the approach to $F_4$ via $\mu_2$-invariants on the $E_8$ quotient is
the fact that $\mu_2$ acts trivially on the quotient scheme.  We can work
around this by instead considering the map from the cone over
$\Lambda_{F_4}\otimes E$ to the quotient of the cone over
$\Lambda_{E_8}\otimes E$, and again find that the invariance condition is
closed so remains true in finite characteristic.  It is then automatic that
the codomain is flat!

The approach via $E_6$, though it requires some additional verification in
characteristic 2, has simpler combinatorics, and in particular makes it
easy to determine the analogous quotient by $W(D_4)$.  Indeed, we can
choose our canonical form so that the coefficient of $y^2$ is $z$, and note
that the subscheme $z=0$ is then generically a union of three coincident
lines.  There is then an $S_3$-cover of the family over which those lines
are rational, and since $W(E_6)$ acts faithfully on lines, that $S_3$-cover
must be the quotient by a normal subgroup of $W(F_4)$.  Since the quotient
is still a weighted projective space (with degrees $1,1,1,1,2$), the normal
subgroup must be a reflection group, and is thus $W(D_4)$.  Note that this
gives an alternate approach to $\Lambda_{G_2}\otimes E/W(G_2)$ as the locus
of $\Lambda_{D_4}\otimes E/W(D_4)$ fixed by triality.  (This can also be
derived from the $\mu_3$-fixed locus in the $E_8$ quotient, but this again
requires a careful choice of component of the fixed subgroup.)  We can also
obtain the other $G_2$ quotient via surfaces, in this case as a family of
{\em cubic} surfaces with an extension of the action of a cyclic
$3$-torsion group.

\subsection{$j=0$ cases}
\label{ss:complexsurf0}

In general, the application of the above results to surfaces involve
families of surfaces on which some automorphisms of the graded algebra
extend to the filtered algebra.  For general curves, there is not much more
to say.  Indeed, the automorphism group of the graded algebra is
$\G_m.E[d].\Aut_0(E)$, and there is an induced filtration of any subgroup.
In particular, if the group meets $\G_m$ nontrivially, then we may first
pass to the fixed locus of the intersection $\mu_l$, which is only
interesting (i.e., corresponds to an imprimitive reflection group) for
$d=1$, $l\in \{2,3\}$.  Moreover, when $d=1$, the other element $[-1]$ of
the group actually extends canonically, so those are the only cases arising
from $E_8$.  Similarly, when $d=2$, the element $x\mapsto D-x$ mapping to
$[-1]$ extends canonically, and we have already considered translations.
Finally, for $d=3$, we have already considered extensions of $[-1]$ and of
translations, and since $[-1]\in W(G_2)$ the surfaces on which a
translation extend also have extensions of $[-1]$.

If $j(E)\in \{0,1728\}$, the curve has additional automorphisms, and thus
we may consider extensions of those automorphisms instead.  Note that the
extension to bad characteristic will tend to fail for components not of the
maximal dimension, as the smaller components tend to coalesce into the
larger components upon reduction.

\begin{eg}
  In characteristic not 3, consider filtered deformations of the
  Weierstrass curve $y^2+w^3y=x^3$ on which the automorphism
  $(y,x,w)\mapsto (y,\zeta_3 x,w)$ extends.  The corresponding family is a
  weighted projective space with degrees $2,3,4,5,6$:
  \[
  y^2+w^3y-x^3 + i_2 wy + i_3 y + i_4 w^2 + i_5 w + i_6=0
  \]
  The corresponding subvariety is a component of $\ker(g-[\zeta_3])$ and
  there is a unique choice of $g$ (up to conjugation) making the kernel
  $4$-dimensional.  We may then verify that the kernel is reduced and
  connected (with polarization of degree 9), and the stabilizer of the
  kernel is the Shephard-Todd group $ST_{32}$ of order $155520$.
  This is the unique crystallographic action of $ST_{32}$, and thus it
  remains only to show that the quotient remains a weighted projective
  space in finite characteristic.  As in the $F_4$ cases, this is immediate
  in characteristic prime to $3$, and thus it remains only to show that the
  fixed subscheme is still a weighted projective space in characteristic
  3.  In that case, the curve is $y^2=x^3-w^4x$ with $[\zeta_3]$ acting by
  $x\mapsto x+w^2$, and the corresponding family of filtered algebras
  before extending the $[\zeta_3]$ action is
  \[
  y^2-x^3+w^4x + a_1 wx^2 + a_2 x^2+b_2 w^2x + a_3 wx + b_3 w^3 + a_4 x +
  b_4 w^2 + a_5 w + a_6=0.
  \]
  The action of $[\zeta_3]$ is straightforward to compute (i.e., act by
  $x\mapsto x+w^2$ then reduce back to standard form).  In this case, the
  fixed locus is not actually a weighted projective space, but its reduced
  subscheme {\em is}, and the domain of the usual map is reduced, so the
  argument still works to show that the main theorem holds for $ST_{32}$ in
  characteristic 3.
\end{eg}

\begin{eg}
  Similarly, over $\Z[1/3]$, extensions of $(x,y,z)\mapsto (\zeta_3 x,y,z)$
  to filtered deformations of the cubic curve $x^3+y^3+z^3=0$ have the form
  \[
  x^3+y^3+z^3 + a_1 yz + a_2 y + b_2 z + a_3
  \]
  and correspond to the unique case of the main theorem corresponding to
  $ST_{25}$ (acting on its unique crystallographic lattice, an abelian
  3-fold with degree $9$ polarization).  In characteristic 3, the
  (reduced!) fixed subscheme corresponds to the family
  \[
  y^2z - x^3 + xz^2 + a_1 y^2 + a_2 y + b_2 z + a_3=0
  \]
  so that again the characteristic 3 case has polynomial invariants.  In
  either case, there is a residual action of $E[\sqrt{-3}].[\mu_2]$ commuting
  with the action of $[\mu_3]$, so acting on the above families.  The
  quotient by $[\mu_2]$ (which swaps $y$ and $z$ in the $\Z[1/3]$ model and
  negates $y$ in the characteristic 3 model) is in either case polynomial
  with invariants of degrees $1,2,3,4$, corresponding to the degree $9$
  lattice for $ST_{26}$, while the quotient by $E[\sqrt{-3}].[\mu_2]$
  corresponds to the degree $3$ lattice for $ST_{26}$ and has invariants of
  degrees $1,3,4,6$.  Indeed, over $\Z[1/3]$, $E[\sqrt{-3}]$ acts by
  $(a_1,a_2,b_2,a_3)\mapsto (a_1,\zeta_3 a_2,\zeta_3^{-1}b_2,a_3)$, so the
  calculation is straightforward even in characteristic 2.  In
  characteristic 3, $\alpha_3$ acts on the graded algebra by
  $(x,y,z)\mapsto (x+\epsilon y-\epsilon^2 z,y+\epsilon z,z)$ (with
  $\epsilon^3=0$) and thus on the invariants by $(a_1,a_2,b_2,a_3)\mapsto
  (a_1,a_2+ a_1^2 \epsilon,b_2+\epsilon a_2-\epsilon^2
  a_1^2,a_3-a_1a_2\epsilon+a_1^3\epsilon^3)$.  We thus find that the
  (algebraically independent) polynomials
  $a_1,a_3+a_1b_2,a_2^2+a_1^2b_2,b_2^3$ are invariant under
  $E[\sqrt{-3}].[\mu_2]$ and thus by degree considerations give generating
  invariants as required.
\end{eg}

\begin{rem}
  We could also obtain the degree 9 $ST_{26}$ case directly via the
  $[\mu_3]$-fixed subscheme of the $E_7$ quotient.  The corresponding
  surface has a natural nodal anticanonical curve, and the $ST_{25}$
  quotient is the double cover that splits the two branches at the node.
  The degree 3 case of $ST_{26}$ also arises directly as a $3$-dimensional
  component of the $[\mu_3]$-fixed subscheme of the $E_8$ quotient, but
  that description fails in characteristic 3.
\end{rem}

\begin{rem}
  There is another conjugacy class of order 3 elements in $E[3].\Aut_0(E)$
  arising by composing $[\zeta_3]$ with a translation, but the
  corresponding fixed subscheme is a $1$-dimensional weighted projective
  space, so not interesting for our purposes.  (In addition, this conjugacy
  class degenerates to the above conjugacy class in characteristic 3!)
\end{rem}

\begin{eg}
  The $F_4$ case corresponds to the group $ST_5$; computing it directly
  from $F_4$ gives a weighted projective space of degrees $1,2,3$ away from
  characteristic 3.  We may also interpret this via the $[\mu_6]$-fixed
  subscheme of the $E_6$ quotient, and thus as the $[-1]$-fixed subscheme
  of the $ST_{26}$ quotient, the latter giving a weighted projective space
  away from characteristic 2.  Thus the claim holds in general.  The
  intersection of $ST_5$ with $W(D_4)$ is $ST_4$, and thus the quotient by
  $ST_4$ arises by splitting the corresponding (cyclic!) cubic, and has
  invariants of degrees $1,1,2$.  (This corresponds to a twisted action of
  triality on the $D_4$ quotient.)
\end{eg}

\begin{eg}
  The other lattice for $F_4$ depends on a choice of $2$-torsion point, so
  does not normally admit an action of $[\mu_3]$.  However, in
  characteristic 2, the 2-torsion subgroup is the kernel of Frobenius and
  thus {\em is} normalized by $[\mu_3]$.  We may thus take the
  $[\mu_3]$-invariants in this case, and find that the resulting (sporadic)
  degree 3 lattice for $ST_5$ has invariants of degrees $1,3,4$.
\end{eg}

\subsection{$j=1728$ cases}
\label{ss:complexsurf1728}

\begin{eg}
  In characteristic 0, the Weierstrasss curve with $j=1728$ can be put in
  the form $y^2=x^3-xw^4$, and the canonical form of the filtered algebra
  over the $E_8$ quotient is
  \[
  y^2-x^3+(w^4+a_2 w^2+b_3 w+a_4)x+(a_1 w^5+b_2 w^4+a_3 w^3+b_4 w^2+a_5 w+a_6)=0.
  \]
  The $[\mu_4]$-fixed subscheme has a unique $4$-dimensional component
  corresponding to surfaces on which $(y,x,w)\mapsto (-iy,-x,-w)$ extends,
  or equivalently with $b_2=b_3=b_4=a_6=0$.  This is a weighted projective
  space with generators of degrees $1,2,3,4,5$, corresponding to a
  crystallographic action of $ST_{31}$.  The corresponding lattice has
  polarization of degree 16, and the equivariant line bundle is a square
  as a line bundle, but not as an {\em equivariant} line bundle.  Again,
  the analogous claim in odd characteristic follows immediately.  However,
  the weighted projective space property appears to break in characteristic
  2, as in that case the fixed subscheme is a hypersurface of degree 6 in a
  weighted projective space with degrees $1,2,3,4,5,6$.  (This has the same
  Hilbert series as that of the main component in characteristic 0, so is
  indeed equal to the quotient by $ST_{31}$.) This is not actually a
  counterexample to the main theorem, as in characteristic 2 the line
  bundle {\em is} equivariantly a square and thus the degree 1 invariant
  has divisor $2\Theta$ and is a square of an invariant section of the
  square root bundle.  Where the degree 1 invariant vanishes, the equation
  of the fixed subcheme is thus itself a square of a degree 3 invariant
  (which can be used as a generator), and we can then choose the degree 5
  invariant to make the equation of the form $a_1a_5=a_3^2$.  Writing
  $a_1=a_{1/2}^2$, we see that $a_3/a_{1/2}$ is an integral invariant of
  degree $5/2$ (the square root of the degree 5 invariant), and thus that
  the quotient in characteristic 2 is a weighted projective space with
  degrees $1/2,2,5/2,4,6$ (or degrees $1,4,5,8,12$ relative to the
  principal polarization).
\end{eg}

\begin{rem}
  The failure of the line bundle to be equivariantly a square is measured
  by a class in $H^1(ST_{31};A[2])$, and the corresponding action without a
  global fixed point {\em does} have an invariant theta divisor (and is
  still crystallographic).  We will see below that the quotient by this
  torsor action is a weighted projective space with degrees $1,4,5,8,12$,
  just as in the characteristic 2 case.  (Unfortunately, this action does
  not appear to correspond to a family of surfaces in odd characteristic.)
\end{rem}

\begin{rem}
  There is a homomorphism $ST_{31}\to S_6$ such that the preimage of the
  point stabilizer is the reflection group $ST_{29}$.  The corresponding
  action of $ST_{29}$ on the principally polarized abelian variety is again
  crystallographic, and the quotient is thus a degree 6 cover of the above
  family of surfaces.  However, it is unclear how to describe this covering in
  natural geometric terms (i.e., to give a simple description of the six
  structures being permuted by the $S_6$ quotient), and thus we have been
  unable to give a surface-based argument that this quotient is a weighted
  projective space.
\end{rem}

\begin{eg}
  Consider the $E_7$ quotient associated to the curve $y^2=x^4+w^4$ (i.e.,
  the $j=1728$ curve embedded by the divisor $E[1-i]$).  The $[\mu_4]$-fixed
  locus is contained in the $[-1]$-fixed locus, which is in turn equal to
  the $E[1-i]$-fixed locus (since $y\mapsto -y$ differs from $[-1]$ by
  translation in the nontrivial $1-i$-torsion point).  In other words, the
  $E_7$ case reduces to the $F_4$ case corresponding to the $2$-torsion
  group $E[1-i]$.  The corresponding family is
  \[
  y^2=x^4+w^4+a_2 w^2+ a_3 w +a_4,
  \]
  so a weighted projective space of degrees $2,3,4$, corresponding to the
  crystallographic action of $ST_8$ on a lattice of degree $2$.  In
  characteristic 2, we instead have (after passing to the reduced
  subscheme) the family
  \[
  y^2 + w^2 y = wx^3 + a_2 (x^2+wx)+a_3 w+a_4=0,
  \]
  so again the result follows by flatness.
\end{eg}

\begin{rem}
  Again, there is another conjugacy class of actions (embedding by
  the divisor $2[0]$), but without interesting fixed subscheme.
\end{rem}

\begin{eg}
  Similarly, an action of $[\mu_4]$ on a cubic curve restricts to an action
  of $[\mu_2]$, so again reduces to the analogous $F_4$ quotient.  The
  group is again $ST_8$, now with a polarization of degree 8 (which is
  twice the invariant polarization of degree 2), and having degrees
  $1,2,3$.  The corresponding family over $\Z[1/2]$ is
  \[
  y^2z+xz^2=x^3 + a_1 x^2 + a_2 x + a_3
  \]
  (with an extension of $(y,x,z)\mapsto (-iy,x,-z)$).
\end{eg}

\begin{rem}
  Of course, this graded algebra is nothing other than the second Veronese
  of the graded algebra appearing in our previous discussion of $ST_8$, and
  is a weighted projective space because the previous algebra has precisely
  one generator of odd degree.  In particular, the second algebra has an
  invariant of degree 3 which is a square on the abelian variety.  A likely
  candidate for this invariant is $a_3$, which cuts out a component of the
  locus where the surface is singular.
\end{rem}

\subsection{A $j=8000$ case}
\label{ss:complexsurf8000}

\begin{eg}
  \label{eg:ST12}
There is another complex crystallographic case that arises as a fixed
subscheme of a real crystallographic group, although again it is difficult
to describe the family explicitly.  The key idea is that the Coxeter group
$F_4$ has an additional diagram automorphism swapping the two conjugacy
classes of roots.  This automorphism does not act on $\Lambda_{F_4}\otimes
E$, and in the $2$-isogeny case becomes an isomorphism to the variety
corresponding to the dual $2$-isogeny.  There are, however, two cases in
which $E$ has a {\em self-dual} (up to isomorphism) $2$-isogeny: for
$j=1728$, $1-i$ is self-dual (up to composition with an automorphism), as
is the endomorphism $\sqrt{-2}$ on the curve with $j=8000$.  (The remaining
case $(-1+\sqrt{-7})/2$ of an endomorphism of degree 2 is not self-dual.)

The $\sqrt{-1}$ case gives nothing interesting, but in the $\sqrt{-2}$
case, we obtain a lattice for $ST_{12}$ with induced polarization twice a
principal polarization.  The quotient must be a weighted projective space
with degrees a submultiset of the degrees $1,2,2,3,4$ for $F_4$, and thus
degree considerations force it to have degrees $2,4,6$ relative to the
principal polarization.  Of course, this argument fails in characteristic
2, and we will see below that in that case the quotient has degrees $1,3,8$
(as does the relevant torsor action of $ST_{12}$.)
\end{eg}

\subsection{Quaternionic cases}
\label{ss:quatsurf}

When $E$ is supersingular of characteristic 2 or 3, we have even more
automorphisms to extend.  Of course, extending a single automorphism will
just give the corresponding special case of a complex group, but we can
consider cases in which a larger subgroup extends.  Of course, the
resulting quaternionic representation embeds in the original
representation, and since this embedding multiplies dimensions by 4, the
only interesting cases lie in $E_8$.

\begin{eg}
  Let $E$ be the curve $y^2=x^3-w^4x$ over $\bar\F_3$, and consider surfaces to
  which the entire group $\Aut_0(E)$ extends.  The center $[\mu_2]$ extends
  canonically, so we may start by extending
  the normal subgroup $[\mu_6]$ (which gives a weighted projective space as
  shown above) and then extending the residual action of the quotient.
  The result is the family
  \[
  y^2 = x^3 - (t^2+c_2)^2 x + c_3 t^3 + c_5 t,
  \]
  and shows that the quaternionic reflection group $2.\Alt_6\subset
  \GL_2(\Q_3)$ has quotient a weighted projective space of degrees $2,3,5$.
\end{eg}

\begin{eg}
  \label{eg:2_-^{1+4}.Alt_5}
  Let $E$ be the curve $y^2+y=x^3$ over $\bar\F_2$.  We can again write
  down the typical surface with an extension of $\Aut_0(E)$:
  \[
    y^2+t^3 y = x^3 + (c_4 t^2 + c_5 t + c_6),
  \]
  and conclude that the corresponding crystallographic action of
  $2_-^{1+4}.\Alt_5$ has quotient the weighted projective space with
  degrees $4,5,6$.  Note that the polarization is actually twice a
  principal polarization; we will see (aided by the above calculation) that
  the principal polarization also gives polynomial invariants, of degrees
  $5,8,12$.  (We also get polynomial invariants from $4\Theta$, of degrees
  $2,3,5$.)
\end{eg}

\begin{rem}
  We could also consider extensions of $Q_8$ in characteristic 2.  In
  addition to the above family, we get another family with degrees
  $1,2,6$, but the corresponding group is imprimitive, so does not give
  anything new.
\end{rem}

Both of the above examples also arise by taking fixed subschemes by
suitable outer automorphisms of quotients by complex reflection groups (in
both cases arising from the fact that the complex conjugate group is
isomorphic).  There is another interesting example of such automorphisms
coming from the group $ST_{34}$ of rank 6.  This has an outer automorphism
acting by complex conjugation on the center, which in characteristic 3 acts
on the abelian variety (using the fact that there is a unit in $Q_3$ acting
by complex conjugation on $\zeta_3$).  Since this is an involution and the
characteristic is odd, we immediately conclude (presuming that the main
theorem holds for $ST_{34}$) that it gives rise to a case of the main
theorem, and in particular shows that the corresponding lattice (twice a
principal polarization) for $[\mu_2]\times \PGU_3(\F_3)$ has quotient a
weighted projective space.  Once we have shown that $ST_{34}$ has degrees
$1,3,4,6,7,9,12$, it will follow (since we know the product of degrees)
that the degrees of $[\mu_2]\times \PGU_3(\F_3)$ are either $1,3,7,12$ or
$1,4,7,9$.  (In fact, since the quotient in this case is tamely ramified,
we can rule out the former by computing that the canonical class has degree
$42$ relative to the principal polarization, and $1+3+7+12\ne 42/2=1+4+7+9$.)

\section{Direct calculations}

\subsection{Generalities on using normal subgroups}

For \'etale groups acting linearly on affine space, we can always find the
invariants of any given degree by solving the equations $gm=m$ for $g$
ranging over generators of the group.  This can be done more generally for
(cones over) projective schemes by replacing the equations by $gm-m\in I$
where $I$ is the ideal cutting out the homogeneous coordinate ring.
(Similarly, when the group is not \'etale, we may take $g$ to be a suitable
point of $G$ over some $k$-algebra $R$ and insist that $gm-m\in I\otimes
R$.)  When the quotient is a polynomial ring, we can then simply find
invariants of degrees $1,2,\dots$ and check for invariants of each degree
that are not in the ring generated by the invariants of lower degree, until
we get $n+1$ such invariants (which we then verify are algebraically
independent).

However, for this approach to be feasible in our case, we need to be able
to explicitly write down the homogeneous coordinate ring of the polarized
abelian variety, and this tends to be difficult even for abelian surfaces!
Part of the difficulty is that we are given the polarization as a Hermitian
matrix in $\End(E^n)$, and although there are (at least in principle) ways
to compute a divisor representing the polarization, it is much more
difficult to find {\em invariant} such divisors.  (Indeed, in several
cases, there {\em is} no $G$-invariant divisor representing the minimal
polarization.)  Thus even when we {\em can} write down the homogeneous
coordinate ring, it is typically unclear how to write down the action of $G$.

The one exception is when the polarization is split, so that $A\cong E^n$.
More generally, if there is {\em some} lattice (possibly not
crystallographic) on which our group preserves a split polariation, we can
express $A$ as a quotient of $E^n$ by the relevant torsion subscheme, or
equivalently replace $G$ by a somewhat larger group scheme.  This, of
course, is precisely what we did in the imprimitive case; unfortunately, it
is easy to see that this can {\em only} happen when $G$ is imprimitive!

If our group $G$ has a normal subgroup $N$, we can hope that the quotient
by $N$ may be easier to work with.  The simplest case is that $N$ is itself
a crystallographic reflection group, as then we may take the invariants of
$G/N$ on the corresponding weighted projective space.  (This still has the
difficulty that the action of $G/N$ may not be obvious, but we will be able
to finesse this.)  Another case is that $N=[\mu_2]$, so that we are looking
for the action of $G/N$ on the {\em Kummer variety} $A/[\mu_2]$ of $A$.
The Kummer variety tends to be simpler than $A$ itself (e.g., for a
principally polarized abelian surface which is {\em not} split, the Kummer
is a quartic hypersurface in $\P^3$, while the homogeneous coordinate ring
of $A$ itself requires at leat 9 generators), and moreover comes with an
action of an extension $\mu_4.A[2].(G/[\mu_2])$, which we can hope is
enough (together with the known Hilbert series of the Kummer: e.g.,
$1+\sum_{n>0} (2^{d-1}(n^d+1)) t^n$ for the $2\Theta$ embedding of a
principal polarization) to uniquely determine the ideal.  (We will also
encounter one more case in which $N$ is not a reflection group but the
quotient by $N$ turns out to be a hypersurface in $\P^5$.)

\subsection{Subgroups of $2^{1+4}_-.\!\Alt_5$}
\label{ss:alt5}

The simplest instance of this approach involves the (characteristic 2)
quaternionic group $2^{1+4}_-.\!\Alt_5$ of order 1920 and its
primitive reflection subgroups of order 640 and 192.  The normal subgroup
$2^{1+4}_-$ is itself a reflection group, namely $G_2(Q_8,[\mu_2])$, and the
corresponding ring of invariants is the free polynomial ring
$\bar\F_2[x,w,y]$ with $\deg(x)=\deg(w)=1$, $\deg(y)=8$.  Thus to compute
the invariants of these primitive quaternionic reflection groups, it
suffices to determine how $\Alt_5$ acts on this ring.

The modular representation theory of $\Alt_5$ is of course well-understood,
and in particular we find that it has precisely three $2$-dimensional
representations: in addition to a trivial representation, it has two
representations coming from isomorphisms $\Alt_5\cong \SL_2(\F_4)$.  In
particular, the action on $\langle x,w\rangle$ must be one of these three
representations.  Since $\Alt_5$ must act faithfully on $\bar\F_2[x,w,y]$,
we can immediately rule out the trivial case, and the other two cases are
related by automorphisms of $\Alt_5$, so that WLOG $\Alt_5$ acts as
$\SL_2(\F_4)$ on $x$ and $w$.  Unfortunately, this does not {\em quite} pin
down the representation on $\bar\F_2[y,x,w]$: determining the action in
degree 8 is an extension problem, and nontrivial extensions exist.

Luckily, we have already computed the ring of invariants in even degree, in
Example \ref{eg:2_-^{1+4}.Alt_5} above (as a family of surfaces with an
action of $\Aut_0(E)$), and in particular see that
$\bar\F_2[x,w,y]^{\Alt_5}$ has an element $y'$ of degree 8 (this was degree
4 in Example \ref{eg:2_-^{1+4}.Alt_5}, but this was relative to twice the
principal polarization).  On the other hand, the ring
$\bar\F_2[x,w]^{\SL_2(\F_4)}$ is a polynomial ring in generators
\[
i_5 = x^4w+xw^4\qquad
i_{12} = \frac{x^{16}w+xw^{16}}{x^4w+xw^4},
\]
so does not itself contain any elements of degree 8.  It follows that $y'$
is not in $\bar\F_2[x,w]$ and thus generates $\bar\F_2[x,w,y]$ over
$\bar\F_2[x,w]$.  We may thus assume WLOG that $y=y'$, so that $y$ is
$\Alt_5$-invariant.

It then follows that for any $H\subset \Alt_5$, one has
\[
\bar\F_2[x,w,y]^H
\cong
\bar\F_2[x,w]^H[y],
\]
and thus we reduce to understanding the invariants of the {\em linear}
action of $H$.  In the three cases of interest, the classical invariants
are again a polynomial ring, and thus the main theorem holds for those
cases.  To be precise, we have the following degrees relative to the
principal polarization:
\begin{align}
  2_-^{1+4}.S_3&: 2,3,8\\
  2_-^{1+4}.\!\Dih_5&: 2,5,8\\
  2_-^{1+4}.\!\Alt_5&: 5,8,12
\end{align}
We also see that in each case $2\Theta$ gives a polynomial ring:
\begin{align}
  2_-^{1+4}.S_3&: 1,3,4\\
  2_-^{1+4}.\!\Dih_5&: 1,4,5\\
  2_-^{1+4}.\!\Alt_5&: 4,5,6,
\end{align}
as does $4\Theta$ for the last case:
\[
  2_-^{1+4}.\!\Alt_5: 2,3,5.
\]

\subsection{$ST_{12}$ and $2.S_5$}
\label{ss:ST12_normal}

The unique crystallographic lattice for $ST_{12}$ is a principally
polarized abelian surface with non-split polarization (except in
characteristic 2!), and thus its Kummer variety (embedded via $2\Theta$) is
a quartic hypersurface in $\P^3$ with a linear action of
$ST_{12}/[\mu_2]\cong S_4$.  (Moreover, $A$ is itself a Jacobian, of the
hyperelliptic curve $y^2=x^5w-xw^5$, that being the only genus 2 curve with
an action of $ST_{12}$.)

Including the action of $A[2]$ extends this to an action of a subgroup
$2^{2+4}.S_4\subset 2^{2+4}.\Sp_4(\F_2)\cong ST_{31}$.  There are four
subgroups of the given form, but only one admits a section of the map to
$S_4$, and thus the action of $\mu_4.A[2].S_4$ is uniquely determined.  The
result turns out to be (after putting the Heisenberg group $2^{2+4}$ into
standard form) a monomial matrix group, and thus it is easy to see that it
has a 2-dimensional space of invariants of degree $4$, spanned by
$x_1^4+x_2^4+x_3^4+x_4^4$ and $x_1x_2x_3x_4$.  To see which quartic in this
pencil cuts out the Kummer, we note that the natural action of $S_4$ fixes
a point of the Kummer.  There are four classes of section of the
homomorphism to $S_4$ forming a single orbit under (a) multiplying by the
sign character and (b) conjugating by the normalizer inside $ST_{31}$, and
thus WLOG the point which is fixed is $(i{:}1{:}1{:}1)$.  We thus conclude that
the Kummer has the form
\[
x_1^4+x_2^4+x_3^4+x_4^4+4i x_1x_2x_3x_4=0,
\]
and can readily check that this not only contains the given point, but is
singular there (and at the other 15 images under $A[2]$).

If we rescale the first coordinate to make the fixed point $(1{:}1{:}1{:}1)$,
then the subgroup $S_4$ simply permutes the coordinates, so that the
non-torsor problem reduces to understanding the $S_4$-invariants on
the surface
\[
x_1^4+x_2^4+x_3^4+x_4^4-4 x_1x_2x_3x_4=0.
\]
Writing this in terms of the elementary symmetric functions gives the
explicitly $S_4$-invariant form
\[
e_1^4-4e_1^2e_2+4e_1e_3+2e_2^2-8e_4=0
\]
and thus we immediately see that $e_1$, $e_2$, $e_3$ are algebraically
independent invariant functions on the Kummer, so that the quotient by
$ST_{12}$ is a weighted projective space with degrees $1$, $2$, $3$
(relative to twice the principal polarization), recovering the result of
Example \ref{eg:ST12} above.

If we instead rescale to make the fixed point $({-}1{:}1{:}1{:}1)$, then
the Kummer is still permutation-invariant, but now the permutation action
of $S_4$ does not fix any point of the surface.  In terms of the
invariants, the surface becomes
\[
e_1^4-4e_1^2e_2+4e_1e_3+2e_2^2=0
\]
or equivalently
\[
e_2^2 = e_1 (-2e_3 + 2e_1e_2-e_1^3/2)
\]
Moreover, we see that the divisor $e_1=0$ on the Kummer has multiplicity 2,
and thus is the square of a section of ${\cal L}(\Theta)$.  We thus
conclude that
\[
\sqrt{e_1},e_2/\sqrt{e_1},e_4
\]
are algebraically independent invariant elements of the homogeneous
coordinate ring coming from the principal polarization, and thus that the
quotient in this torsor case is again weighted projective, with degrees
$1,3,8$ relative to the principal polarization.

The above calculation continues to work in odd finite characteristic (the
invariants of $S_4$ work over $\Z$), but fails in characteristic 2 since
the above equation no longer represents a Kummer surface.  (This is not
surprising, as the above calculation implicitly assumed $A$ was
ordinary!)  Luckily, in that case the abelian surface has split
polarization, corresponding to the fact that $ST_{12}$ is isomorphic to an
imprimitive quaternionic reflection group of characteristic 2.  Thus in
this case we can easily write down both the homogeneous coordinate ring of
$A$ and the explicit action of the group, and thus can readily compute the
invariants up to degree 8.  We find as expected that the invariant ring is
polynomial with degrees $1,3,8$, just as in the torsor case.

There is also sporadic behavior in characteristic 5, where the quartic
surface turns out to have additional symmetries.  To be precise, the
quartic is preserved by a larger group of the form $2^{2+4}.S_5$
(corresponding to an action of $2.S_5$ on the abelian surface).  Again,
there are essentially two sections of $S_5$, one fixing a (singular) point
and one fixing a theta divisor.

For the subgroup fixing a point, the action of $S_5$ is the standard
reflection representation with invariants of degrees $2,3,4,5$, and the
equation of the Kummer becomes $e_4+e_2^2=0$, so that the quotient is again
weighted projective, now with degrees $2,3,5$.

For the other subgroup, this argument fails: the action of $S_5$ is now the
{\em dual} of the standard representation, and thus no longer has
polynomial invariants.  This does not in itself contradict polynomiality of
the invariant ring of the quotient, but after passing to the principal
polarization, we find that the invariant ring has Hilbert series
\[
\frac{1+t^{15}+t^{30}}{(1-t)(1-t^9)(1-t^{40})}.
\]
That is, by an explicit computation in degree 40, we find that the the
invariant ring is generated to that degree by elements of degrees $1$, $9$,
$15$, and $40$, the invariants of degrees $1$, $9$, and $40$ are
algebraically independent, and thus by degree considerations the invariant
ring is a cubic extension of the ring they generate.  This implies that the
quotient is a hypersurface of degree 45 in a weighted projective space with
degrees $1$, $9$, $15$, and $40$.  (To see that this is the strongly
crystallographic bundle, note that degree considerations imply that the
degree 1 invariant does not vanish on any reflection hypersurface and thus
every reflection acts as a strong reflection on ${\cal L}(\Theta)$.)

We thus find that the torsor case of $2.S_5\subset \GL_2(Q_8)$ is one of
the counterexamples to the main theorem!

\subsection{$ST_{24}$ and $\pm \PGU_3(\F_3)$}
\label{ss:ST24}

For the $ST_{24}\cong \pm \PSL_2(\F_7)$ case\footnote{Our approach is
similar in spirit to the approach of \cite{MarkushevichD/MoreauA:2022b},
except that we work with an {\em algebraic} description rather than an
analytic description as done op.~cit.}, the crystallographic lattice is a
principally polarized abelian $3$-fold with non-split polarization.  Such a
3-fold is the Jacobian of a curve, which (at least in characteristic 0)
cannot be hyperelliptic, lest it induce (via the action on the quotient
$\P^1$) a $2$-dimensional projective representation of $\PSL_2(\F_7)$.
(The only 2-dimensional projective representation of $\PSL_2(\F_7)$ is the
obvious one in characteristic 7.)  In particular, we see that $A$ must be
the Jacobian of a plane quartic invariant under some $3$-dimensional
faithful projective representation of $\PSL_2(\F_7)$, and thus conclude
that the curve must be the Klein quartic $x^3y+y^3z+z^3x=0$.  (This fails
in characteristic 7, where the curve becomes the hyperelliptic curve
$y^2=x^7w-xw^7$.)

Although in principle we could use this to explicitly compute the Kummer,
we can also compute the Kummer directly as in the $ST_{12}$ case.  The idea
here is that the action of the complex multiplication naturally splits the
$2$-torsion into two summands: the kernels of $(-1\pm \sqrt{-7})/2$.  Thus
in the action of $\mu_4.A[2].\PSL_2(\F_7)$ on the sections of ${\cal
  L}(2\Theta)$, we may diagonalize one of the summands and make the other
summand act by permutations, at which point $\PSL_2(\F_7)\cong
\PSL_3(\F_2)$ itself acts by permutations.  In particular, we again obtain
a monomial action.

One difficulty here is that the Kummer is no longer a hypersurface, so that
we cannot simply find an invariant of the appropriate degree.  However, in
characteristic 0 the Jacobian of a non-hyperelliptic curve is known to be
the {\em singular locus} of a hypersurface, the ``Coble quartic''
(originally constructed in \cite{CobleAB:1961}; see also
\cite{BeauvilleA:2003} for a treatment that extends cleanly to odd
characteristic).  This cuts the problem down to a $\P^2$ worth of candidate
invariant quartics, which can be reduced to $\P^1$ by observing that the
quartic must be singular at the image of the identity of $A$, which must be
invariant under the permutation group $\PSL_2(\F_7)$ and thus of the form
$(a{:}1{:}1{:}1{:}1{:}1{:}1{:}1)$.  (One must further have $a\ne 1$ since
it should have trivial stabilizer for the action of $A[2]$.)  This gives
two equations (vanishing of derivatives in the two orbits) and thus
generically gives a unique candidate quartic with coefficients polynomial
in $a$.  (This in fact works over $\Z[a,1/14(a-1)]$.)  We can further
reduce the search space by noting that the point must in fact be singular
on the 3-dimensional Kummer, which implies that the Hessian of the quartic
must have rank less than 5 at the point.  The minors of the Hessian give
polynomials in $a$ that must vanish, and we find that there are precisely
three common irreducible factors of those minors:
\[
a-1, a^2+7, 2a^2+7a+7.
\]
We have already ruled out the case $a=1$, while the case $a=\sqrt{-7}$
fails because it makes the quartic a square.  We thus conclude that
$a=\frac{-7\pm \sqrt{-7}}{4}$.  This gives two possible choices for the
quartic, but the results are actually equivalent.  Indeed, when normalizing
the group action, we chose one of the two CM-invariant subgroups of the
$2$-torsion, but those subgroups are only defined over $\Q(\sqrt{-7})$, so
that we actually expect to get two Galois-conjugate solutions.  And, of
course, the two conjugates will give us the same invariant theory!
Moreover, if we compute the Gr\"obner basis of the singular locus over
$\Q(\sqrt{-7})$, we find that the coefficients of the resulting ideal are
integral away from the primes over $2$ and $7$, and thus (since the Kummer
is integral!) this is the correct Gr\"obner basis over $\Z[\sqrt{-7},1/14]$.

In good characteristic (i.e., away from $2$, $3$, $7$), we can compute the
invariants by first computing generating invariants of the group on the
ambient polynomial ring, then looking for relations modulo the ideal (and
in particular for expressions of invariants as polynoimals in other
invariants).  Since we are expecting to get a polynomial ring with degrees
independent of the characteristic, we can instead find the invariants in
characteristic 0 and then hope to show that they (or suitably massaged
versions thereof) remain algebraically independent modulo the primes of
interest.  Finding the invariants in characteristic 0 is straightforward,
with the only real complication being that we are working with the square
of the correct line bundle, so actually get a hypersurface.  We conclude
that the invariant ring in characteristic 0 is polynomial with invariants
of degrees $1,2,4,7$.

To descend to finite characteristics, we need to show that the generating
invariants remain algebraically independent, say by showing via elimination
that the invariants of degrees $1,1,2,4$ in the $2\Theta$ ring remain
algebraically independent.  The corresponding Gr\"obner basis calculation
over $\Q(\sqrt{-7})$ is almost valid over $\Z[\sqrt{-7},1/14]$, except that
occasionally it divides by a non-unit.  If we adjoin an inverse every time
that happens, we will get a finitely generated subring of $\Q(\sqrt{-7})$
over which the claim remains valid.  In other words, this gives us an
explicit finite collection of primes away from which the invariant ring
remains polynomial!  We then need simply check that the calculation remains
valid at the remaining primes.  As stated, this is not quite feasible, as
it requires an implementation of Gr\"obner bases over $\Q(\sqrt{-7})$ that
keeps track of divisions.  However, this can be finessed by observing that
the Kummer descends to $\Q$ (indeed, the Klein quartic can be defined over
$\Q$) and using this to do the calculation over $\Q$ instead.

In fact, as implemented, we finesse things further, by using the
description as a Jacobian to come up with an alternate rational model of
the Kummer and check by comparing to the above model that it remains valid
over $\Z[1/14]$.  The point is that the Jacobian of a plane quartic is
birational to the moduli space of representations of the quartic as a
$2\times 2$ determinant of quadrics.  (Such a representation gives a line
bundle by taking the cokernel of the corresponding morphism of vector
bundles.)  This involves a quotient by $\SL_2\otimes \SL_2\cong \SO_4$, and
the Kummer is the quotient by $\GO_4$, so the invariant theory is completely
classical.  The resulting 3-fold in $\P^6$ is missing one section of
$2\Theta$, but it is not {\em too} difficult to guess what that section
must be and thus give a putative model of the Kummer.  One then finds that
there is a unique quartic singular along the putative Kummer and can find
(using the linear dependences between the singular points to pin down a
bijection between the singular points that comes from $\PGL_8$) an explicit
linear transformation putting the quartic into standard form.  This makes
the standard invariant theory calculation somewhat more complicated (since
now $\PSL_2(\F_7)$ has coefficients in $\Q(\zeta_7)$), but simplifies the
rest of the calculation.  In particular, the initial elimination turns out
to be valid away from nine small primes, and the result remains valid at
the seven of those primes that are not $2$ or $7$.  ($7$ fails because the
curve is hyperelliptic, while $2$ fails because we twisted the automorphism
that swaps the subgroups $\mu_2^3$ and $\Z/2\Z^3$ of the $2$-torsion.)
Note that although this does involve computation, it is not a particularly
arduous computation, as the entire calculation takes around $32$
seconds.\footnote{In Magma on a 2021 Macbook Pro.}

For characteristic 2, we note that the standard form of the Kummer has good
reduction modulo one of the two primes over $2$ (the abelian variety has
good reduction and the 2-torsion group we diagonalized reduces to
$\mu_2^3$), and indeed the reduction of the invariant quartic modulo that
prime very nearly has the correct singular locus, differing only by some
embedded components.  Removing those embedded components gives a scheme
with the correct Hilbert series, which must therefore be the desired
reduction of the Kummer.  We can then proceed by a direct calculation to
find the invariants of degree up to 7 and verify that we have the same
hypersurface structure as in characteristic 0.

In characteristic 7, the invariant quartic becomes $(\sum_i x_i^2)^2$,
reflecting the fact that the curve becomes hyperelliptic.  However, we can
avoid this issue by adjoining a new generator $y$ (of degree 2) of the form
$(\sum_i x_i^2)/\sqrt{-7}$ to the characteristic 0 ring.  In the weighted
grevlex ordering with $y$ appearing after the original variables, the
resulting Gr\"obner basis is now valid over $\Z[\sqrt{-7},1/2]$, and thus
the family is flat over $7$, giving the desired reduction of the Kummer.
We can then proceed as usual to verify that we still get polynomial
invariants.

As in the $ST_{12}$ case, there is an additional sporadic case to consider:
in characteristic 3, the Klein quartic is isomorphic to the Fermat quartic,
and thus has symmetry group $\PGU_3(\F_3)$.  We find by the usual
calculation that the invariant ring in the non-torsor case is polynomial of
degrees $1,4,7,9$ (relative to $2\Theta$).  For the torsor case, we find
(by a version of the bootstrapping method considered below) that not only
is the invariant ring not a polynomial ring, but it is apparently not even
Cohen-Macaulay!

\subsection{$2.\!\Alt_5$ and $2.\!\Alt_6$}
\label{ss:alt6}

For surfaces of dimension $>3$ or with non-principal polarizations, there
is relatively little known about the form of the ideal of the Kummer.
There is, however, one more case in which the representation theory of the
group allows us to determine the ideal, namely the characteristic 3 abelian
variety on which the quaternionic reflection group $2.\!\Alt_6$ acts.  In
this case, the polarization has degree 3, and thus the Kummer has Hilbert
series
\[
1+\sum_{n>0} (6n^2+2) t^n
\]
(Using the other eigenspace of $[-1]$ in $\Gamma(2\Theta)$, we actually get
a birational model of the Kummer as the Fermat quartic in $\P^3$, but one
must contract 16 $-2$-curves to get the true Kummer, and it is difficult to
construct the requisite line bundle in an equivariant way.)  We thus see
that we have 8 generators in degree 1 (relative to $2\Theta$) on which
$\mu_4.A[2]$ acts as the sum of two copies of the usual $4$-dimensional
irreducible representation.  It follows that $\mu_4.A[2].\!\Alt_6$ must
respect the given tensor product structure, and since the map to $\Alt_6$
has a section (while the $4$-dimensional projective representation of
$A[2].\!\Alt_6$ does not!), we conclude that the projective action on the
$2$-dimensional tensor factor must come from an isomorphism
$\Alt_6\cong\SL_2(\F_9)$.  Although there are four such isomorphisms and
thus in principle four such representations, they are actually isomorphic:
one outer automorphism of $\Alt_6$ comes from $\GL_2(\F_9)/\SL_2(\F_9)$ so
doesn't change the representation, while Galois conjugation on
$\SL_2(\F_9)$ can be transported to the action of $\Sp_4(2)/\!\Alt_6$ on the
4-dimensional tensor factor, so again doesn't change the representation.

Having determined how the group acts, we then note that since there are
$26$ invariant sections of $4\Theta$ and $36$ quadratic polynomials in the
generators, the ideal must contain at least 10 quadratic relations.  (It
could contain more, as there is no a priori reason to expect the
homogeneous coordinate ring of the Kummer to be generated in degree 1.)
However, it turns out that every $G$-invariant space of quadrics of
dimension at least 10 contains the {\em same} 10-dimensional space of
quadrics.  Moreover, there are two minimal supersubmodules, but one of them
gives a finite quotient algebra while the other gives a geometrically
reducible quotient.  In particular, in neither case can we get an integral
algebra by adjoining generators and relations of higher degree.  This tells
us that the 8 invariant sections of $2\Theta$ satisfy precisely $10$
relations, and thus that all 26 invariant sections of $4\Theta$ are
contained in the ring.  Since we are in odd characteristic, we can then
refer to \cite[Thm.~1.4]{SasakiR:1981} to conclude that the ring is generated in
degree 1 after all!  Moreover, the 10 relations give a ring with the
correct Hilbert series, and thus we cannot be missing any relations.  We
have thus completely determined the ideal of the Kummer!

We can then proceed as usual.  For the subgroup $\Alt_5$ (corresponding to
the quaternionic reflection group $2.\!\Alt_5$), we find that we have
invariants of degrees $1,2,3,5$ satisfying a relation of degree 6 and such
that the invariants of degrees $1,3,5$ have trivial intersection on the
Kummer.  Moreover, the relation is a square modulo the degree 1 invariant,
and thus we can pass to $\Theta$ to get independent invariants of degrees
$1,4,5$ as required.  For the subgroup $\Alt_6$ fixing the identity, we
have invariants of degrees $2,3,5$ with trivial intersection, which again
give a polynomial ring by degree considerations (of course, we already knew
this by Example \ref{eg:2_-^{1+4}.Alt_5} above!).

For the torsor version of $\Alt_6$, we again fail to get a polynomial ring;
the natural equivariant structure on ${\cal L}(\Theta)$ has
non-polynomial (but hypersurface) invariant ring with Hilbert series
\[
\frac{1+t^{15}}{(1-t)(1-t^{10})(1-t^{24})},
\]
and again we note that the degree 1 invariant does not vanish on any
reflection hypersurface, so this is the strongly crystallographic bundle.

\subsection{$ST_{31}$ and $ST_{29}$}
\label{ss:ST31}

For abelian varieties of dimension $>3$, even principally polarized Kummers
tend to have too many generators to allow us to reconstruct the equivariant
ideal; in fact, even {\em writing down} the invariants can become
cumbersome.  For instance, we know from a surface calculation that
one of the generating invariants of $ST_{31}$ is a section of $10\Theta$,
but there are $5008$ sections of this bundle on the Kummer, and thus we
need to specify $5008$ coefficients to determine that invariant.  (And this
becomes far worse for $ST_{34}$, where we will see that we have too many
odd-degree invariants to readily reconstruct from the Kummer, and the
largest invariant lies in a 2985984-dimensional space of sections!)

Luckily, for $ST_{31}$ and $ST_{29}$, we have a normal subgroup $2^{2+4}$
which is large enough to cut down the space of invariants considerably, but
small enough to not require invariants of large degree.  This still leaves
the problem of computing the quotient by the normal subgroup, as again it
is neither a reflection group nor one for which quotients have been
otherwise understood.  At first blush, this appears to require us to
compute the Kummer.  It appears that the group action again gives enough
information to uniquely determine the Kummer (as in the previous section),
but it is cumbersome even to write down the equations lying in the relevant
subrepresentations (as the most natural approach involves summing over the
group), and unclear whether the ensuing Gr\"obner basis calculation could
be done in a reasonable time.

If we are willing to forgo the full action of the group, we can resolve
this problem.  The point is that, although $ST_{31}$ and $ST_{29}$ both
have unique lattices, the normal subgroup $2^{2+4}$ also has a lattice
$E^4$ with split polarization.  We may thus write the true abelian variety
$A$ as the quotient of $E^4$ by an extension $T.2^{2+4}$ where $T$ is a
$2$-torsion group of order 16.  (To be precise, $T$ is generated by
elements $(t_1,t_1,0,0)$, $(t_1,0,t_1,0)$, $(t_1,0,0,t_1)$ and
$(t_2,t_2,t_2,t_2)$, where $t_1\in E[1-i]$ and $t_2\in E[2]\setminus
E[1-i]$.)  Since this makes the group imprimitive, the action of $2^{2+4}$
on $E^4$ is straightforward to compute (we simply permute the coordinates
and act by suitable elements of $[\mu_4]$), and thus it is easy to find
invariants.  (We can, moreover, embed $E$ in $\P^{[1,1,2]}$ via the divisor
$E[1-i]$, so that it becomes the curve $u^2=v^4-w^4$ and the action in
those coordinates is itself monomial.)

In particular, we can readily find the 16 sections of $2\Theta$ on $A$ as
the relevant invariants of $T$ acting on $E^4$.  We find (by linear algebra
on monomials of degree $\le 4$) that they satisfy 10 quadratic relations
and 505 additional quartic relations, and can check that the resulting
quotient has {\em almost} the correct Hilbert series; it agrees with the
Hilbert series of the Kummer in all but degree 2.  In particular, since we
have certainly found all {\em cubic} relations, we see that the ring
generated by sections of $2\Theta$ contains all sections of $6\Theta$ and
thus (again using \cite[Thm.~1.4]{SasakiR:1981} and avoiding characteristic
2) of all higher degree.  In other words, this ideal {\em does} cut out the
Kummer, but the model is not projectively normal.  (It is missing precisely
10 sections of $4\Theta$.)  But this is good enough for our purposes: since
the true invariant ring is normal, we can recover it as the normalization
of the invariant ring in our non-normal model.

We next wish to compute the invariants of the residual action of
$T.2^{2+4}/T.[\mu_2]$.  This is an elementary abelian $2$-group and thus can
be simultaneously diagonalized over $\Z[1/2]$, and in that basis the
problem of finding generators of the invariant ring becomes purely
combinatorial.  (I.e., finding generators of the monoid of monomials in
which the degrees satisfy various constraints modulo 2.)  We thus see that
the generating invariants remain valid in any odd characteristic.
Moreover, we can check that the invariants of degree $>1$ can be expressed
modulo the ideal as polynomials in the linear invariants (with everything
working over $\Z[1/2]$), and thus the invariant ring is generated (apart
possibly from degree 2) by those 6 linear invariants.  Dimension
considerations tell us that these invariants must satisfy a single relation
(the putative missing degree 2 elements have no effect on this), which must
have degree 6.  We can thus solve for this relation to obtain the desired
quotient hypersurface $X\subset \P^5$.

Having done so, the problem is to determine the residual action of $S_6$,
or more generally the residual (projective) action of $A[1-i].S_6$.  The
image of the identity has multiplicity 4 on $X$, and we readily find (in
characteristic 0) that $X$ has precisely 16 points of multiplicity 4.  If
we guess that the projective representation of $A[1-i].S_6$ is the usual
reflection representation of $D_6$, then there are two orbits of 16 points
(swapped by $C_6/D_6$), and we can find a putative linear bijection between
the two sets of points by asking for the sets of four linearly dependent
points to be preserved.  It is then straightforward to solve for the linear
map and verify that $X$ is indeed equivalent over $\Z[1/2]$ to a
$D_6$-invariant sextic, namely
\[
i_2^3 - 4 i_2i_4 + 8 i_6 - 16 e_6=0,
\]
where $i_{2k} := e_k(x_1^2,\dots,x_6^2)$.  (The entire calculation to this
point is again quite fast, requiring less than 3 seconds!)

For the quotient by $ST_{29}/2^{2+4}\cong S_5$, we can either directly
express this in terms of $S_5$-invariants or note that the normal
2-subgroup of $D_5$ surjects on $A[1-i]$ and thus gives the action of
$ST_{29}$ on the quotient.  The advantage of the latter is that it actually
gives us the {\em principal} polarization directly.  We in particular find
that the degree 6 invariant of $D_5$ can be eliminated using the equation
of $X$ and thus that we get invariants of degrees $1,2,4,5,8$.  (Similarly,
we can clearly use the equation to eliminate a degree 6 invariant of $D_6$,
recovering degrees $2,4,6,8,10$ relative to $\Theta$ as we saw via the
surface argument.)

For the torsor action of $ST_{31}$, the quotient acts by the ``other''
$S_6$ in $D_6$ (i.e., the conjugate of the obvious $S_6$ by an element of
$C_6/D_6$), and we find that $e_6$ drops out from the corresponding
expression for $X$.  However, $X$ is still a square modulo $e_1$ and thus
as usual we can adjoin a square root of $e_1$ to see that the torsor case
still gives a weighted projective space, with degrees $1,4,5,8,12$
(agreeing with our guess based on the surface calculation in characteristic
2).

Note that the above calculation fails in characteristic 2, and indeed it
appears that the quotient by $T.2^{2+4}$ is not nearly as nice in
characteristic 2.  (This should not be {\em too} surprising, as the group
$D_6/\mu_2$ reduces to $\alpha_2^4.S_6$.)  We will thus need to use
alternate means to deal with $ST_{29}$ in characteristic 2 (having already
dealt with $ST_{31}$ via a surface argument).

\section{Indirect calculations}

\subsection{Bootstrapping from non-normal subgroups}

In the remaining cases, none of the normal subgroups have particularly nice
(or particularly computable) invariants, and thus we need an alternate
approach.  Part of the benefit of working with the Kummer above was that it
made the linear algebra problem of finding invariants of a given degree
easier by making the space of potential invariants smaller.  This advantage
applies to {\em any} subgroup, even if that subgroup fails to be normal.
Of course, when the subgroup isn't normal, the group no longer acts on the
space, but we can still find equations satisfied by the invariants in many
cases.

The basic idea is as follows.  Let the group scheme $G$ act on a variety
$X$ (both defined over a field $k$), and suppose $V\subset k(X)$ is a
finite-dimensional space of {\em functions} on $X$ which are invariant
under a subgroup $H\subset G$.  Suppose moreover that $G$ is generated over
$H$ by an \'etale subscheme $S$.  Then a function in $V$ is $G$-invariant
iff it satisfies $f(x) = f(s(x))$ for every closed point $(x,s)$ of
$X\times S$.  In particular, each point of $X\times S$ cuts out a subspace
of $V$ (which may have codimension $>1$ if $(x,s)$ is defined over an
extension of $k$), and the intersection of that infinite family of
subspaces is precisely the space $k(X)^G\cap V=:W$.  But we are dealing
with finite-dimensional spaces, and thus there exists a {\em finite}
collection of points $(x,s)$ that gives the correct intersection.

If $k$ is a large finite field $\F_q$, an obvious approach is to choose
{\em random} points $(x,s)$.  Heuristically, the corresponding equations
should behave randomly, and thus after accumulating $\dim(V)-\dim(W)$ such
equations, the probability that the equations fail to cut out $W$ would be
$O(1/q)$.  Choosing more points than needed improves this considerably, and
we note the following bound (which should be applied to $(V/W)^*$); this
gives a bound depending only on the number of excess equations.

\begin{lem}
   For all $m,n\ge 0$, the probability that $m+n$ uniformly chosen random
   points of $\F_q^m$ fail to span $\F_q^m$ is bounded above by
   $q^{-n}/(q-1)$.
\end{lem}

\begin{proof}
   Let $v_1,\dots,v_{n+m}$ be the uniform random points.  If they fail to
   span, then there is a linear functional $\lambda$ that annihilates all
   of the points.  For any given linear functional, the probability that it
   annihilates $v_1$ is $1/q$, and thus by independence the probability
   that it annihilates all of the points is $1/q^{m+n}$.  Modulo scalars,
   there are $(q^m-1)/(q-1)$ nonzero linear functionals, and thus the
   expected number of functionals vanishing on the points is
   $(q^{-n}-q^{-m-n})/(q-1)$.  If $N$ is the random variable giving the
   number of such functionals, then
   \[
   E(N) = \sum_{k\ge 1} k \Pr(N=k)
        = \sum_{k\ge 1} \Pr(N\ge k)
        \ge \Pr(N\ge 1),
   \]
   and thus the probability is bounded above by
   \[
   \frac{q^{-n}-q^{-m-n}}{q-1}\le \frac{q^{-n}}{q-1}.
   \]
\end{proof}

Of course, in our situation, this is only heuristic (there are too few
random points in general to let the corresponding equation be even close to
uniformly distributed), though certainly good enough for exploratory
purposes.  The problem, of course, is that even though this lets us compute
$V\cap k(X)^G$ in practice, there is always {\em some} positive probability
that some spurious functions survive.  Luckily, this can be finessed, as we
will explain below.

One important caveat in the above is that the approach requires a space of
{\em functions}, while our usual question is about invariant sections of
line bundles.  Thus to find all invariant sections of a line bundle, we
will first need {\em one} invariant section of the line bundle (which may
be meromorphic).  Of course, an invariant section of a line bundle induces
invariant sections of all powers of that line bundle, so that the question
essentially reduces to finding an invariant section of the minimal
equivariant line bundle.

Another caveat is that we need to be able to compute pairs
$(\phi(x),\phi(gx))$ for $x\in A$, which is a problem in general since the
cases above with explicit formulas for generating invariants (imprimitive
cases and $A_n$, though surface cases could in principle be made explicit
as well) tend to give explicit formulas in terms of some {\em other}
lattice for the group.  (Indeed, $A$ itself is often not a crystallographic
lattice for $H$, and in the imprimitive cases the formulas are based on a
lattice coming from the system of imprimitivity.)  Thus $\phi$ is actually
expressed as a function on some {\em other} abelian variety $B$ (i.e.,
$E^n$ in the imprimitive cases and the sum zero locus in $E^{n+1}$ for
$A_n$) with an $H$-equivariant isogeny $B\to A$.  Suppose this isogeny has
kernel of exponent $N$.  Then there is a natural quasi-inverse isogeny
$A\to B$ that such that the compositions $B\to A\to B$, $A\to B\to A$ are
multiplication by $N$, and for any $g\in \Aut_0(A)$, we can define an
endomorphism $Ng$ of $B$ as the composition $B\to A\xrightarrow{g} A\to B$.
We then claim that for any point $x\in B$, the pullback to $B$ of any
invariant $\phi$ will satisfy $\phi(Nx)=\phi(Ngx)$.  Indeed, the pullback
of $\phi$ is a composition
\[
B\to A\to A/H\ratto \A^1
\]
and thus $\phi\circ [N]$ can be written as the composition
\[
B\to A\to B\to A\to A/H\ratto \A^1,
\]
while $\phi\circ Ng$ can be written as the composition
\[
B\to A\xrightarrow{g} A\to B\to A\to A/H\ratto \A^1.
\]
Since the composition $A\to B\to A$ is multiplication by $N$, which
commutes with $g$, we can rewrite this as
\[
B\to A\to B\to A\xrightarrow{g} A\to A/H\ratto \A^1,
\]
from which invariance follows immediately.  Moreover, since the composition
$B\to A\to B\to A$ is geometrically surjective, this doesn't actually
change the set of equations we get!  (It {\em does} make the equations
coming from random points even less random, but this is not an issue in
practice.)

In fact, we can often make this work even if the functions $\phi:B\ratto
\A^1$ are not invariant under the kernel of $B\to A$.  If we compose with
the quasi-inverse $A\to B$, we get functions on $A$ which are invariant
under a group $T\rtimes H$ with $T\subset A$.  If $A[N]\rtimes G$ is
generated by an \'etale set over $T\rtimes H$ (e.g., if $\langle
GT\rangle=A[N]$), then we can still use random points to find equations on
spaces of $T\rtimes H$-invariant functions cutting out $A[N]\rtimes
G$-invariant functions.  Directly applying the above reduction gives
equations of the form $\phi(N^2x)=\phi(N^2gx)$, but since $[N]$ is
surjective, this implies $\phi(Nx)=\phi(Ngx)$ still holds for all $x$ iff
$\phi$ comes from a $G$-invariant function on $A$.

The advantage of this is that it is not always feasible to compute the
action of $\ker(B\to A)$ on $B/H$; we have seen how to do this for $H=A_n$,
but it is more subtle for the imprimitive cases, especially when $B\to A$
is inseparable.  Of course, the cost is that the space of functions on
which we compute the equations tends to be larger by a factor of
$|\ker(B\to A)|$, but this is a surprisingly small cost in practice, to the
point that it is usually not worth computing the $\ker(B\to A)$-invariants
even when we {\em can} compute the action.

One non-computational application of this approach is to give conceptual
proofs of flatness of invariants in some bad characteristics.  The point is
that there is a {\em relative} Reynolds operator projecting from
$H$-invariants on $B$ to $G$-invariants on $A$, which makes sense as long
as $|\ker(B\to A)|[G:H]$ is invertible.  So once we have shown that the
ring of $H$ invariants is flat, it will follow immediately that the ring of
$G$-invariants is flat whenever $|\ker(B\to A)|[G:H]$ is invertible.

\subsection{$ST_{29}$ in characteristic 2}
\label{ss:ST29_2}

We first explain how to apply this approach to the characteristic $2$ case
of $ST_{29}$.  The first thing we need is a subgroup with computable
invariants, which in practice means either an imprimitive subgroup or one
of type $A$.  (In general, this is just to get started; in several cases
below, we will use a chain of subgroups and only need the first one to have
directly computable invariants.)  In this case, all of the maximal
reflection subgroups of $ST_{29}$ are of this form, but in two of the cases
the associated lattice is inseparable over the lattice we want, so for
simplicity we work with one of the two (Galois conjugate) subgroups of type
$A_4$.  The kernel $K$ of the isogeny $E\otimes \Lambda_{A_4}\to A$ is one
of the two eigenspaces of $[i]$ on $E[5]$ (acting diagonally on the
sum-zero subvariety of $E^5$).  The action of the kernel on the coordinates
of $E\otimes \Lambda_{A_4}\cong \P^4$ is the dual of its action on the
functions with poles along the kernel, so that it is straightforward to
compute the $K$-invariant polynomials of any given degree $d$.

In particular, we can apply this to degree $1$, where we find that there is
a unique $K$-invariant section $i_1$.  By the calculation in Subsection
\ref{ss:complexsurf1728}, we know that the group $ST_{31}$ has an invariant
section of degree 1, and thus so does its subgroup $ST_{29}$, so that this
unique $K$-invariant section is actually $ST_{29}$-invariant.  This lets us
apply the boostrapping method in any degree: find the $K$-invariant
polynomials of degree $d$, divide them by $i_1^d$, and then find the
subspace of $ST_{29}$-invariant functions.  (We can if we like improve this
slightly by noting that $A_4$ is normalized by $[i]$ and thus we may work
with the space of $K\rtimes [i]$-invariant polynomials, cutting the
dimension by a factor of $~4$.)

Applying this to degree $8$ gives (with very high probability, assuming we
use enough points) a 12-dimensional space, from which we can recover (by
pulling out factors of $i_1$) corresponding spaces in lower degrees.  We
thus find that the functions are polynomials in generators of the form
$i_2/i_1^2$, $i_4/i_1^4$, $i_5/i_1^5$, $i_8/i_1^8$ with at most an 8-th
order pole along $i_1^8$.  Moreover, the polynomials $i_1,i_2,i_4,i_5,i_8$
have trivial intersection on $\P^4$, implying that they have trivial
intersection on $A$ as well.  Thus either the main theorem continues to
hold for $ST_{29}$ in characteristic 2, or we were unlucky enough to have
spurious solutions survive.

In fact, we can rigorously rule out the latter possibility!  The point here
is that we already know that the theorem applies to $ST_{29}$ in
characteristic 0, and thus in characteristic 0 the invariant ring has
Hilbert series $1/(1-t)(1-t^2)(1-t^4)(1-t^5)(1-t^8)$, so that the space of
degree $8$ invariants has dimension $12$.  But semicontinuity then tells us
that the space of degree $8$ invariants in characteristic 2 has dimension
{\em at least} 12; our 12-dimensional subspace certainly {\em contains} the
true subspace, and thus must be correct!

\subsection{Dimensions}

What made the above calculation work is that we had a prior information
about the dimension of the true invariant space, allowing us to prove that
the probable invariants were actually invariant.  For polynomial invariants
in characteristic 0, character theory allows one to perform an a priori
computation of the Hilbert series of the invariant ring.  Although it is
unclear how to perform an analogous calculation in the abelian variety
setting, it turns out that we can often compute {\em some} of the
coefficients of the Hilbert series directly.  (In
\cite[\S3]{MarkushevichD/MoreauA:2022b}, the full Hilbert series for the
$ST_{24}$ case was computed; their method could most likely be applied in
other characteristic 0 cases, but our approach is easier in the degrees to
which it applies.)

In general, given a group $G$ acting on an abelian variety $A$ with (ample)
equivariant line bundle ${\cal L}$, the action of $G$ on sections of ${\cal
  L}$ extends to an action of a larger group $\G_m.\ker(\lambda).G$ where
$\lambda:A\to \Pic^0(A)$ is the isogeny corresponding to ${\cal L}$.
(I.e., $\lambda(x) = \tau_{x*}{\cal L}\otimes {\cal L}^{-1}$, where
$\tau_x$ is the corresponding translation.)  The finite group scheme
$\ker(\lambda)$ has order $\dim(\Gamma(A;{\cal L}))^2$ and the action of
$\G_m.\ker(\lambda)$ is the unique irreducible representation with $\G_m$
acting in the obvious way.  Moreover, we can decompose $\ker(\lambda)$ as a
product $\prod_l T_l$ over primes and the representation of
$\G_m.\ker(\lambda).G$ splits as a tensor product of representations of
$\G_m.T_l.G$.

If $T_l$ is an elementary abelian $l$-group with $l$ odd, then this
representation is the restriction of a natural (and well-studied, starting
with \cite{WeilA:1964}) representation of $\G_m.l^{2n}.\Sp_{2n}(l)$.  This
representation comes with a natural subgroup $\G_m.\Sp_{2n}(l)$
(corresponding to the centralizer of the essentially unique element of
order 2 lying over $[-1]$) which in turn has a subgroup $\Sp_{2n}(l)$
(which is unique unless $n=1$, $l=3$), and one has explicit formulas for
the character \cite{GerardinP:1977}.  In particular, the action of $G$ on
the tensor factor lies in $\G_m\times \Sp_{2n}(l)$, and thus its
representation can be obtained from the restriction of the natural
representation of $\Sp_{2n}(l)$ by twisting by a character.  (Of course,
when we have multiple tensor factors, we can combine the characters to a
single global choice.)

Now, suppose that $A\cong E^n$ with $E$ a curve in characteristic 0 with
complex multiplication, and suppose that $l$ splits in $\End(E)$.  Then $E$
admits an endomorphism $\psi$ of degree $l$, and $A[l]$ splits
(equivariantly) as $\ker\psi^n\oplus \ker\bar\psi^n$.  If we use that
splitting to rigidify the action of $\G_m.A[l]$ on the relevant tensor
factor of $\Gamma(l\Theta)$ (assuming that $l$ is prime to the degree of
the invariant polarization $\Theta$), then $\G_m.A[l].G$ becomes a
semidirect product in which the splitting of $G$ acts as permutations of
the coordinates (in the same way that it acts on $\ker\psi^n$).  This
agrees up to a twist with the above splitting, and thus the character of
$G$ on the tensor factor has the form (assuming $l$ is also a good prime)
\[
\chi_l(g) =
|\{x\in \ker\psi|gx=x\}|
=
|\ker(g|_{\ker\psi}-1)|
=
l^{\dim(\ker(g\otimes\F_l-1))}
=
l^{\dim(\ker(g-1))},
\]
possibly multiplied by a 1-dimensional character of $l$-power order.  (This
also works when $E$ has endomorphism ring $\Z$, in that any splitting of
$E[l]$ extends to a $\GL_n(\Z)$-equivariant splitting of $E^n[l]$ which
continues to work on isogenous varieties as long as the isogeny has degree
prime to $l$.)

Thus in the complex and real cases, if $l_1$,\dots,$l_m$ are (split) primes
not dividing twice the degree of the polarization, the order of $G$, or (in
the real case) the degree of any relevant isogeny, then the character of
$G$ on $\Gamma(A;l_1\cdots l_m\Theta)$ has the form
\[
\Bigl(\prod_{1\le i\le m} l_i\Bigr)^{\dim(\ker(g)-1)} \chi_\Theta(g)
\]
where $\chi_\Theta$ is an unknown character of degree $\deg(\Theta)$.  In
fact, this character is independent of the chosen set of primes!  The key
observation is that $\G_m.A[l_m].G$ has a natural subgroup
$\mu_{l_m}.A[l_m].G$ and the splitting of $A[l_m]$ induces a subgroup
$\mu_{l_m}.\ker\psi_m.G$.  The abelian group $\mu_{l_m}.\ker\psi_m$ can be
simultaneously diagonalized, and each eigenvector is stabilized by a
subgroup $\ker\psi_m\rtimes G$.  The homogeneous coordinate ring of the
quotient $A/\ker\psi_m$ is nothing other than the $\ker\psi_m$-invariant
subring of the homogeneous coordinate ring of $A$, but $A/\ker\psi_m\cong
A$, so this recovers the original ring.  The different choices of
eigenvector correspond to different choices of line bundle on the quotient,
and thus there is a canonical choice corresponding to the pullback of
$l_1\cdots l_{m-1}\Theta$.  This represents
$\Gamma(A;l_1\cdots l_{m-1}\Theta)$ as a $G$-invariant subspace of
$\Gamma(A;l_1\cdots l_m\Theta)$ in a way that respects the tensor product
decomposition, and thus by induction the subspace corresponding to
$\Gamma(A;\Theta)$ is isomorphic to the corresponding tensor factor.

We thus see that in the complex case, to determine the dimension in degrees
that are products of split odd primes, it suffices to determine the
character $\chi_\Theta$ of $G$ on $\Gamma(A;\Theta)$, since then
\[
\dim\Gamma(A;l_1\cdots l_m\Theta)^G
=
\frac{1}{|G|}\sum_{g\in G} (\prod_{1\le i\le m} l_i)^{\dim(\ker(g)-1)}
\chi_\Theta(g).
\]

One natural approach to computing $\chi_\Theta$ is to directly compute the
action of $G$ on $\Gamma(A;\Theta)$, or better yet on the reduction to some
good characteristic.  Of course, we have the usual problems with computing
$\Gamma(A;\Theta)$ in an equivariant way, but given any representation of
$\Gamma(A;\Theta)$, we can use a variation on the random points method to
compute the action of $G$.  To be precise, the corresponding rational map
$\phi:A\ratto \P^{n-1}$ will be defined on a random point with high
probability (as long as we work over a sufficiently large field),
and if $x_1,\dots,x_{d+1}$ are random points, we expect with high
probability that $\phi(x_1),\dots,\phi(x_{d+1})$ will form a projective
coordinate frame.  But then for any $g\in G$, we can solve for the unique
element of $\PGL_d$ taking $\phi(x_1),\dots,\phi(x_{d+1})$ to
$\phi(gx_1),\dots,\phi(gx_{d+1})$; doing this for generators of $G$ gives
the desired map $G\to \PGL_d$, and thus the desired
representation.

There are two remaining issues: we need to be able to compute $\phi$, and
we need to resolve the residual 1-dimensional character freedom.  Of
course, the latter choice simply corresponds to a choice of equivariant
structure on the bundle, so really the question is which of the equivariant
structures can potentially give polynomial invariants.  But by degree
considerations, there is a relatively small set of possibilities for the
Hilbert series of a polynomial invariant ring, and in practice the
``wrong'' values of $\chi_\Theta$ can be ruled out by observing that none
of the possible Hilbert series is compatible with the given partial Hilbert
function.  (This could also be done by observing that
\cite[Lem.~2.35]{elldaha} allows us to compute the character of any strong
reflection by reducing to a calculation on the root curve.)

We are thus left with computing $\phi$, or equivalently finding sections of
$\Gamma(A;\Theta)$.  It turns out that bootstrapping tends to work well for
this question as well: most reflection groups act trivially on the global
sections of the equivariant line bundle!  So if the subgroup $H$ acts
crystallographically on $A$, the ability to compute invariants for the
subgroup immediately gives us the ability to compute sections of
$\Gamma(A;\Theta)$, as they are nothing other than the $H$-invariants of
degree 1!  If the crystallographic lattice $A'$ is only isogenous to $A$,
then $\Gamma(A;\Theta)$ is the space of invariants in $\Gamma(A';\Theta)$
for the corresponding translation subgroup.  But the translation subgroup
is \'etale, so we can use random points to determine how it acts on
$\Gamma(A';\Theta)$.  (We could also use the kernel of the polarization in
place of the random points, as the image of the kernel of the polarization
will always suffice to rigidify the projective space; this would work in
characteristic 0, but will not be needed below.)

As we already noted for $ST_{29}$, the dimension in finite characteristic
is bounded below by the dimension in characteristic 0, giving us the
(presumably tight) lower bound we need.  Of course, for quaternionic cases,
this apparently fails, for the simple reason that the problem doesn't come
from characteristic 0.  This is only an apparent issue, however: although
the abelian variety does not lift, the group {\em does}.  The point is that
the above representation theory description works for any prime other than
the characteristic.  (That is, the group $\G_m.\ker(\psi)$ always acts
irreducibly, the tensor product decomposition always works, and the group
acting on any prime-to-$p$ tensor factor is \'etale.  The $p$-power factor
is well-behaved when the abelian variety has full $p$-rank, but this never
happens in the quaternionic cases.)  In particular, the representation of
$\G_m.l^{2n}.\Sp_{2n}(\F_l)$ coming from the monodromy of the moduli stack
of abelian varieties with full level $l$ structure {\em is} the reduction
to characteristic $p$ of the corresponding complex representation.  The
1-dimensional character freedom remains, but tends to be helpful: when the
$1$-dimensional character is trivial mod $p$, twisting by the character
gives a different lift to characteristic 0, and {\em any} lift gives us a
lower bound.

Of course, we only need to do this calculation for the few remaining cases
in which we have not yet proved the main theorems.  The needed results for
those groups are summarized in the following table (noting that in the
quaternionic cases, we only obtain a lower bound).  Note that for $\pm
\SU_5(\F_2)$, the restriction from $\Sp_{10}(\F_{23})$ has no invariants in
characteristic 2, but the twist by the sign character has 14 invariants as
stated.
\begin{table}[H]
  \begin{center}
\begin{tabular}{c|c|c}
  $G$& $\deg$ & $\dim$ \\
\hline
  $ST_{33}$ & 7 & 32\\
  $ST_{34}$ & 13 & 27\\
  $2_-^{1+6}.3^3\rtimes S_4$ & 17 & $\ge 15$\\
  $2_-^{1+6}.\Omega^-_6(\F_2)$ & 41 & $\ge 14$\\
  $\pm \SU_5(\F_2)$ & 23 & $\ge 14$
\end{tabular}
  \end{center}
\end{table}

There is one group ($2_-^{1+6}.3^{1+2}.Z_2\subset \GL_4(Q_2)$) not in the
above list, as its crystallographic lattice has minimal polarization of
degree 2, making it difficult to control the action on the 2-torsion
factor.  However, we can resolve this by noting that there is a larger
group (with no additional reflections) acting on the same abelian variety,
namely $2_-^{1+6}.\GU_3(\F_2)$.  (This is the normalizer inside
$2_-^{1+6}.\Omega^-_6(\F_2)$, noting that $\Omega^-_6(\F_2)$ acts on the
kernel of Frobenius as $\SU_4(\F_2)$.)  The group $2_-^{1+6}.\GU_3(\F_2)$
has no 2-dimensional irreps in characteristic 2, and thus {\em its} action
on the 2-dimensional tensor factor must be an extension of 1-dimensional
irreps.  In particular, the restriction to $2_-^{1+6}.\SU_3(\F_2)$ must be
a self-extension of the trivial representation, and thus factors through an
elementary abelian 2-group quotient.  Since our reflection group is the
derived subgroup of $2_-^{1+6}.\SU_3(\F_2)$, we conclude that it must act
trivially on the 2-dimensional tensor factor!  This lets us apply the above
discussion to see, for instance, that it has $\ge 16$ invariants of degree
7.

\subsection{$2_-^{1+6}.\Omega^-_6(\F_2)$ and $2_-^{1+6}.G_4(Z_3,1)$}
\label{ss:omega6}

Since the method as we have described it so far relies on random points, we
see that it is best adapted to finite characteristic.  (In characteristic
0, there is no uniform distribution for generating ``random'' points, and
elliptic curve arithmetic tends to greatly increase coefficient sizes.)  As
a result, before dealing with the missing complex cases, we will first deal
with some quaternionic cases.

The quaternionic group $2_-^{1+6}.\Omega^-_6(\F_2)$ and its subgroup
$2_-^{1+6}.G_4(Z_3,1)$ both act on a principally polarized abelian 4-fold of
characteristic 2, and in fact we have a chain
\[
C_4\subset F_4\subset 2_-^{1+6}.G_4(Z_3,1)\subset 2_-^{1+6}.\Omega^-_6(\F_2)
\]
of subgroups with all but $C_4$ acting crystallographically.  The variety
$A'$ with polarization of degree 2 on which $C_4$ acts is originally
constructed as the quotient of $E^4$ by the subgroup of $\alpha_4^4$ with
sum in $\alpha_2$, but we can quotient by $\alpha_2^4$ to instead write it
as the quotient of $E^4$ by the sum zero subgroup of $\alpha_2^4$.  In that
form, we see that $A$ is the quotient by the residual diagonal action of
$\alpha_4$.  The resulting descripton of $\Gamma(A';\Theta')$ is simple
enough that we can explicitly compute the action of $\alpha_4$ and in
particular find the unique section of $\Gamma(A;\Theta)$.  (The reflections
all have order 2, and thus their actions on the $1$-dimensional space of
sections are necessarily trivial.)  In particular, since we have an
invariant of degree 1, we can perform bootstrapping.

We may thus proceed as follows: (a) Using random points, find the
$C_4$-invariants of degree 4 such that $f/i_1^4$ is invariant under some
chosen generator of $F_4$ over $C_4$.  (In the extremely unlikely event
that the random point algorithm leaves more than 8 survivors, redo the
calculation from scratch.)  Use these to find the generating invariants of
$F_4$.  (b) Using random points, find the polynomials of degree 17 in the
$F_4$ invariants which (after dividing by $i_1^{17}$) are invariant under
some chosen generator of $2_-^{1+6}.G_4(Z_3,1)$ over $F_4$, and check that
we get precisely 15 survivors, giving rise to putative generating
invariants of degrees 1,4,6,9,16.  Since the product of degrees is correct
and the putative invariants have no common intersection on $A/F_4$, the
invariant ring is polynomial as required.  (c) Similarly, we may bootstrap
from this group to get 14 degree 41 invariants for
$2_-^{1+6}.\Omega^-_6(\F_2)$ and conclude that we have generators of
degrees 1,9,16,24,40.

\subsection{$2_-^{1+6}.3^{1+2}.Z_2$}
\label{ss:badomega6}

For the remaining primitive reflection subgroup of
$2_-^{1+6}.\Omega^-_6(\F_2)$, the crystallographic lattice is different, and
has polarization of degree 2.  Here, the main technical issue is computing
the invariant sections of degree 1; the maximal reflection subgroups are
all $D_4$, but the kernel of the isogeny is particularly complicated to
express in terms of the (non-crystallographic) lattice we usually use
when computing $D_4$-invariants.  There is an alternate lattice we may use,
however: there is an isogeny $E^4\to \Lambda_{D_4}\otimes E$
such that the pullback is generated by a single reflection over
$[\mu_2]^4$, with matrix
\[
\frac{1}{2}
\begin{pmatrix}
  1&\hphantom{-}1&\hphantom{-}1&\hphantom{-}1\\
  1&\hphantom{-}1&-1&-1\\
  1&-1&\hphantom{-}1&-1\\
  1&-1&-1&\hphantom{-}1
\end{pmatrix}.
\]
The quotient by $[\mu_2]^4$ is $(\P^1)^4$ and bootstrapping to get
invariants of $D_4$ requires only that we identify some invariant of degree
1.  The kernel is the diagonal action of $\alpha_4$, which acts on $\P^1$
as $(x,w)\mapsto ((1+\epsilon^3)x+\epsilon w,\epsilon^2 x+w)$, and thus
it is straightforward to find the invariants of degree 1 for $D_4$ and thus
to bootstrap to find the additional invariant of degree 2 in this form.

As we noted above when discussing dimensions, $G$ in this case has
nontrivial normalizer, and it turns out that there are elements of the
normalizer that act diagonally on $E^4$ but do not preserve the chosen
$D_4$.  This allows us to easily compute the invariants of a {\em
  different} copy of $D_4$, and thus by taking intersections to find the
two invariants of degree 1 for $G$.  It is then straightforward to
bootstrap to get the 16 invariants of degree 7.  We find that through
degree 7, the ring is generated by elements of degrees $1,1,4,6,6$.  If
these were algebraically independent, they would generate the invariant
ring, but they have nontrivial intersection on $A/D_4$, so cannot actually
generate!  We thus conclude that the invariant ring is not polynomial in
this case.

\subsection{$ST_{33}$}
\label{ss:ST33}

Since the final quaternionic case most naturally boostraps from $ST_{33}$,
we turn to that group next.  The lattice in this case has polarization of
degree 2.  The only irreducible reflection subgroup of $ST_{33}$ is $A_5$,
which is related by an isogeny of degree 3, with the two invariants under
the kernel of the isogeny both being invariant under $ST_{33}$.  For any
given finite characteristic, it is straightforward to bootstrap from $A_5$
(by computing the 32-dimensional space of degree 7 invariants), but we
require an additional idea to deal with {\em all} finite characteristics.

To resolve this, we need to compute generating invariants in characteristic
0.  The problem, of course, is that we no longer have a reasonable notion
of random points to use for bootstrapping.  On the other hand, any point
gives a relation, so we can still hope that any reasonably large set of
points will work.  Since we need to perform arithmetic on the points, we
will encounter coefficient explosion unless the points have height 0, i.e.,
are torsion.  This suggests working with $A[N]$ for some $N$, which for
$\gcd(N,3)=1$ is naturally identified with $E[N]^5$.  In that case,
although the representation of $G$ as a group of rational endomorphisms of
$E^5$ has denominators, those denominators are prime to $N$ and thus $G$
acts naturally on $E[N]^5$.  A big caveat is that the equations
$\phi(x)=\phi(gx)$ are far from independent, and thus one should instead
structure the equations as saying that $\phi$ is constant on $G$-orbits.
Since it is already constant on $H$-orbits, a $G$-orbit which splits into
$m$ $H$-orbits will give $m-1$ equations.  (Of course, there is also an
action of Galois in general, and we should only work with one $G$-orbit
from each Galois cycle, giving $(m-1)d$ equations where $d$ is the degree
of the relevant field extension.)  We thus see that $2$-torsion can give at
most $15$ equations (out of a needed $120$), suggesting that we use the
$4$-torsion.

Since this is a geometric question, we have some choice in the model of the
curve we take, and in particular may take the curve to be $y^2=x^3-27$,
which has full 4-torsion over the extension field
$\Q(\zeta_{12},(-3)^{1/4})$, with the action of $G$ defined over that same
field.  Note that the images of the $4$-torsion points in $A/S_5\times
[\mu_2]$ are actually defined over $\Q(\zeta_{12})$, since $[-1]$ acts on
the $4$-torsion in the same way as the generator of the Galois group.
Since we know how $E[\sqrt{-3}]\ltimes [\mu_2]$ acts on $E^4/S_5$, it is
easy to find the $S_5\times [\mu_2]$-invariant sections of $7\Theta$.  It
is also straightforward to compute orbit representatives of $G\rtimes \Gal$
on the $4$-torsion (one can first compute them on $2$-torsion and then
lift), and to find the $H$-orbits in each resulting $G$-orbit.  This gives
significantly more equations than expected, so one can simply use the $5$
largest $G$-orbits to get enough equations.

Given the space of invariants of degree 7, we can immediately find
candidate generating invariants with coefficients in $\Z$ (after clearing
denominators).  Verifying that these generate is a Gr\"obner basis
calculation (showing that they have trivial intersection in $\P^5$) and
keeping track of the divisions occuring in the process gives us a
collection of bad primes, where we can check things directly.  This deals
with everything except for the two primes (2 and 3) where the model of
$E_0$ we used has bad reduction, but as already mentioned it is easy to
check those cases by bootstrapping using random points.  (The full
calculation for $ST_{33}$, including the two small bad primes, takes less
than 7 seconds.)

\subsection{$\pm \SU_5(\F_2)$}
\label{ss:SU5}

For $\pm \SU_5(\F_2)$, we bootstrap from $ST_{33}$ (i.e., bootstrap along the
chain $A_5\subset ST_{33}\subset \pm \SU_5(\F_2)$) to find the invariants of
degree 23.  We then find that the invariant ring is generated through
degree $23$ by invariants of degrees $1,6,9,12,16,22$.  Again, degree
considerations tell us that the invariant ring is polynomial iff these
invariants have trivial intersection, and thus we see that this is a (final)
counterexample to the main theorems.

\subsection{$ST_{34}$}
\label{ss:ST34}

Finally, for $ST_{34}$, we proceed as in $ST_{33}$, now bootstrapping from
$G_6([\mu_3],1)$ to compute the 27 invariants of degree 13.  Here we first
verify that the main theorem holds in characteristic 5 (and thus in
characteristic 0 since the group we are bootstrapping from has index prime
to 5, so the relative Reynolds operator argument tells us that module of
invariants is flat for reduction mod 5), which suffices to tell us the full
Hilbert series in characteristic 0.  (The generators have degrees
$1,3,4,6,7,9,12$.)  In particular, once we have done this, we can work in
degree 12 rather than 13, which significantly reduces the number of
candidate functions in later steps.  It is also easy to verify the result
in characteristics 2 and 3; indeed, the verification in characteristics
$2,3,5$, and thus $0$ takes less than 2 seconds!

For the characteristic 0 calculation, we again use $4$-torsion (using the
same model of $E_0$ with bad reduction over $2$ and $3$) to find the
generating invariants and then keep track of divisions in the Gr\"obner
basis calculation that verifies that they have trivial intersection.  Here
we encounter a complication: the denominators that appear are composites,
and it would actually take a significant amount of time to factor them!
However, we can largely resolve this by computing Gr\"obner bases with
respect to two different orderings on the $G_6([\mu_3],1)$-invariants and
taking the pairwise $\gcd$s of the two sequences of composites.  This
eliminates enough large prime factors to let us factor the remaining
numbers in a reasonable amount of time, at which point we can finish off as
usual.  The total calculation takes around 23 seconds, about half of which
is consumed by evaluating the candidate invariants at $4$-torsion points.

\section{Calabi-Yau quotients}
\label{sec:derived}

By Proposition \ref{prop:CM_and_CY} above, for any subgroup sandwiched
between a crystallographic real or complex reflection group and its derived
subgroup, the invariant ring is Cohen-Macaulay in characteristic 0, and if
the determinant character vanishes, the quotient is Calabi-Yau.  This
suggests that we should in particular look at the invariant ring of the
derived subgroup.  It turns out that the structure is particularly nice if
we look at the {\em crystallographic} derived subgroup.  That is, rather
than take the derived subgroup of $G$ itself, we consider the derived
subgroup of the extension $\Lambda\rtimes G$ that acts on $\C^n$.  In more
geometric terms, this is equivalent to the action of the usual derived
subgroup $G'$ on the natural covering abelian variety $A^{++}$.  Similarly,
the natural choice of equivariant line bundle on $A^{++}$ comes from the
derived subgroup of $\pi_1(T)\rtimes G$ where $T$ is the original
$\G_m$-torsor, and is nothing other than the natural $G'$-equivariant
structure on the ample generator ${\cal L}^{++}$ of the group of line
bundles pulled back from $A$.  (Since the $G$-equivariant structure on that
bundle form a torsor under the character group of $G$, there is a {\em
  canonical} $G'$-equivariant structure!)

If $D$ is a $G$-invariant effective divisor on $A$, then the pullback of
$D$ to $A^{++}$ is $G$-invariant and thus represents a power of ${\cal
  L}^{++}$.  Since the corresponding section is uniquely determined up to
scalar multiples, its failure to be $G$-invariant is measured by a
character, so that it {\em will} be $G'$-invariant.  In particular, we may
obtain $G'$-invariants from {\em any} divisor on $A$ by first summing over
the $G$-orbit and then taking the corresponding invariant on $A^{++}$.  The
most important cases come from the orbits of integral reflection
hypersurfaces (i.e., connected components of the fixed subschemes of
reflections).

Consider in particular the $G$-invariant divisors $D_i$ associated to
integral strong reflection hypersurfaces (i.e., for which the reflection
has eigenvalue 1 on the fiber over the generic point).  There is one such
divisor for each orbit of such hypersurfaces, and thus we obtain a
corresponding collection of elements $\delta_i$ of the ring of
$G'$-invariants in powers of ${\cal L}^{++}$.  Note that each such element
is $H_0(G;\Lambda)\rtimes G$-invariant up to scalars, so determines a
character $\chi_i$ of $G^{++}$ (the relevant reflection subgroup of
$\G_m.H_0(G;\Lambda)\rtimes G$).

\begin{prop}
  Let $r$ be a reflection in $G^{++}$ and let $\delta$ be one of the above
  natural invariants, with associated divisor $D$ on $A$.  Then
  $r\delta=\delta$ unless $r$ fixes a component of $D$ with eigenvalue 1,
  in which case $r\delta = \det(r)\delta$.
\end{prop}

\begin{proof}
  By construction, $\delta$ is the pullback from $A$ of the canonical
  section $\bar\delta\in {\cal L}(D)$.  The image $\bar{r}\in G$ acts
  trivially on $\bar\delta$ relative to the natural $G$-equivariant
  structure on ${\cal L}(D)$, so this reduces to understanding the
  difference between the induced $G^{++}$-equivariant structure and the one
  coming from the relevant power of ${\cal L}^{++}$.

  Let $x$ be an associated point of the strong reflection hypersurface for
  $r$.  Since $r$ acts trivially on the fiber of ${\cal L}^{++}$ over $x$,
  it remains to consider how it acts on that fiber of the pullback of
  ${\cal L}(D)$, or equivalently how the image of $\bar{r}\in G$ acts on
  the corresponding fiber of ${\cal L}(D)$ itself.  But this reduces to a
  local calculation along the image of $x$.  If $D$ does not contain the
  image of $x$, then $\bar{r}$ acts trivially and thus $r\delta=\delta$ as
  required.  Otherwise, we can recover the action of $\bar{r}$ from the
  action on the normal bundle; since $\det(r)$ is the only nontrivial
  eigenvalue of $\bar{r}$ on the tangent bundle, the claim follows.
\end{proof}

\begin{cor}
  The character group of $G^{++}$ is naturally isomorphic
  to $\prod_i \langle \chi_i\rangle$.
\end{cor}

\begin{proof}
  Choose a representative of each orbit of strong reflection hypersurfaces,
  and for each representative, choose a generator $r_i$ of the stabilizer
  of its generic point.  Then the conjugacy classes of the $r_i$ clearly
  generate $G^{++}$, so that the natural map from the character group of
  $G^{++}$ to the character group of $\prod_i \langle r_i\rangle$ is
  injective.  But the Proposition tells us that the latter character group
  is $\prod_i \langle \chi_i\rangle$ and thus the map is a retract, so an
  isomorphism (being a retract of groups).
\end{proof}

\begin{cor}
  Let $\chi=\prod_i \chi_i^{m_i}$ (with $0\le m_i<\ord(\chi_i)$) be a
  character of $G^{++}$, and let $f$ be an integral
  eigen-invariant with character $\chi$.  Then $f/\prod_i \delta_i^{m_i}$
  is an integral $G^{++}$-invariant.
\end{cor}

\begin{proof}
  It suffices to consider the valuation of $f/\prod_i \delta_i^{m_i}$ along
  a strong reflection hypersurface, say a component of $\delta_i$, with
  associated reflection $r_i$.  For $0\le j<m_i$, suppose we have already
  shown that $f/\delta_i^j$ is integral.  This is an eigen-invariant with
  nontrivial character for $r_i$, and thus must vanish on the fixed
  subscheme of $r_i$.  Since this holds for the entire conjugacy class of
  $r_i$, we conclude that $\div(f/\delta_i^j)-\div(\delta_i)$ is effective
  and thus that $f/\delta_i^{j+1}$ is again integral.  Thus by induction
  (with trivial base case $j=0$) we conclude that $f/\delta_i^{m_i}$ is
  integral for each $i$, and thus so is $f/\prod_i \delta_i^{m_i}$.
\end{proof}
  
\begin{prop}
  The ring of $G'$-invariants in $\bigoplus_d \Gamma(A^{++};({\cal L}^{++})^d)$
  is the free module over the ring of $G$-invariants in $\bigoplus_d
  \Gamma(A;{\cal L}^d)$ with basis $\prod_i \delta_i^{j_i}$, $0\le
  j_i<\ord(\chi_i)$.
\end{prop}

\begin{proof}
  The ring of $G'$-invariants is naturally graded by $(G^{++}/G')^\vee$, and the
  Corollary tells us that each homogeneous subspace is the free module of
  rank 1 generated by $\prod_i \delta_i^{j_i}$ over the $G^{++}$-invariants.
\end{proof}

In other words, the ring of $G'$-invariants is obtained from the ring of
$G$-invariants (the same as the ring of $G^{++}$-invariants) by adjoining
elements $\delta_i$, each of which is an $\ord(\chi_i)$-th root of a
corresponding $G$-invariant.  In particular, it is a complete intersection.

The fact that this is Calabi-Yau tells us the following.

\begin{prop}\label{prop:canclass}
  Let the ring of $G$-invariants have degrees $d_0,\dots,d_n$.  Then
  \[
  \sum_i d_i
  =
  \sum_i (\ord(\chi_i)-1)\deg(\delta_i).
  \]
\end{prop}

\begin{proof}
  From the standard adjunction formula for the canonical class of a
  complete intersection in a weighted projective space, we find
  \[
  \omega \cong \sO(-\sum_i d_i -\sum_i \deg\delta_i + \sum_i \ord(\chi_i)\deg\delta_i).
  \]
  But the quotient is Calabi-Yau, and thus its canonical class is trivial.
\end{proof}

For any subgroup $G'\subset H\subset G^{++}$, the ring of $H$-invariants has
a similar description, given in terms of the $(G^{++}/G')^\vee$-grading by
taking those homogeneous components corresponding to the annihilator of
$H/G'$.  We may in particular apply this to the subgroup $G_0$ of elements
with trivial determinant.

\begin{prop}
  Suppose that $G$ is generated by reflections of order 2.  Then
  the ring of $G_0$-invariants is obtained by adjoining the square
  root of an invariant of order $2\sum_i d_i$.
\end{prop}

\begin{prop}
  Suppose that $G$ is generated by reflections of order 3.  Then
  the ring of $G_0$-invariants is obtained by adjoining the cube
  root of an invariant of order $(3/2)\sum_i d_i$.
\end{prop}

Something similar applies in the general case in which $G$ has no
reflections of order $>3$, except that now the ring of $G_0$-invariants is
generated by a square root and a cube root, both of which are products
$\prod_i \delta_i$ (over invariants associated to order 2 and order 3
reflections respectively).  (The corresponding contributions to the degree
of the canonical class can be read off the tables in Section
\ref{sec:classify} above.)

The only {\em primitive} complex reflection group with reflections of order
$>3$ is $ST_8$.  Here there are actually two cases, since there are two
choices of strongly crystallographic line bundle.  For the smaller bundle
(with degrees 2,3,4), the kernel of the determinant character agrees with
the crystallographic derived subgroup and thus we see that the invariant
ring is obtained by adjoining the fourth root of an invariant of degree 12.
For the other bundle (with invariants of degrees $2,4,6$), the structure is
more complicated.  The invariant ring of the crystallographic derived
subgroup is then generated by elements $i_6^{1/2}$, $i_{12}^{1/4}$ and thus
the invariant ring of $(ST_8)_0$ has basis
\[
1,i_6^{1/2}i_{12}^{1/4},i_{12}^{1/2},i_6^{1/2}i_{12}^{3/4},
\]
so is a quadratic extension of a quadratic extension.

For the {\em imprimitive} groups, the only cases left to consider are
$G_n([\mu_4],[\mu_4])$ and $G_n([\mu_6],[\mu_6])$.  Since the respective
subgroups $G_n([\mu_4],[\mu_2])$ and $G_n([\mu_6],[\mu_2])$ are generated
by reflections of order $2$ and have the same determinant 1 subgroups, the
structures in those cases are straightforward.

\medskip

Given that polynomiality of invariants persists in finite characteristic,
it is natural to ask whether we can say anything nice about the rings of
$G'$ or $G_0$-invariants in finite characteristic.  For reductions of
characteristic 0 cases, this indeed turns out to be well-behaved, and we
will see that they remain Calabi-Yau complete intersections in general.

The best result involves index 2 subgroups.

\begin{prop}
  Let $H\subset G^{++}$ be an index 2 subgroup.  Then there is a degree $d'$
  such that in any characteristic the ring of $H$-invariants is a
  hypersurface of degree $2d'$ in a weighted affine space of degrees
  $d_0,\dots,d_n,d'$.
\end{prop}

\begin{proof}
  In characteristic 0, this follows from the structure of the ring of
  $G'$-invariants, and one finds that there is an invariant $\delta'$ (a
  product of the form $\prod_{i\in S}\delta_i$) of degree $d'$ such that
  any $H$-invariant which has nontrivial eigenvalue for $G^{++}$ is a
  multiple of $\delta'$.  This persists in any {\em odd} characteristic by
  the same argument, and thus the claim is immediate away from
  characteristic 2.  In characteristic 2, the difficulty is that $\delta'$
  reduces to an {\em invariant} element, so cannot be used to generate the
  module of $H$-invariants.  However, semicontinuity tells us that there is
  {\em some} additional $H$-invariant $\iota$ of degree $d'$.

  In general, there is an operation $L$ on the ring of $H$-invariants given
  by $Lf:=(1-g)f/\delta'$ where $g\in G^{++}\setminus H$ is any coset
  representative.  In odd characteristic, $(1-g)f$ is an anti-invariant, so
  a multiple of $\delta'$, and this remains true in characteristic 2: for
  each irreducible (but possibly nonreduced) component of $\delta'$, we may
  take $g$ to be the corresponding reflection, so that $(1-g)f$ vanishes on
  the fixed subscheme of $g$ (and thus on the chosen component of
  $\delta'$).  We thus conclude that $L$ is a map from the module of
  $H$-invariants to the ring of $G^{++}$-invariants with kernel the ring of
  $G^{++}$-invariants.  In particular, $L\iota$ is nonzero (since $\iota$
  is not $G^{++}$-invariant) and a scalar by degree considerations, so that
  we may rescale $\iota$ so that $L\iota=1$.  Since $L$ is linear over the
  ring of $G^{++}$-invariants, $L(f - (Lf)\iota) = Lf(1-L\iota)=0$ and thus
  $f-(Lf)\iota$ is $G^{++}$-invariant.

  But then for any $H$-invariant, we have an expression $f = (f-(Lf)\iota)
  + (Lf)\iota$ of $f$ as a linear combination of $1$ and $\iota$ with
  $G^{++}$-invariant coefficients.  It follows that the ring of
  $H$-invariants is the free module generated by $1$ and $\iota$ over the
  ring of $G^{++}$-invariants, and is thus a hypersurface of degree $2d'$
  in the polynomial ring generated over the ring of $G^{++}$-invariants by
  $\iota$.
\end{proof}

A similar argument gives the following.

\begin{prop}
  Suppose that $H\subset G^{++}$ is a normal subgroup of index 3 such that
  in characteristic 0 the ring of $H$-invariants is generated by an element
  $\delta'$ of the form $\prod_{i\in S} \delta_i$.  Then in any
  characteristic, the ring of $H$-invariants is a hypersurface of degree
  $3\deg\delta'$ in a weighted affine space of degrees
  $d_0,\dots,d_n,\deg\delta'$.
\end{prop}

\begin{proof}
  In this case, characteristic 3 is the only possible characteristic where
  things could fail, and semicontinuity again gives an additional invariant
  $\iota$ of degree $\delta'$.  Let $gH$ be the coset of $G^{++}/H$ that
  has eigenvalue $\zeta_3$ on $\delta'$.  Then $L:f\mapsto (1-g)f/\delta'$
  is again linear over $G^{++}$-invariants, has kernel the
  $G^{++}$-invariants, and preserves integrality.  Moreover, $L^3=0$ and
  thus $L^2f$ is always $G^{++}$-invariant.  We may again normalize $\iota$
  so that $L\iota=1$, implying $L^2\iota^2=1+\zeta_3$.  In particular, if
  $f=a+b\iota+c\iota^2$, with $a,b,c$ $G^{++}$-invariant, then
  $L^2f=c(1+\zeta_3)$ and $Lf = b + c(2\iota-\delta')$ so that one can
  recover the coefficients from $f$, $Lf$, and $L^2f$.  This gives rise to
  an expansion $f=Af+(Bf)\iota+(Cf)\iota^2$ for general $f$ in which $A$,
  $B$, $C$ are operators satisfying $LA=LB=LC=0$, showing that the ring of
  $H$-invariants is a free module with basis $(1,\iota,\iota^2)$ over the
  ring of $G^{++}$-invariants as required.
\end{proof}

\begin{rem}
Note that the structure of the
invariant ring is different if some $\delta_i$ appear with exponent $2$;
this can cause the invariant ring to fail to be flat in
characteristic 3.
\end{rem}

\begin{prop}
  Suppose that $G$ has no reflections of order $>3$ and $H\subset G^{++}$
  is of the form $H=H_2\cap H_3$ where $H_2$ has index 2 and $H_3$ has
  index 3 and satisfies the additional hypothesis of the preceding
  Proposition.  Then the ring of $H$-invariants has the same Hilbert series
  in any characteristic, and is always a complete intersection.
\end{prop}

\begin{proof}
  Again, in characteristic prime to $6$, the characteristic 0 structure
  carries over.  We have shown that the ring of $H_2$-invariants remains a
  hypersurface in characteristic 2, and the usual argument shows that any
  $H$-invariant which is an eigen-invariant of $H_2$ is a multiple of the
  appropriate generating eigen-invariant or its square.  But this imples that
  the ring of $H$-invariants is a free module over the ring of
  $H_2$-invariants with basis of the form $(1,\delta'_3,(\delta'_3)^2)$, so
  that it has a single additional generator and relation and thus remains a
  complete intersection.  Similarly, in characteristic 3, the ring of
  $H$-invariants is a free module over the ring of $H_3$-invariants with
  basis $(1,\delta'_2)$ and is itself a complete intersection.
\end{proof}

These Propositions cover every $G_0$ case for which $G$ has no reflection
of order $>3$, and in particular show that those cases have Calabi-Yau
quotient.  We have already seen that the trivial determinant subgroups of
$G_n([\mu_4],[\mu_4])$ and $G_n([\mu_6],[\mu_6])$ reduce to such cases, and
thus it remains only to consider $ST_8$.  This automatically behaves well
in odd characteristics, so it remains only to consider characteristic 2.
The normal subgroup of $ST_8$ generated by reflections of order 2
($G_2([\mu_4],[\mu_2])$) meets the determinant 1 subgroup of $ST_8$ in {\em
  its} determinant 1 subgroup, which certainly has well-behaved invariants
in characteristic 2.  These subgroups do not agree, but the one is normal
in the other with quotient of order 3.  In particular, the ring of
$G_2([\mu_4],[\mu_2])\cap ST'_8$-invariants has a natural mod 3 grading
away from characteristic 3, and thus flatness for it implies flatness for
$ST'_8$-invariants.  Similarly, the ring of $ST'_8$-invariants is an
invariant subring of a Cohen-Macaulay ring in good characteristic (for the
quotient group), so is itself Cohen-Macaulay.  We thus conclude that the
ring of $ST'_8$-invariants remains a free module over the ring of
$ST_8$-invariants in characteristic 2, with generators of the same degrees.
(Recall that the Hilbert series is either
$(1+t^3+t^6+t^9)/(1-t^2)(1-t^3)(1-t^4)$ or
$(1+2t^6+t^{12})/(1-t^2)(1-t^4)(1-t^6)$, depending on the choice of
equivariant line bundle.)  We similarly conclude that the quotient by
$ST'_8$ remains Calabi-Yau in characteristic 2.  (It is unclear how to use
this approach to conclude that it is a complete intersection, but this can
be verified by a direct calculation.  For the minimal line bundle, the
$G_2([\mu_4],[\mu_2])\cap ST'_8$-invariants have generators of degrees
$1,1,4,6$ with a relation of degree $12$.  Rationality of the $\mu_3$
action implies that the invariants of degree 4 and 6 must be
$ST'_8$-invariant and that the two degree 1 invariants can be chosen to
have opposite eigenvalues, from which we conclude that the invariant ring
has a presentation with generators of degrees $2,3,3,4,6$ and relations of
degrees $6$ and $12$.)

For general normal subgroups of $G^{++}$ with abelian quotient, it can be
difficult to control the reduction to bad characteristic (more precisely to
characteristics which are bad for the abelian quotient).  There is,
however, another natural family for which we can show that the reduction is
Calabi-Yau, namely $G'$ itself (acting on $(A^{++},{\cal L}^{++})$).  For
nonsporadic imprimitive groups, this immediately reduces to
$G_n([\mu_d],1)$, and thus reduces to the determinant 1 case.  For
primitive groups, the only case with $G'\ne G_0$ is $ST_5$, but this still
has Calabi-Yau quotient since $ST'_5\cong ST'_4$.  It thus remains only to
consider the sporadic lattice for $G_2([\mu_6],[\mu_2])$, but again we have
$G_2([\mu_6],[\mu_2])'=G_2([\mu_6],1)'$ and $G_2([\mu_6],1\cong W(G_2))$ is
still crystallographic (this being the lattice associated to the
$3$-isogeny $\sqrt{-3}$).

The above arguments fail in cases that do not lift to characteristic 0.  In
the one {\em complex} case that does not lift (the sporadic lattice for
$ST_5$ in characteristic 2), the characteristic is good for the
abelianization, and thus we automatically get the desired free module (and
complete intersection!) structure.  In particular, we find that it has a
presentation with generators of degrees $1,2,2,3,4$ and relations of
degrees $6,6$, so is a Calabi-Yau complete intersection.  Similarly, the
appropriate subgroup of index 3 has quotient a Calabi-Yau hypersurface, with
generators $1,3,4,4$ and relation of degree $12$.  (One caveat: since the
sporadic lattice is obtained from the nonsporadic lattice by an inseparable
isogeny, the characters given by determinants of the action on the Lie
algebra do not agree, and in particular this index 3 subgroup actually
contains the center, in contrast to the non-sporadic case.)

Note that for the fully quaternionic cases, this tends to fail, and indeed
some quaternionic reflection groups (e.g., $2_-^{1+4}.\!\Alt_5$ and
$2_-^{1+6}.\Omega^-_6(\F_2)$) are perfect, so that the derived subgroup
still has polynomial invariants!  Similarly, for $2_-^{1+4}.S_3$, the
derived subgroup has index $2$ in the normal subgroup $ST_8$ generated by
order 4 reflections, so is a non-Calabi-Yau hypersurface.  Thus the fact
that the derived subgroup of $\SL_2(\F_3)\otimes A_2$ {\em does} have
Calabi-Yau invariants (coming from the index 3 subgroup of $ST_5$) appears
to be largely coincidental.

\bibliographystyle{plain}

\end{document}